\documentclass[11pt]{amsart}
\usepackage[colorlinks,citecolor=blue,urlcolor=black,linkcolor=black]{hyperref}
\usepackage[dvipsnames]{xcolor}
\usepackage{amssymb,stmaryrd}
\usepackage[shortlabels]{enumitem}

\newtheorem{thm}{Theorem}[section]
\newtheorem{mainthm}{Theorem}

\newtheorem{theorem}[thm]{Theorem}
\newtheorem{lemma}[thm]{Lemma}

\newtheorem{cor}[thm]{Corollary}
\newtheorem{claim}{Claim}[thm]

\newtheorem{prop}[thm]{Proposition}

\newtheorem{fact}[thm]{Fact}

\theoremstyle{definition}
\newtheorem{defn}[thm]{Definition}

\theoremstyle{remark}
\newtheorem{remark}[thm]{Remark}

\newcommand\s{\subseteq}
\newcommand\sq{\sqsubseteq}
\newcommand\br{\blacktriangleright}

\DeclareMathOperator\ord{Ord}
\DeclareMathOperator\h{ht}

\DeclareMathOperator\U{U}
\DeclareMathOperator\diam{diam}
\DeclareMathOperator\ubd{{\sf unbounded}}
\DeclareMathOperator\onto{{\sf onto}}
\newcommand\half{\frac{1}{2}}
\newcommand\nmeetarrow{\mathrel{\overset{\wedge}{\longarrownot\longrightarrow}}}
\newcommand\sect[1]{\hypersetup{linkcolor=MidnightBlue}\textcolor{MidnightBlue}{Section~\ref{#1}}\hypersetup{linkcolor=black}}

\renewcommand\restriction{\mathbin\upharpoonright}
\renewcommand\mid{\mathrel{|}\allowbreak}
\newcommand*\axiomfont[1]{\textsf{\textup{#1}}}
\newcommand\stp{\axiomfont{STP}}
\newcommand\zfc{\axiomfont{ZFC}}
\newcommand\gch{\axiomfont{GCH}}
\newcommand\pfa{\axiomfont{PFA}}
\newcommand\ch{\axiomfont{CH}}
\newcommand\ns{\textup{NS}}
\newcommand\bd{\textup{bd}}
\newcommand\last[2]{\eth_{#1,#2}}
\newcommand\lex{\mathrm{lex}}

\DeclareMathOperator{\reg}{Reg}

\DeclareMathOperator{\cf}{cf}
\DeclareMathOperator{\dom}{dom}
\DeclareMathOperator{\im}{Im}
\DeclareMathOperator{\otp}{otp}
\DeclareMathOperator{\acc}{acc}
\DeclareMathOperator{\cg}{CG}
\DeclareMathOperator{\nacc}{nacc}
\DeclareMathOperator{\Tr}{Tr}
\DeclareMathOperator{\tr}{tr}
\DeclareMathOperator{\ssup}{ssup}

\setlist[enumerate,1]{label={(\roman*)}}
\newenvironment{why}[1][Proof]{\proof[#1]\mbox{}}{\endproof}

\title[Walks and higher Aronszajn lines]{Walks on uncountable ordinals\\and non-structure theorems\\for higher Aronszajn lines}
\author{Tanmay Inamdar}
\address{Department of Mathematics, Ben-Gurion University of the Negev, P.O.B. 653, Be’er Sheva, 84105 Israel}
\email{tci.math@protonmail.com}

\author{Assaf Rinot}
\address{Department of Mathematics, Bar-Ilan University, Ramat-Gan 5290002, Israel.}
\urladdr{http://www.assafrinot.com}

\date{Preliminary preprint as of October 11, 2024. More results are upcoming. For the latest version, visit \textsf{http://p.assafrinot.com/71}.}


\begin{document}
\maketitle
\begin{abstract}
It is proved that if there is an $\aleph_2$-Aronszajn line,
then there is one that does not contain an $\aleph_2$-Countryman line.
This solves a problem of Moore and stands in a sharp contrast with his Basis Theorem for linear orders of size $\aleph_1$.

The proof combines walks on ordinals,
club guessing, strong colourings of three different types, and a bit of finite combinatorics.
This and further non-structure theorems for Aronszajn lines and trees are established for successors of regulars, successors of singulars, as well as inaccessibles.
\end{abstract}

\section{Introduction}

\subsection{Classifying uncountable structures}

One of the great successes of set theory is the development of consistent axioms asserting that the universe is saturated enough to the extent that objects that may be added by forcing must already exist.
A primary example is the \emph{Proper Forcing Axiom} ($\pfa$) that entails plenty of \emph{rough classification} results for structures of size $\aleph_1$, the first uncountable cardinal.
Recall that by `rough classification' we mean results obtained from the following template: fixing a class of structures $\mathfrak C$, we then consider some natural ordering $\prec$ on $\mathfrak C$ or some natural equivalence relation $\sim$ weaker than equality, and then study global properties of $(\mathfrak C, \prec)$ or $(\mathfrak C, \sim)$ such as the existence of universal and minimal elements, of a small basis, well quasi-orderedness, a small number of inequivalent elements etc. This general programme is the subject of Todor{\v c}evi{\' c}'s ICM survey \cite{MR1648055}, which contains several examples as well as several promising conjectures many of which have been resolved in the intervening years.
An update and a discussion of these solutions can be found in Moore's ICM survey \cite{MR2827783}.

We mention here a few of these consequences of $\pfa$:\footnote{All the relevant definitions are given in the inner sections. See Page~\pageref{indexpage} for an index.}
\begin{enumerate}
\item every two $\aleph_1$-dense sets of reals are isomorphic, and any set of reals of size $\aleph_1$ forms a one element basis for the class of the uncountable separable linear orders \cite{MR317934};
\item every two $\aleph_1$-Aronszajn trees are isomorphic on a club \cite{Sh:114},
and a classification of $\aleph_1$-Aronzajn trees up to another notion of equivalence \cite{MR2802589};
\item there are only five cofinal types of directed sets of cardinality $\aleph_1$ \cite{MR792822};
\item a similar classification of transitive relations on $\aleph_1$ \cite{MR1407459}
and various results for partially ordered sets \cite{MR1442306, MR1617909, MR1703188};
\item a rich rough classification for the class of Lipschitz trees \cite{MR2329763};
\item the class of $\aleph_1$-Aronszajn lines admits a basis of size two, consisting of a Countryman line and its reverse.
Consequently, the class of all uncountable linear orders admits a five element basis \cite{MR2199228};
\item there is a universal $\aleph_1$-Aronszajn line \cite{MR2480566};
\item the class of $\aleph_1$-Aronszajn lines is well-quasi-ordered under the relation of order embeddability \cite{MR2822417},
and every $\aleph_1$-Aronszajn line has a finite one-dimensional Ramsey degree \cite{MR4568265};
\item additional results for the class of wide $\aleph_1$-Aronszajn trees \cite{Sh:1186}.
\end{enumerate}

A few remarks are in order here.
The first is that such classification results are relevant even to those agnostic to forcing axioms since a classification result as above often leads to the isolation of a list of \emph{critical members} of a class, and the consistency result can be interpreted as showing the completeness of this list. The reader is referred to \cite{MR1648055} for further discussion of this approach.

Second, the above results at $\aleph_1$
often mimic the situation at $\aleph_0$, where various positive rough classification results can be obtained for the analogous classes, often as a consequence of Ramsey's Theorem. As a concrete example we mention that $\{(\mathbb N,{<}),(\mathbb N,{>})\}$ forms a basis for the class of all infinite linear orders.
This advances a thesis that, assuming forcing axioms, various classes of objects of size $\aleph_1$ are \emph{tame} in some sense.

Third, the above results should be contrasted with classical as well as more modern results that the cardinality of the continuum, $\mathfrak c$, behaves markedly differently, falling into the `wild' category.
For example, work of Sierpi{\'n}ski \cite{sierpinski1932probleme}, Dushnik-Miller \cite{MR1919}, Ginsburg \cite{MR53994, MR62808, MR70683} shows that considering the class of linear suborders of $\mathbb R$ of size $\mathfrak c$, any basis for it
must be as large as possible, that is, of size $2^{\mathfrak c}$; that it is not well-founded etc. (see \cite{MR661296} for details).
Another example for directed sets can be found in \cite{MR792822}. It follows that $\ch$ (the continuum hypothesis, asserting that $\mathfrak c=\aleph_1$) is incompatible with these positive classification results at $\aleph_1$.

Parallel to the development of forcing and the applications of $\pfa$ in the 1980s, there was also a breakthrough in infinite combinatorics with Todor{\v c}evi{\'c}'s introduction of the method of \emph{walks on ordinals} \cite{TodActa}.
Whereas the main application of this method was to show the ultimate failure of the Ramsey theorem at the level of $\aleph_1$ (a result earlier proved from $\ch$ \cite{MR202613}),
the deeper application was
a gallery of canonical examples of quite a few critical objects that had previously been constructed in the study of various rough classification problems for $\aleph_1$ (this is showcased in \cite{MR2355670}).
As well, this method and the ensuing ideas have led to some anti-classification results for structures of size $\aleph_1$
which were previously known assuming $\ch$,
see \cite{lspace} for two examples
and \cite{paper60} for a more recent one.

Ever since, these two approaches, forcing axioms and walks on ordinals have functioned as a pincer movement for deciding the structure theory of objects of size $\aleph_1$ and for delineating the exact boundary between order and chaos at the level of $\aleph_1$. As a concrete example, see \cite{MR716846} and \cite{lspace} for the interesting split between the $S$-space problem and its dual the $L$-space problem.

These successes at $\aleph_1$ and recent advances in forcing constructions at $\aleph_2$ \cite{MR3307877,MR3201836,MR3696074,MR3796288,MR3826541,MR3913157,MR4290492,MR4586835,benneria2023aronszajntreesmaximality}
have led to a hope that a similar classification theory for structures of size $\aleph_2$ may be consistently possible,
with an ultimate goal of finding a consistent higher analog of $\pfa$ which is similarly able to `tame' $\aleph_2$.
This has been the subject of numerous workshops in the last decade devoted in part or wholly to this programme,
as well as the motivation for the second author's ERC project of which this paper is the final piece.
To continue here the pincer analogy, we challenge these advances in forcing constructions at $\aleph_2$ by pushing from the side of walks on uncountable ordinals.

In order to explain the new ideas, it will be helpful to draw a distinction between ordinal walks conducted along an arbitrary $C$-sequence, and those constructed along a carefully chosen $C$-sequence.
At the level of $\aleph_1$ this distinction is rarely made and almost all constructions proceed along an arbitrary standard $C$-sequence.\footnote{\label{footnote2}Two notable exceptions are \cite{MR2213652,MR2444284}.
Both will play a role later in this paper.}
Above $\aleph_1$ however, the sensitivity to the choice of the $C$-sequence is fundamental
and the task of cooking up the right one has became an integral part of any application of higher walks.
Furthermore, the challenge that goes into this task increases gradually from $\aleph_1$ to successors of regulars, to successors of singulars, to non-Mahlo inaccessibles,
to Mahlos until we arrive at the final collapse somewhere between greatly Mahlos and weakly compacts.\footnote{See Remark~\ref{fact43} below.}

There are countless successful examples of applications of walks at higher cardinals
along a well-crafted $C$-sequence,
but a literature survey we conducted reveals that each time the crafting of the $C$-sequence is rather ad-hoc.
In this paper we shall work towards
identifying some features of $C$-sequences that allow for a smooth \emph{transfer of ideas} from the theory at $\aleph_1$ to higher cardinals.
These \emph{characteristics} of a $C$-sequence take the best possible values for any standard $C$-sequence on $\omega_1$,
and it will be shown that they may as well be secured at higher cardinals in quite a few natural scenarios.
As well,
it will be established that with respect to some other features, at the level of higher cardinals,
$\zfc$ entails better $C$-sequences compared to those at $\aleph_1$.
When put together this changes the balance dramatically, providing a clear advantage to walks over forcing axioms at these levels.
The key is that in the context of a triple $\kappa_0<\kappa_1<\kappa_2$ of infinite cardinals,
it is possible to study objects at $\kappa_2$ by contrasting how they project to $\kappa_0$ and to $\kappa_1$;
the discrepancy between the two projections together with the availability of pigeonhole-type stabilization features
are at the heart of many strong combinatorial constructions.
This $3$-cardinal constellation also plays a role in Shelah's club-guessing theorem
that goes into his famous $\zfc$ bound for the power of the first singular cardinal \cite{Sh:400}. For other examples, see \cite{Sh:922, paper15, MR4529976}.

As made clear by the examples given earlier, the class of linear orders of size $\aleph_1$ have shown themselves to be particularly pliable to classification results.
More generally, the rough classification problem for various classes of linear orders has served as an important proving ground for set-theoretic techniques (see \cite{MR416925, MR661296} for many examples).
For this reason, we take here the higher analog of Clause~(vi) above, Moore's Basis Theorem, as a test question for the new ideas in walks on ordinals which we introduce.
Bearing in mind that forcing axioms for $\aleph_2$ may potentially be compatible with both $\ch$ as well as a large continuum, the development here will not rely on any cardinal arithmetic assumptions.

\subsection{Non-structure theory for higher Aronszajn lines}
In the proceedings \cite{highandlowcontri} to a workshop at the American Institute of Mathematics in 2016 dedicated to the development of forcing axioms at $\aleph_2$ and beyond,
Moore asks whether it is consistent that every $\aleph_2$-Aronszajn line contains an $\aleph_2$-Countryman line.
To avoid trivialities, one needs to also assume that there are $\aleph_2$-Aronszajn lines,
as a theorem of Mitchell \cite{MR0313057} shows that this need not be the case.

An affirmative answer to Moore's question would have given a higher analog of his work \cite{MR2199228} assuming $\pfa$, the key step in the proof of his basis theorem. As well, a negative answer would have given a higher analog of his work \cite{MR2444284} assuming the axiom $\mho$. The latter result is an important inspiration for our work, so we now discuss it.

As mentioned in Footnote~\ref{footnote2}, unlike most applications of walks on ordinals at $\omega_1$
which are conducted against an arbitrary choice of an $\omega$-bounded $C$-sequence over $\omega_1$,
Moore showed that walking along a particular $\omega$-bounded $C$-sequence $\vec C$ over $\omega_1$ (whose existence is asserted by the axiom $\mho$)
ensures that the corresponding walks-on-ordinals $\aleph_1$-Aronszajn tree $T(\rho_0^{\vec C})$ admits an ordering that makes it into an $\aleph_1$-Aronszajn line that contains no $\aleph_1$-Countryman line.
The proof uses the observation that it is possible to force Martin's Axiom without killing $\mho$, together with the following fact.
\begin{fact}[{\cite[p.~79]{Sh:114} and \cite[Proposition~8.7]{MR2329763}}]\label{fact1} Suppose that Martin's Axiom holds. Let $T$ be an $\aleph_1$-Aronszajn tree
for which there is a colouring $c:T\rightarrow2$ such that for every uncountable $X\s T$,
for every $i<2$, there are $x\neq y$ in $X$ such that $c(x\wedge y)=i$.\footnote{Here, $x\wedge y$ denotes the meet of $x$ and $y$ in the tree $T$.}
Then there exists an ordering of $T$ that makes it into an $\aleph_1$-Aronszajn line that contains no $\aleph_1$-Countryman line.
\end{fact}

In trying to generalise the above proof approach to $\aleph_2$, one bumps into three immediate problems.
The first is that $T(\rho_0^{\vec C})$ for an $\omega_1$-bounded $C$-sequence $\vec C$ over $\omega_2$ need not be an $\aleph_2$-tree.
The second is finding a higher analog of $\mho$ that supports some force-and-pull-back argument in the spirit of Fact~\ref{fact1}.
The third is that the proof in \cite{MR2444284} (alike many proofs involving walks at $\omega_1$)
uses the fact that for every two incompatible nodes $s,t$ of the tree $T(\rho_0^{\vec C})$,
their meet $s\wedge t$ is either trivial or admits a maximal element --- a feature that does not carry on to higher cardinals.

The easiest way to overcome the first problem is to assume $\ch$, as it is in fact equivalent
to the assertion that $T(\rho_0^{\vec C})$ is a special $\aleph_2$-Aronszajn tree
for \emph{every} choice of an $\omega_1$-bounded $C$-sequence $\vec C$ over $\omega_2$.
The second problem is also resolvable assuming $\ch$, by building on recent developments in the theory of club-guessing \cite[Corollary~4.18(1)]{paper46} and some more work.
The third problem is indifferent to cardinal arithmetic and the solution we found here is an adaption of the concept of \emph{vanishing levels} from the recent paper \cite{paper58}
and which takes the form of a new characteristic of $C$-sequences.

Taking $\ch$ as a base assumption is also aligned with Specker's classical theorem \cite{MR0039779},
and the more recent theorem of Todor{\v c}evi{\'c} \cite{transitivecolourings} that $\ch$ entails the existence of
$2^{\aleph_1}$ many pairwise far $\aleph_2$-Countryman lines (this answered another question of Moore \cite{highandlowcontri}).
However, in light of the equivalence mentioned in the previous paragraph, we find $\ch$ to be an overkill for our purpose.
After all, we do not need all $T(\rho_0^{\vec C})$'s to be an $\aleph_2$-tree --- we just need one (well, a good one).
By results of Jensen \cite{jen72} and Todor\v{c}evi\'c \cite{TodActa}, the existence of one $\omega_1$-bounded $C$-sequence $\vec C$ over $\omega_2$ for which $T(\rho_0^{\vec C})$ is an $\aleph_2$-tree
is equivalent to the existence of a special $\aleph_2$-Aronszajn line, which, unlike $\ch$, is intrinsically an assertion about the class of $\aleph_2$-Aronszajn lines.
More importantly, unlike $\ch$ that may easily be turned off by forcing,
deep results of Jensen \cite{jen72} and Mitchell \cite{MR0313057} show that the non-existence of a special $\aleph_2$-Aronszajn line has large cardinal strength of a Mahlo cardinal.

Once decided to give up on $\ch$, the proof approach of running a force-and-pull-back argument breaks down.
And looking forward to dealing also with successors of singulars,
it became evident that an alternative approach is needed.
For this,
we devised a higher dimensional variation of the colouring assertion of Fact~\ref{fact1}. For simplicity, we present here the special case of dimension two.\footnote{The general case is given in Definition~\ref{meetarrow}.}
\begin{defn}\label{def12}
For a regular uncountable cardinal $\kappa$, a lexicographically ordered $\kappa$-tree $\mathbf T=(T,{<_T},{<_{\lex}})$,
the partition relation $\mathbf T\nmeetarrow[\kappa]^2_\theta$
asserts the existence of a colouring $c:T\rightarrow\theta$ satisfying the two:
\begin{enumerate}[(1)]
\item for every pairwise disjoint $S\s T\times T$ of size $\kappa$,
for every $\tau<\theta$, there are $(x_0,x_1)\neq (y_0,y_1)$ in $S$ such that:
\begin{itemize}
\item $c(x_0\wedge y_0)=\tau=c(x_1\wedge y_1)$, and
\item for every $i<2$, $x_i<_{\lex}y_{i}$ iff $\h(x_i)<\h(y_i)$.
\end{itemize}
\item for every $T'\in[T]^\kappa$, there exists a pairwise disjoint $S'\s T'\times T'$ of size $\kappa$
such that for every $S\s S'$ of size $\kappa$,
for all $\tau_0,\tau_1<\theta$, there are $(x_0,x_1)\neq (y_0,y_1)$ in $S$ such that:
\begin{itemize}
\item $c(x_0\wedge y_0)=\tau_0$,
\item $c(x_1\wedge y_1)=\tau_1$, and
\item for every $i<2$, $x_i<_{\lex}y_{i}$ iff $\h(x_i)<\h(y_i)$.
\end{itemize}
\end{enumerate}
\end{defn}

It is not hard to verify that if $\mathbf T=(T,{<_T},{<_{\lex}})$ is a lexicographically ordered $\mu^+$-Aronszajn tree
for which $\mathbf T\nmeetarrow[\mu^+]^2_2$ holds, then some unorthodox lexicographic ordering turns $T$ into a $\mu^+$-Aronszajn line that contains no $\mu^+$-Countryman line.
More generally, we have the following consequences.
\begin{fact}\label{ffact13} Suppose $\kappa$ is a regular uncountable cardinal,
$\theta$ is an infinite cardinal, and $n\ge 2$.
Each of the following clauses implies the next:
\begin{enumerate}[(1)]
\item There exists a lexicographically ordered (resp.~special) $\kappa$-Aronszajn tree $\mathbf T$ such that
$\mathbf T\nmeetarrow[\kappa]^n_\theta$ holds;
\item There exists a $(\kappa,n)$-entangled sequence of $2^{\theta}$-many (resp.~special) $\kappa$-Aronszajn lines.
If $\kappa$ is a successor cardinal, then none of them contains a $\kappa$-Countryman line;
\item There are $2^{\theta}$-many pairwise far (resp.~special) $\kappa$-Aronszajn lines;
\item No basis for the class of $\kappa$-Aronszajn lines has size less than $2^\theta$.
\end{enumerate}
\end{fact}

Our first main result reads as follows.
\begin{mainthm}\label{thma} Suppose that $\kappa=\mu^+$ for a regular uncountable cardinal $\mu$ that is non-ineffable.
Then all of the following are equivalent:
\begin{itemize}
\item There exists a special $\kappa$-Aronszajn tree;
\item There exists a $\mu$-bounded $C$-sequence $\vec C$ over $\kappa$ such that $T(\rho_0^{\vec C})$ is a special $\kappa$-Aronszajn tree
satisfying that $(T(\rho_0^{\vec C}),{\s},{<_{\lex}})\nmeetarrow[\kappa]^n_{\kappa}$ holds for every positive integer $n$;
\item Any basis for the class of special $\kappa$-Aronszajn lines has size $2^\kappa$.
\end{itemize}
\end{mainthm}

Since $\mu=\aleph_1$ is non-ineffable, the preceding together with Fact~\ref{ffact13}
yield a negative solution to Moore's question.
In fact considering higher analogues of Moore's question, we prove a separate theorem covering the missing case of ineffable cardinals as well,
together yielding our second main result.

\begin{mainthm}\label{corb} For every regular uncountable cardinal $\mu$,
if there is a $\mu^+$-Aronszajn line, then there is one without a $\mu^+$-Countryman subline.
\end{mainthm}

In particular, we see here a sharp distinction between the independence results \cite{MR2199228,MR2444284} for $\mu=\aleph_0$ and the $\zfc$ result of Theorem~B for all regular $\mu>\aleph_0$.
This is another validation of the asymptotic approach to infinite combinatorics (see the discussion in \cite{paper63}).

\medskip

As alluded to in Footnote~\ref{footnote2},
in \cite{MR2213652}, Todor\v{c}evi\'c showed it is consistent that for some $\omega$-bounded $C$-sequence $\vec C$ over $\omega_1$,
the $\aleph_1$-tree $T(\rho_1^{\vec C})$ is non-special.
In \cite[Question~2.2.18]{MR2355670}
he asked for a condition to put on $\vec C$ to ensure that $T(\rho_1^{\vec C})$ be special.
Here, we present such a condition that applies not only to $\omega_1$ but to all successors of regulars.
We also address the problem of when $T(\rho_2^{\vec C})$ is special.
As a corollary we obtain the following:

\begin{mainthm}\label{thmd} For every infinite cardinal $\mu$, the following are equivalent:
\begin{itemize}
\item There exists a special $\mu^+$-Aronszajn tree;
\item There exists a $\mu$-bounded $C$-sequence $\vec C$ over $\mu^+$ for which $T(\rho_0^{\vec C})$, $T(\rho_1^{\vec C})$, and $T(\rho_2^{\vec C})$
are all special $\mu^+$-Aronszajn trees.
\end{itemize}
\end{mainthm}
This provides another example that walks on uncountable ordinals can be used to obtain canonical witnesses at cardinals beyond $\aleph_1$.

\subsection{Upcoming results}
As hinted at earlier on, our main goal is to enrich the theory of walks on uncountable ordinals in some generality.
Consequently, the results we prove are not restricted to successors of regular cardinals.
In the next version of this paper, by developing more machinery to construct $C$-sequences,
we shall demonstrate how similar corollaries can be also derived for a broader class of cardinals such as successors of singulars,
inaccessibles, and even Mahlo cardinals.
For instance, the three bullets of Theorem~\ref{thma} hold true for any successor $\kappa=\mu^+$ of a cardinal $\mu$ such that $\square_\mu$ holds.
The following is another sample result.
\begin{mainthm}\label{thmc} In G\"odel's constructible universe,
for every regular uncountable cardinal $\kappa$ that is non-subtle, if there exists a $\kappa$-Aronszajn tree,
then there exists a $C$-sequence $\vec C$ over $\kappa$ such that $T(\rho_0^{\vec C})$ is a $\kappa$-Aronszajn tree
for which $(T(\rho_0^{\vec C}),{\s},{<_{\lex}})\nmeetarrow[\kappa]^n_{\kappa}$ holds for every positive integer $n$.
Furthermore, if $\kappa$ is non-Mahlo, then $T(\rho_0^{\vec C})$ can be secured to be special.
\end{mainthm}

\subsection{Organization of this paper}
In \sect{sec2}, we provide preliminaries on trees and lines. We study two assertions involving colourings of Aronszajn trees,
one of which gives rise to an entangled sequence of Aronszajn lines none of which contains a Countryman line.
The proof of Fact~\ref{ffact13} will be found there.

In \sect{sec3}, we provide preliminaries on strong colourings and
prove the existence of various strong projection maps $f_n:\mathbb Z\rightarrow\mathbb N$ using a bit of finite combinatorics.
We also explain how the main results of this paper shed new light on an old question of Erd\H{o}s and Hajnal \cite{MR280381}.

In \sect{sec4}, we provide preliminaries on $C$-sequences, walks on ordinals, and club-guessing.
We introduce old and new characteristics of $C$-sequences, and of club-guessing.
The utility of them is briefly demonstrated in the section, whereas more substantial applications are deferred to the later sections.

In \sect{sec5}, we study the impact of the characteristics of $C$-sequences on the tree $T(\rho_0)$.
We use it to lift various features that are available for free at the level of $\omega_1$ to higher cardinals,
including a theorem of Mart{\'{\i}}nez-Ranero \cite{MR2995913}. We also establish a connection between $T(\rho_0)$ and $T(\rho_2)$
which this time lifts a theorem of Peng \cite{peng} from $\omega_1$ to higher cardinals.

In \sect{sec6}, we study the tree $T(\rho_1)$,
answering a question of Todor\v{c}evi\'c \cite{MR2355670}
not only at $\omega_1$ but at any successor of a regular cardinal.
The proof of Theorem~\ref{thmd} will be found there.

In \sect{sec7}, we study the impact of the characteristics of $C$-sequences on the tree $T(\rho_2)$.
Results of Peng \cite{peng}, Cox-L\"ucke \cite{MR3620068} and L\"ucke \cite{Lucke} are generalised.

In \sect{easysec}, we provide a sufficient condition for the existence of a special $\mu^+$-Aronszajn line with no $\mu^+$-Countryman subline.
We shall later show that this condition is satisfied for every $\mu$ that is ineffable.
This section introduces some of the ideas which will be elaborated on in the tree-colouring results of the next section.

In \sect{sec8}, we provide sufficient conditions (in the language of Section~\ref{sec4}) for $T(\rho_0)$ to satisfy the strong colouring assertion for trees studied in Section~\ref{sec2}.
This lifts the constructions from Moore's paper \cite{MR2444284} at the level of $\omega_1$ to higher cardinals.

In \sect{canonicalsect}, we study the consistency of the sufficient conditions presented in Section~\ref{easysec} and Section~\ref{sec8},
using much of the $\zfc$ results obtained in our previous papers \cite{paper47,paper53,paper46}.
Extensions of theorems of Jensen \cite{jen72}, Todor\v{c}evi\'c \cite{TodActa}, and Krueger \cite{MR3078820} will be given.

In \sect{sec10}, we derive corollaries arising from all the previous sections.
The proofs of Theorem~\ref{thma} and \ref{corb} will be found there.

\subsection{Notation and conventions}\label{nandc}
Throughout the paper, $\kappa$ denotes a regular uncountable cardinal,
$\mu$ denotes an infinite cardinal, $\lambda,\theta$ denote arbitrary cardinals,
and $n$ denotes a positive integer.
Let $\log(\mu)$ denote the least cardinal $\theta\le\mu$ such that $2^\theta\ge\mu$.
We write $[\kappa]^\theta$ for the collection of all $\theta$-sized subsets of $\kappa$.
Let $E^\kappa_\theta:=\{\alpha < \kappa \mid \cf(\alpha) = \theta\}$,
and define $E^\kappa_{\le \theta}$, $E^\kappa_{<\theta}$, $E^\kappa_{\ge \theta}$, $E^\kappa_{>\theta}$, $E^\kappa_{\neq\theta}$ analogously.
For $A,B$ sets of ordinals, we denote $A\circledast B:=\{(\alpha,\beta)\in A\times B\mid \alpha<\beta\}$

$\reg(\kappa)$ denotes the collection of all infinite regular cardinals below $\kappa$.
$H_\kappa$ denotes the collection of all sets of hereditary cardinality less than $\kappa$.
For any map $c:[\kappa]^2\rightarrow H_\kappa$ and $\alpha<\kappa$,
while $(\alpha,\alpha)\notin[\kappa]^2$, we extend the definition of $c$, and agree to let $c(\alpha,\alpha):=0$.
For a set of ordinals $A$, we write
$\ssup(A) := \sup\{\alpha + 1 \mid \alpha \in A\}$,
$\acc^+(A) := \{\alpha < \ssup(A) \mid \sup(A \cap \alpha) = \alpha > 0\}$,
and $\acc(A) := A \cap \acc^+(A)$, and $\nacc(A) := A \setminus \acc(A).$

For functions $f,g\in{}^{<\kappa}H_\kappa$,
we say that $f$ and $g$ are \emph{incomparable} iff $f\nsubseteq g$ and $g\nsubseteq f$.
Also let
$$\Delta(f,g):=\min\{\dom(f),\dom(g),\delta\mid \delta\in\dom(f)\cap\dom(g)\ \&\ f(\delta)\neq g(\delta)\},$$
so that $f$ and $g$ are incomparable iff $\Delta(f,g)<\min\{\dom(f),\dom(g)\}$.

For a family $S$ of sequences, and a map $f:\bigcup_{\sigma\in S}\im(\sigma)\rightarrow\ord$,
we say that $S$ is \emph{$f$-rapid}
iff for all $\sigma\neq\sigma'$ in $S$, either $\sup(f[\im(\sigma)])<\min(f[\im(\sigma')])$ or $\sup(f[\im(\sigma')])<\min(f[\im(\sigma)])$.
If we omit $f$, then we mean that $f:V\rightarrow\ord$ is the von Neumann rank map, i.e., $f(x)$ is the least ordinal $\alpha$ such that $x\in V_{\alpha+1}$.

For a sequence $\sigma$ of ordinals, we write $\sup(\sigma)$ for $\sup(\im(\sigma))$
and $\min(\sigma)$ for $\min(\im(\sigma)\cup\{0\})$.
Given a function $h:\kappa\rightarrow\omega$, and a sequence $\sigma\in{}^{<\omega}\kappa$,
we let $\sigma_h:=\sup(h\circ\sigma)$.

For a poset $(P,<_P)$ we let ${\le_P}:=\{(p,q)\in P\times P\mid(p<_Pq)\text{ or }(p=q)\}$,
and ${>_P}:=\{(q,p)\in P\times P\mid p<_Pq\}$;
the \emph{reverse} of $(P,<_P)$ is the poset $(P,{>_P})$.

\section{Trees and lines}\label{sec2}

A \emph{tree} is a partially ordered set $\mathbf T=(T,<_T)$ such that, for every $x\in T$,
the cone $x_\downarrow:=\{y\in T\mid y<_T x\}$ is well-ordered by $<_T$; its order type is denoted by $\h(x)$.
For any ordinal $\alpha$, the $\alpha^{\text{th}}$-level of the tree is the collection $T_\alpha:=\{ x\in T\mid \h(x)=\alpha\}$.
For a set of ordinals $A$, we write $T\restriction A$ for $\bigcup_{\alpha\in A}T_\alpha$.
The \emph{height} of the tree is the first ordinal $\alpha$ for which $T_\alpha=\emptyset$.
Given two distinct nodes $x,y$ in a tree $T$, we let $x\wedge y$ denote the highest node $z\in T$ to satisfy $z\le_T x$ and $z\le_T y$.
The tree is \emph{$\varsigma$-splitting} if each of its nodes admit at least $\varsigma$ many immediate successors.

A \emph{$\kappa$-tree} is a tree $\mathbf T$ of height $\kappa$ all of whose levels are of size less than $\kappa$.
A \emph{$\kappa$-Aronszajn tree} is a $\kappa$-tree with no chains of size $\kappa$.
A \emph{$\kappa$-Souslin tree} is a $\kappa$-Aronszajn tree with no antichains of size $\kappa$.

\begin{defn}[L\"ucke, \cite{Lucke}]\label{ascending}
Let $\mathbf T=(T,<_T)$ be a tree of height $\kappa$, and let $\theta$ be an infinite cardinal.
A \emph{$\theta$-ascending path} through $\mathbf T$ is a sequence $\vec f =\langle f_\gamma\mid \gamma<\kappa\rangle$ satisfying the following two:
\begin{enumerate}
\item for every $\gamma<\kappa$, $f_\gamma:\theta\rightarrow T_\gamma$;
\item for all $\gamma<\delta<\kappa$, there are $i,j<\theta$ such that $f_\gamma(i)<_T f_\delta(j)$.
\end{enumerate}
\end{defn}
\begin{defn}[Todor\v{c}evi\'c, {\cite[p.~266]{TodActa}}]\label{special}
A $\kappa$-tree $\mathbf T=(T,<_T)$ is \emph{special} iff there exists a map $f:T\rightarrow T$ satisfying the following:
\begin{itemize}
\item for every non-minimal $x\in T$, $f(x)<_T x$;
\item for every $z\in T$, $f^{-1}\{z\}$ is covered by less than $\kappa$ many antichains.
\end{itemize}
\end{defn}

\begin{remark} If $\mathbf T$ is a special $\kappa$-tree, then $\mathbf T$ is an Aronszajn tree that is not Souslin.
By \cite[Lemma~1.6]{Lucke}, if a special $\kappa$-tree admits a $\theta$-ascending path, then $\kappa=\mu^+$ for some cardinal $\mu$ of cofinality $\le\theta$.
So, at the level of inaccessible cardinals the concepts of Definitions \ref{ascending} and \ref{special} are mutually exclusive.
\end{remark}

\begin{defn}\label{good}
A \emph{good colouring} of a tree $\mathbf T=(T,<_T)$
is a function $c$ from $T\restriction D$ to the ordinals such that:
\begin{itemize}
\item $D$ is a club in the height of $\mathbf T$;
\item $c(x)<\h(x)$ for every $x\in T\restriction D$;
\item for every pair $x<_T y$ of nodes in $T\restriction D$, $c(x)\neq c(y)$.
\end{itemize}
\end{defn}

The next proposition is well-known. A proof of the backward implication may be found at the beginning of the proof of \cite[Theorem~4.4]{MR3078820}.

\begin{lemma}[folklore]\label{fact13} Suppose that $\mathbf T=(T,<_T)$ is a $\kappa$-tree. Then $\mathbf T$ is special iff there exists a good colouring of $\mathbf T$.
\end{lemma}
\begin{proof} $(\implies)$: Suppose that $\mathbf T$ is special, and let $f:T\rightarrow T$ be as in Definition~\ref{special}.
For each $z\in T$, let $\theta_z$ denote the least size of a family of antichains covering $f^{-1}\{z\}$,
so that $\theta_z<\kappa$.
Fix a bijection $\pi:T\times\kappa\leftrightarrow\kappa$,
and let $$D:=\{\delta<\kappa\mid \forall z\in T\restriction\delta\,(\pi[\{z\}\times \theta_z]\s\delta)\}.$$
It is clear that $D$ is closed. It is also unbounded in $\kappa$, since it covers the intersection of the following clubs:
\begin{itemize}
\item $D_0:=\{\delta<\kappa\mid \pi[(T\restriction\delta)\times\delta]=\delta\}$, and
\item $D_1:=\{\delta<\kappa\mid \forall z\in T\restriction\delta\,(\theta_z<\delta)\}$.
\end{itemize}

For each $z\in T$, as $f^{-1}\{z\}$ is covered by $\theta_z$ many antichains, we may fix a map $g_z:f^{-1}\{z\}\rightarrow\theta_z$
such that the preimage of any singleton is an antichain.
Finally, define a map $c:T\restriction D\rightarrow\kappa$ via:
$$c(x):=\pi(f(x),g_{f(x)}(x)).$$
The definition of $D$ makes it is clear that for every $x\in T\restriction D$, $c(x)<\h(x)$.
In addition, if $x<_Ty$ is a pair of nodes in $T\restriction D$, then either $f(x)\neq f(x)$,
or $g_{f(x)}(x)\neq g_{f(z)}(z)$, so that $c(x)\neq c(z)$.

$(\impliedby)$: Suppose that $c:T\restriction D\rightarrow\kappa$ is as in Definition~\ref{good}.
By possibly thinning out, we may assume that $D\s\acc(\kappa)$.
Define a map $f:T\rightarrow T$, as follows:
$$f(x):=\begin{cases}
x\restriction (c(x)+1);&\text{if }\h(x)\in D;\\
x\restriction \sup(D\cap\h(x)),&\text{otherwise}.
\end{cases}$$

It is clear that $f(x)=x$ for every minimal $x\in T$ and that $f(x)<_T x$ for every non-minimal $x\in T$.
Now, given $y\in T$, there are two cases to consider:

$\br$ If $\h(y)$ is a successor ordinal, say it is $\epsilon+1$,
then $f^{-1}\{y\}$ is a subset of $\{ x\in T\restriction D\mid c(x)=\epsilon\}$
which is an antichain.

$\br$ If $\h(y)$ is a limit ordinal, say it is $\epsilon$,
then $f^{-1}\{y\}$ is a subset of $\{ x\in T\mid \h(x)\le \min(D\setminus(\epsilon+1))\}$
which is a set of size less than $\kappa$.

Thus, in both cases $f^{-1}\{y\}$ is the union of less than $\kappa$ many antichains.
\end{proof}

\begin{fact}[Todor\v{c}evi\'c, {\cite[Theorems 13 and 14]{MR793235}}]\label{criticalcof} For a $\mu^+$-tree $\mathbf T=(T,<_T)$, the following are equivalent:
\begin{enumerate}[(1)]
\item $T$ is special;
\item There exist a club $D\s\mu^+$ and a map $f:T\restriction(E^{\mu^+}_{\cf(\mu)}\cap D)\rightarrow T$ such that:
\begin{itemize}
\item for every $x\in\dom(f)$, $f(x)<_Tx$;
\item for every $z\in T$, $f^{-1}\{z\}$ is covered by $\mu$ many antichains.
\end{itemize}
\item $T$ may be covered by $\mu$ many antichains.
\end{enumerate}
\end{fact}

A \emph{streamlined $\kappa$-tree} is a subset $T\s{}^{<\kappa}H_\kappa$
that is downward-closed, i.e, for every $t\in T$, $\{ t\restriction \alpha\mid \alpha<\kappa\}\s T$,
and such that $\mathbf T:=(T,{\s})$ is a $\kappa$-tree.
Note that in this case, $T_\alpha$ is nothing but $\{t\in T\mid \dom(t)=\alpha\}$.
Another feature of streamlined trees is that of being \emph{Hausdorff}, i.e., for every limit ordinal $\alpha$ and all $x,y\in T_\alpha$,
if $x_\downarrow=y_\downarrow$, then $x=y$. In particular, for two nodes $x,y$ of a streamlined tree $T$,
$x\wedge y\in\{x,y\}$ iff $x$ and $y$ are comparable. Indeed,
for two nodes $x,y$ of a streamlined tree $T$, $x\wedge y=x\restriction \Delta(x,y)$.

Out of convenience, we shall mostly be working with streamlined trees $T$,
and we shall identify them with the corresponding structure $\mathbf T:=(T,{\s})$,
asserting that $T$ (rather than $\mathbf T$) is a special $\kappa$-tree, admits a $\theta$-ascending path, etc.

\begin{defn}
For a subset $S\s\kappa$, a streamlined tree $T$ of height $\kappa$ is \emph{$S$-coherent} iff for every $\delta\in S$, for all $s,t\in T_\delta$,
$$\sup\{\gamma<\delta\mid s(\gamma)\neq t(\gamma)\}<\delta.$$
\end{defn}
\begin{lemma}\label{lemma15} Suppose that $T$ is a $E^{\mu^+}_{\cf(\mu)}$-coherent streamlined $\mu^+$-tree.
Suppose that $d:\{b_\delta\mid\delta\in \Delta\}\rightarrow\mu^+$ is a function, where:
\begin{itemize}
\item $\Delta=D\cap E^{\mu^+}_{\cf(\mu)}$ for some club $D$ in $\mu^+$;
\item for every $\delta\in \Delta$, $b_\delta\in T_\delta$ and $d(b_\delta)<\delta$;
\item $d$ is injective on chains.
\end{itemize}

Then $T$ is special.
\end{lemma}
\begin{proof} Before we start, we recall the following piece of notation;
for two nodes $s,t\in T$ with $\dom(s)\le\dom(t)$, we denote by $s*t$ the unique map with $\dom(t)$ satisfying:
$$(s*t)(\beta):=\begin{cases}s(\beta),&\text{if }\beta\in\dom(s);\\
t(\beta),&\text{otherwise.}\end{cases}$$

For every $\alpha<\mu^+$, let $\langle z_\alpha^i\mid i<\mu\rangle$ be some enumeration of $T_\alpha$.
For every $x\in T$, let $i(x):=\min\{i<\mu\mid x=z_{\dom(x)}^i\}$.
Fix a bijection $\pi:\mu^+\times\mu^+\times\mu\leftrightarrow\mu^+$.
Fix a subclub $D\s\{\gamma<\mu^+\mid \pi[\gamma\times\gamma\times\mu]=\gamma\}$ such that $D\cap E^{\mu^+}_{\cf(\mu)}\s \Delta$.
By Fact~\ref{criticalcof}, it suffices to define a map $f:T\restriction(E^{\mu^+}_{\cf(\mu)}\cap D)\rightarrow\mu^+$ such that:
\begin{itemize}
\item for every $x\in\dom(f)$, $f(x)<\dom(x)$;
\item $f$ is injective on chains.
\end{itemize}
Thus, for every $\gamma\in E^{\mu^+}_{\cf(\mu)}\cap D$, for every $x\in T_\gamma$,
find a large enough $\beta<\gamma$ such that
$$\{\alpha<\gamma\mid x(\alpha)\neq b_\gamma(\alpha)\}\s\beta,$$
and then let
$$f(x):=\pi(d(b_\gamma),\beta,i(b_\gamma\restriction\beta)).$$

To see that $f$ is injective over chains,
suppose that $x\subsetneq y$ is a pair of elements of $T\restriction(E^{\mu^+}_{\cf(\mu)}\cap D)$.
Towards a contradiction, suppose that $f(x)=f(y)$. Write $\gamma:=\dom(x)$ and $\delta:=\dom(y)$.
It follows that there exists some $\beta<\gamma$ such that:
\begin{itemize}
\item $b_\gamma=(b_\gamma\restriction\beta)*x$,
\item $b_\delta=(b_\delta\restriction\beta)*y$, and
\item $i(b_\gamma\restriction\beta)=i(b_\delta\restriction\beta)$,
\end{itemize}
so since $x= y\restriction\gamma$, it follows that $b_\gamma=b_\delta\restriction\gamma$. But then $d(b_\gamma)$ must be distinct from $d(b_\delta)$. This is a contradiction.
\end{proof}

We now present Definition~\ref{def12} in full generality.
\begin{defn}\label{meetarrow}
For a structure $\mathbf T=(T,{<},{\lhd})$ such that $(T,{<})$ is a tree and $(T,{\lhd})$ is a partially ordered set,
the partition relation $\mathbf T\nmeetarrow[\kappa]^{n,m}_\theta$
asserts the existence of a colouring $c:T\rightarrow\theta$ satisfying the two:
\begin{enumerate}[(1)]
\item for every $\h$-rapid $S\s {}^nT$ of size $\kappa$,
and every $\tau<\theta$, there are $(x_0,\ldots,x_{n-1})\neq (y_0,\ldots,y_{n-1})$ in $S$ such that the following three hold:
\begin{itemize}
\item for every $k<n$, $c(x_k\wedge y_k)=\tau$;
\item for every $k<n$, if $x_k\lhd y_{k}$ then $\h(x_k)<\h(y_k)$;
\item $|\{ \h(x_k\wedge y_k)\mid k<n\}|\le m$.
\end{itemize}
\item for every $T'\in[T]^\kappa$, there exists an $\h$-rapid $S'\s {}^nT'$ of size $\kappa$
such that for every $S\s S'$ of size $\kappa$,
and every $\tau\in{}^n\theta$, there are $(x_0,\ldots,x_{n-1})\neq (y_0,\ldots,y_{n-1})$ in $S$ such that the following three hold:
\begin{itemize}
\item for every $k<n$, $c(x_k\wedge y_k)=\tau(k)$;
\item for every $k<n$, if $x_k\lhd y_{k}$ then $\h(x_k)<\h(y_k)$;
\item $|\{ \h(x_k\wedge y_k)\mid k<n\}|\le m$.
\end{itemize}
\end{enumerate}

We write $\mathbf T\nmeetarrow[\kappa]^{n}_\theta$ for $\mathbf T\nmeetarrow[\kappa]^{n,n}_\theta$.
\end{defn}

A primary example of an ordering $\lhd$ to be used as the third component is given by the next definition.

\begin{defn}[Kleene-Brouwer lexicographic ordering]\label{lexo}\hfill
\begin{enumerate}
\item For an ordinal $\alpha$ and a linearly ordered set $(L,<_L)$, we define a linear ordering $<_{\lex}$ of ${}^{<\alpha}L$, as follows:

$\br$ If $f$ and $g$ are comparable, then $f<_{\lex}g$ iff $f\supseteq g$;

$\br$ Otherwise, consider $\Delta(f,g):=\min\{i\in\dom(f)\cap\dom(g)\mid f(i)\neq g(i)\}$
and let $f<_{\lex}g$ iff $f(i)<_L g(i)$.
\item For two ordinals $\alpha,\beta$, we let $<_{\lex}$ be the ordering of ${}^{<\alpha}\beta$
obtained from Clause~(i) where $\beta$ is equipped with the well-ordering of $\in$.
\item For three ordinals $\alpha,\beta,\gamma$, we let $<_{\lex}$ be the ordering of ${}^{<\alpha}({}^{<\beta}\gamma)$
obtained from Clause~(i) where ${}^{<\beta}\gamma$ is equipped with the linear ordering ${<_{\lex}}$ of Clause~(ii).
\item If $T\s{}^{<\kappa}H_\kappa$ is a streamlined tree to which the above definitions of $<_{\lex}$ do not apply (in particular, $T\nsubseteq{}^{<\kappa}\kappa$ and $T\nsubseteq{}^{<\kappa}({}^{<\kappa}\kappa)$),
then we let $<_{\lex}$ be the ordering of $T$
obtained from Clause~(i) where $H_\kappa$ is equipped with an arbitrary well-ordering.
\end{enumerate}
\end{defn}
\begin{defn}\label{meetarrow2} For a streamlined tree $T\s{}^{<\kappa}H_\kappa$, write $T\nmeetarrow[\kappa]^{n,m}_\theta$
for $\mathbf T\nmeetarrow[\kappa]^{n,m}_\theta$, where $\mathbf T:=(T,{\s},{<_\lex})$.
Define $T\nmeetarrow[\kappa]^{n}_\theta$ similarly.
\end{defn}

A standard argument shows that if $T$ is a streamlined prolific $\kappa$-Souslin tree,\footnote{The definition of a prolific tree may be found at \cite[Definition~2.5]{paper20}. The existence of a $\kappa$-Souslin tree
is equivalent to the existence of one which is streamlined and prolific.}
then any colouring $c:T\rightarrow\kappa$ that satisfies $c(z)=z(\max(\dom(z)))$ for every nonempty $z\in T\restriction\nacc(\kappa)$
witnesses $T\nmeetarrow[\kappa]^1_\kappa$.
We leave it for the reader to verify that this can be extended to show that if $T$ is $n$-free (see \cite[Definition~1.3]{paper20}), then said colouring moreover witnesses $T\nmeetarrow[\kappa]^{n,1}_\kappa$.
Note however that, in general, the relation $(T,{<},\emptyset)\nmeetarrow[\kappa]^1_2$ implies that
$(T,{<})$ has no chains of order-type $\kappa$, so that examples given by $\kappa$-Souslin trees are destructible by $\kappa$-cc notions of forcing.
In contrast, the examples constructed in Section~\ref{sec8} below will all be $\kappa$-cc indestructible.

\begin{defn}[Abraham-Rubin-Shelah, \cite{ARSh:153}]
Two linear orders $\mathbf L_0, \mathbf L_1$ of size $\kappa$ are \emph{far} (resp.~\emph{monotonically far}) iff
any linear order of size $\kappa$ which embeds into $\mathbf L_0$ does not embed into $\mathbf L_1$ (resp.~into $\mathbf L_1$ nor to its reverse).
\end{defn}
\begin{defn}[Shelah, \cite{Sh:462}]\label{entangledsequence}
A $\lambda$-sequence $\langle (L_i,{\lhd _i})\mid i<\lambda\rangle$ of linear orders is \emph{$(\kappa,n)$-entangled}
iff for every $I\in[\lambda]^n$, for every matrix $\langle x_{\alpha,i}\mid \alpha<\kappa, i\in I\rangle$
such that for every $i\in I$, $\alpha\mapsto x_{\alpha,i}$ is an injection from $\kappa$ to $L_i$,
and for every $c: I \rightarrow 2$,
there are $\alpha<\beta<\kappa$ such that for every $i\in I$, $x_{\alpha,i}\lhd_i x_{\beta,i}$ iff $c(i)=1$.
\end{defn}

\begin{prop}[essentially Sierpi{\'n}ski]\label{entsier} Suppose $\mathcal K$ is a class of $\kappa$-sized linear orders
and $\vec{\mathbf L}=\langle \mathbf L_i\mid i<\lambda\rangle$ is sequence of elements of $\mathcal K$. Then each clause implies the next one:
\begin{enumerate}[(1)]
\item $\vec{\mathbf L}$ is $(\kappa,2)$-entangled;
\item the elements of $\vec{\mathbf L}$ are pairwise monotonically far;
\item the elements of $\vec{\mathbf L}$ are pairwise far;
\item $\mathcal K$ admits no basis of size less than $\lambda$.
\end{enumerate}
\end{prop}
\begin{proof} $(1)\implies(2)$: Let $i\neq j$. Write $\mathbf L_i=\langle L_i,{\lhd_i})$ and $\mathbf L_j=\langle L_j,{\lhd_j})$.
Let $\mathbf K=(\kappa,\prec)$ be any linear order admitting an embedding from $\mathbf K$ to $\mathbf L_i$,
say $f:\kappa\rightarrow L_i$.
Now, let $g:\kappa\rightarrow L_j$ be any injection
and we shall show that $g$ does not constitute an embedding from
$\mathbf K$ to $\mathbf L_j$ nor to its reverse.

For every $\alpha<\kappa$, write $x_{\alpha,i}:=f(\alpha)$ and $x_{\alpha,j}:=g(\alpha)$. Set $I:=\{i,j\}$.
Then $\langle x_{\alpha,l}\mid \alpha<\kappa, l\in I\rangle$ is a matrix
such that for every $l\in I$, $\alpha\mapsto x_{\alpha,l}$ is an injection from $\kappa$ to $L_l$.
Invoking $(\kappa,2)$-entangled-ness with an injective map $c:I\rightarrow2$ yields $\alpha\neq\beta$ in $\kappa$
such that $x_{\alpha,i}\lhd_i x_{\beta,i}$ and $x_{\beta,j}\lhd_j x_{\alpha,j}$
which implies that $g$ does not embed $\mathbf K$ to $\mathbf L_j$.
Invoking $(\kappa,2)$-entangled-ness with a constant map $c:I\rightarrow2$ yields $\alpha\neq\beta$ in $\kappa$
such that $x_{\alpha,i}\lhd_i x_{\beta,i}$ and $x_{\alpha,j}\lhd_j x_{\beta,j}$
which implies that $g$ does not embed $\mathbf K$ to the reverse of $\mathbf L_j$.

$(2)\implies(3)$: This is trivial.

$(3)\implies(4)$: By the pigeonhole principle.
\end{proof}

\begin{defn}
A linear order $(L,{<_L})$ is a \emph{$\kappa$-Aronszajn line} iff all of the following hold:
\begin{itemize}
\item $|L| = \kappa$;
\item for every $X \in [L]^\kappa$, $(X, <_L)$ does not have a dense subset of size less than $\kappa$;
\item neither $(\kappa, \in)$ nor its reverse can be embedded in $(L,{<_L})$ in an order-preserving way.
\end{itemize}
\end{defn}

In \cite[\S3]{MR776625}, one finds a procedure that takes a line $\mathbf L$ and produces a corresponding \emph{partition tree} $\mathbf T$ which is $2$-splitting.
The partition tree of a $\kappa$-Aronszajn line is a $\kappa$-Aronszajn tree.
\begin{defn}
A linear order $\mathbf L$ is a \emph{special $\kappa$-Aronszajn line} iff its partition tree is a special $\kappa$-Aronszajn tree.
\end{defn}

The next fact follows from the content of \cite[\S2]{MR776625}
and is inspired by the proof of \cite[Proposition~8.7]{MR2329763}.
It deals with the dual procedure of transforming a tree into a line.
To stay focused, it is stated for streamlined trees and uses Definition~\ref{lexo} as a building block.

\begin{fact}\label{flippedorder} For a streamlined tree $T$,
a colouring $c:T\rightarrow\theta$ and a subset $A\s\theta$, define an ordering $<^A$ of $T$, by letting for all $s,t\in T$:
\begin{itemize}
\item[$\br$] If $c(s\wedge t)\in A$, then $s<^At$ iff $s<_{\lex}t$;
\item[$\br$] If $c(s\wedge t)\notin A$, then $s<^At$ iff $t<_{\lex}s$.
\end{itemize}

Then:
\begin{enumerate}
\item $<^A$ is a linear ordering of $T$;
\item If $T$ is a (resp.~special) $\kappa$-Aronszajn tree, then $(T,<^A)$ is a (resp.~special) $\kappa$-Aronszajn line.
\end{enumerate}

In particular (using $A:=\theta$), $(T,<_{\lex})$ is a $\kappa$-Aronszajn line.
\end{fact}
\begin{remark} The two translation procedures commute in the sense that the partition tree of $(T,<^A)$ is club-isomorphic to the original tree $(T,{\s})$.
This is best possible since the two trees need not have the same arity.
\end{remark}

\begin{defn} Suppose $\mathbf L=(L,{<_L})$ is a linear order.
\begin{itemize}
\item For a nonzero ordinal $\zeta$, we say that $C\s{}^\zeta L$ is a \emph{chain} iff for all $c,d\in C$,
$$c(0)\le_L d(0)\iff\bigwedge_{i<\zeta}c(i)\le_L d(i);$$
\item $\mathbf L$ is a \emph{$\mu^+$-Countryman line} iff $|L| = \mu^+$ and ${}^2L$ can be decomposed into $\mu$ many chains.
\end{itemize}
\end{defn}
\begin{remark}\label{rmk214} If $\mathbf L$ is a $\mu^+$-Countryman line then it is a $\mu^+$-Aronszajn line,
and for every nonzero ordinal $\zeta$ such that $\mu^{|\zeta|}=\mu$, it is the case that ${}^\zeta L$ can be decomposed into $\mu$ many chains.
Furthermore, a $3$-splitting variation of the partition tree
machinery applied to $\mathbf L$ yields a $\mu^+$-Aronszajn tree whose collection of middle third nodes form a special $\mu^+$-Aronszajn tree.
\end{remark}

The proof of \cite[Proposition~8.7]{MR2329763}
makes it clear that if $T$ is a streamlined $\mu^+$-Aronszajn tree such that $(T,{\s},\emptyset)\nmeetarrow[\mu^+]^1_2$ holds,
then there is no basis for the class of $\mu^+$-Aronszajn lines
consisting of just one $\mu^+$-Countryman line and its reverse. The next two lemmas demonstrate the utility of higher dimensions.

\begin{lemma}\label{omitcm} Suppose that:
\begin{itemize}
\item $T$ is a streamlined $\mu^+$-tree;
\item $c:T\rightarrow\theta$ is a colouring witnessing $T\nmeetarrow[\mu^+]^2_\theta$;
\item $A\in\mathcal P(\theta)$ with $\emptyset\subsetneq A\subsetneq\theta$.
\end{itemize}
Then $(T,<^A)$ does not contain a $\mu^+$-Countryman line.
\end{lemma}
\begin{proof} Suppose not, and fix $T'\in[T]^{\mu^+}$ such that $(T',<^A)$ is a $\mu^+$-Countryman line.
Appealing to Clause~(2) of Definition~\ref{meetarrow} with $T'$
we get a $\h$-rapid $S'\s {}^2T'$ of size $\mu^+$.
As $(T',<^A)$ is a $\mu^+$-Countryman line, $T'\times T'$ admits a decomposition into $\mu$-many chains (with respect to $<^A$),
which means that we may find an $S\s S'$ of size $\kappa$ that is a chain.
As $\emptyset\subsetneq A\subsetneq\theta$, let us fix $\tau:2\rightarrow\theta$ such that $\tau(0)\in A$ and $\tau(1)\notin A$.
Finally, by the choice of $c$, fix $(x_0,x_1)\neq (y_0,y_1)$ in $S$ such that for every $k<2$:
\begin{itemize}
\item $c(x_k\wedge y_k)=\tau(k)$, and
\item if $x_k<_{\lex}y_{k}$ then $\h(x_k)<\h(y_k)$.
\end{itemize}

Since $S$ is $\h$-rapid, we may assume that $\max\{\h(x_0),\allowbreak\h(x_1)\}<\min\{\h(y_0),\allowbreak\h(y_1)\}$.
Therefore, $x_0<_{\lex}y_0$ and $x_1<_{\lex}y_1$.
As $\tau(0)\in A$, it follows that $x_0<^Ay_0$.
As $\tau(1)\notin A$, it follows that $\neg(x_1<^Ay_1)$. This is a contradiction to the fact that $(x_0,x_1)$ and $(y_0,y_1)$ belong to the chain $S$.
\end{proof}

\begin{lemma}\label{getent}
Suppose $T$ is a streamlined $\kappa$-Aronszajn tree and $T\nmeetarrow[\kappa]^n_\theta$ holds with $\theta\ge\omega$ and $n\ge2$.
Then there is a sequence of $\kappa$-Aronszajn lines $\langle \mathbf L_i\mid i<2^\theta\rangle$
that is $(\kappa,n)$-entangled. Furthermore, for every $i<2^\theta$:
\begin{itemize}
\item If $T$ is special, then so is $\mathbf L_i$;
\item If $\kappa$ is a successor cardinal, then $\mathbf L_i$ contains no $\kappa$-Countryman line.
\end{itemize}
\end{lemma}
\begin{proof} Let $c$ witness $T\nmeetarrow[\kappa]^2_\theta$.
As $\theta$ is infinite, a theorem of Hausdorff \cite{hausdorff1936zwei} yields that we may fix an independent family $\vec A=\langle A_i\mid i<2^\theta\rangle$ of subsets of $\theta$.
For every $i<2^\theta$, consider the Aronszajn line $\mathbf L_i:=(T,<^{A_i})$.
By Fact~\ref{flippedorder}, if $T$ is special, then so is $\mathbf L_i$.
As $\vec A$ is an independent family, for every $i<2^\theta$, $\emptyset\subsetneq A_i\subsetneq\theta$,
so that Lemma~\ref{omitcm} applies.

Next, given $I\in[2^\theta]^n$ and $c:I\rightarrow 2$, by the choice of $\vec A$ we may fix a $\tau<\theta$ such that for every $i\in I$, $\tau\in A_i$ iff $c(i)=1$.
Then given a matrix $\langle x_{\alpha,i}\mid \alpha<\kappa, i\in I\rangle$
such that for every $i\in I$, $\alpha\mapsto x_{\alpha,i}$ is an injection from $\kappa$ to $T$,
the fact that $T$ is a $\kappa$-tree implies that the matrix can be thinned out to be $\h$-rapid,
and moreover, the map $\alpha\mapsto\h(x_{\alpha,i})$ is strictly increasing over $\kappa$ for each $i<2$,
so we pick $\alpha<\beta$ in $\kappa$ such that for every $i\in I$,
$c(x_{\alpha,i}\wedge x_{\beta,i})=\tau$ and $x_{\alpha,i}<_{\lex} x_{\beta,i}$.
It follows that for every $i\in I$, $c(i)=1$ iff $\tau\in A_i$ iff $x_{\alpha,i}<^{A_i}x_{\beta,i}$.
\end{proof}

Applications of the strong partition relation $T\nmeetarrow[\kappa]^{n,1}_\theta$ will appear in the next version of this paper.

\section{Colourings}\label{sec3}
\begin{defn}[Associated tree] \label{assoctree}
Given a colouring $c:[\kappa]^2\rightarrow\theta$:
\begin{itemize}
\item For every $\delta<\kappa$, denote by $c_\delta:\delta\rightarrow\theta$
the fiber map defined via $c_\delta(\gamma):=c(\gamma,\delta)$;
\item The associated tree $T(c)$ is the streamlined tree $\{ c_\delta\restriction\gamma\mid \gamma\le\delta<\kappa\}$.
\end{itemize}
\end{defn}

\begin{defn}[{\cite[Definition~2.1]{paper47}}]\label{uprinciple} Let $\mathcal A,J$ be subsets of $\mathcal P(\kappa)$.
\begin{itemize}
\item $\onto(\mathcal A,J,\theta)$ asserts the existence of a colouring $c:[\kappa]^2\rightarrow\theta$ satisfying that
for all $A\in\mathcal A$ and $B\in\mathcal P(\kappa)\setminus J$, there is an $\eta\in A$ such that $c[\{\eta\}\circledast B]=\theta$;
\item $\ubd(\mathcal A,J,\theta)$ asserts the existence of a colouring $c:[\kappa]^2\rightarrow\theta$ satisfying that
for all $A\in\mathcal A$ and $B\in\mathcal P(\kappa)\setminus J$, there is an $\eta\in A$ such that $\otp(c[\{\eta\}\circledast B])=\theta$.
\end{itemize}
\end{defn}
\begin{remark}
Typical values for $J$ are the ideal $J^{\bd}[\kappa]$ of bounded subsets of $\kappa$
and the ideal $\ns_\kappa$ of nonstationary subsets of $\kappa$. If we omit $\mathcal A$, then we mean that $\mathcal A=\{\kappa\}$.
\end{remark}

\begin{defn}[{\cite[Definition~1.2]{paper34}}]\label{uprinciple}
$\U(\kappa,\lambda,\theta,\chi)$ asserts the existence of a colouring $c:[\kappa]^2\rightarrow\theta$ such that
for every $\zeta<\chi$, every pairwise disjoint subfamily $\mathcal A\s[\kappa]^{\zeta}$ of size $\kappa$,
for every $\tau<\theta$, there exists $\mathcal B\s \mathcal A$ of size $\lambda$ such that $\min(c[a\times b])>\tau$ for all $a\neq b$ from $\mathcal B$.
\end{defn}

\begin{prop}[essentially {\cite[Corollary~6.3.3]{MR2355670}}]\label{switchtofibers}
Suppose that $c:[\kappa]^2\rightarrow\theta$ witnesses $\U(\kappa,\kappa,\theta,\chi)$
with $\omega\le\chi\le\cf(\theta)\le\theta<\kappa$.

For every $\zeta<\chi$, for every rapid $S\s{}^\zeta T(c)$ of size $\kappa$,
there exists a rapid $X\s {}^\zeta\kappa$ of size $\kappa$ such that
for all $\langle \alpha_\iota\mid \iota<\zeta\rangle\neq \langle \beta_\iota\mid \iota<\zeta\rangle$ in $X$,
there are $x\neq y$ in $S$ such that for every $\iota<\zeta$:
\begin{itemize}
\item $x(\iota)\s c_{\alpha_\iota}$ and $y(\iota)\s c_{\beta_\iota}$;
\item $x(\iota)$ and $y(\iota)$ are incomparable;
\item $\dom(x(\iota))<\dom(y(\iota))$ iff $\alpha_\iota<\beta_\iota$.
\end{itemize}
\end{prop}
\begin{proof} Let $\zeta$ and $S$ be as above.
As $S$ is rapid, for each $\gamma<\kappa$, we may pick $\mathbf{s}_\gamma\in S$ and $x_\gamma\in{}^\zeta\kappa$ such that:
\begin{itemize}
\item $\gamma\in\bigcap_{s\in \mathbf{s}_\gamma}\dom(s)$;
\item for every $\iota<\zeta$, $\mathbf{s}_\gamma(\iota)\s c_{ x_\gamma(\iota)}$.
\end{itemize}

Fix a stationary $\Gamma\s\kappa$ such that for every pair $\gamma<\delta$ of ordinals from $\Gamma$,
$\sup(x_\gamma)<\delta$. So letting $a_\gamma:= \{\gamma\}\cup\im (x_\gamma)$, it is the case that $\{a_\gamma\mid \gamma\in\Gamma\}$ is a pairwise disjoint subfamily of $[\kappa]^{<\chi}$ of size $\kappa$.
As $\chi\le\cf(\theta)\le\theta<\kappa$ and by possibly shrinking $\Gamma$, we may also assume the existence of some $\tau<\theta$ such that
$\sup\{c_{\beta}(\gamma)\mid \beta\in x_\gamma\}=\tau$ for every $\gamma\in\Gamma$.
As $c$ witnesses $\U(\kappa,\kappa,\theta,\chi)$, we may now find some $\Gamma'\in[\Gamma]^\kappa$
such that for every pair $\gamma<\delta$ of ordinals from $\Gamma'$,
$\min(c[a_\gamma\times a_\delta])>\tau$.
We claim that $X:=\{ x_\gamma\mid \gamma\in \Gamma'\}$ is as sought.
To see this, let $\langle \alpha_\iota\mid \iota<\zeta\rangle\neq \langle \beta_\iota\mid \iota<\zeta\rangle$ be given in $X$.
Fix $\gamma\neq \delta$ in $\Gamma'$ such that $\langle \alpha_\iota\mid \iota<\zeta\rangle=x_\gamma$ and $\langle \beta_\iota\mid \iota<\zeta\rangle=x_\delta$.
Without loss of generality, $\gamma<\delta$.
Then, for every $\iota<\zeta$, $$\mathbf{s}_\gamma(\iota)(\gamma)=c_{\alpha_\iota}(\gamma)\le \tau<c_{\beta_\iota}(\gamma)=\mathbf{s}_\delta(\iota)(\gamma),$$
so that $\Delta(\rho _{1\alpha_\iota}, \rho _{1\beta_\iota}) \leq \gamma < \dom(\mathbf s_\gamma(\iota))<\delta< \dom(\mathbf s_\delta(\iota))$.
\end{proof}

The next result is dual to \cite[Lemma~2.4(1)]{paper36}.
\begin{prop}\label{lemma44} Suppose that a colouring $c:[\kappa]^2\rightarrow\theta$ is a witness to $\U(\kappa,2,\theta,3)$,
with $\theta<\kappa$.
Then, for every $\Upsilon\in[\kappa]^{\kappa}$,
$$\sup\{\delta<\kappa\mid \sup\{ c(\delta,\upsilon)\mid \upsilon\in \Upsilon\setminus\delta\}<\theta\}<\kappa.$$
\end{prop}
\begin{proof} Suppose not, and let $\Upsilon$ be a counterexample. By the pigeonhole principle, pick $\Delta\in[\kappa]^{\kappa}$ and some $\tau<\theta$ such that,
for every $\delta\in \Delta$,
$$\sup\{ c(\delta,\upsilon)\mid \upsilon\in \Upsilon\setminus\delta\}<\tau.$$
For each $\delta\in\Delta$, let $\upsilon_\delta:=\min(\Upsilon\setminus(\delta+1))$.
Fix a sparse enough $\Gamma\in[\Delta]^\kappa$ such that for every $(\gamma,\delta)\in[\Gamma]^2$,
$\upsilon_{\gamma}<\delta$. In particular, $\mathcal A:=\{ \{\delta,\upsilon_\delta\}\mid \delta\in\Gamma\}$ is a pairwise disjoint family of size $\kappa$.
As $c$ witnesses $\U(\kappa,2,\theta,3)$, we may now find $a,b\in\mathcal A$ with $\max(a)<\min(b)$ such that $\min(c[a\times b])\ge\tau$.
By the choice of $\Gamma$, we may find $(\gamma,\delta)\in[\Gamma]^2$ such that $a=\{\gamma,\upsilon_{\gamma}\}$ and $b=\{\delta,\upsilon_\delta\}$.
In particular, $\gamma<\delta<\upsilon_\delta$ and $c(\gamma,\upsilon_\delta)\ge\tau$, contradicting the facts $\gamma\in\Gamma\s\Delta$ and $\upsilon_\delta\in \Upsilon\setminus\gamma$.
\end{proof}

\subsection{A question of Erd\H{o}s and Hajnal}
Recall that $\kappa\nrightarrow[\kappa]^2_\theta$ asserts the existence of a colouring $c:[\kappa]^2\rightarrow\theta$
such that for every $A\in[\kappa]^\kappa$, $c``[A]^2=\theta$.

A subset $R$ of $\kappa$ is a rainbow set (with respect to $c$) if $c$ is injective over $[R]^2$.
A \emph{rainbow triangle} is a rainbow set of size $3$.
Question 68 of \cite{MR280381} asks whether $\gch$ gives rise to a witness to $\aleph_1\nrightarrow[\aleph_1]^2_3$ with no rainbow triangles.
An affirmative answer was given by Shelah in \cite[Theorem~1.1]{MR427073} who got from $2^\mu=\mu^+$ a witness to $\mu^+\nrightarrow[\mu^+]^2_\mu$ with no rainbow triangles.
As for the maximal possible number of colours, Todor\v{c}evi\'c proved \cite[Theorem~9.2]{MR631563} that if there exists a $\kappa$-Souslin tree,
then there exists a witness to $\kappa\nrightarrow[\kappa]^2_\kappa$ with no rainbow triangles,
and Komj\'ath \cite[Theorem~12]{MR4347570} proved that the existence of a witness to $\kappa\nrightarrow[\kappa]^2_\kappa$ with no rainbow triangles
entails the existence of a $\kappa$-Aronszajn tree.
Can Komj\'ath's theorem be improved to infer the existence of a $\kappa$-Souslin tree?
The results of this paper show that this is not case.
The point is that if $(\kappa,<_T)$ is a $\kappa$-tree such that $(\kappa,<_T,\emptyset)\nmeetarrow[\kappa]^1_\theta$ holds with some witnessing colouring $c:\kappa\rightarrow\theta$,
then the induced colouring $d:[\kappa]^2\rightarrow\theta$ defined via $d(\alpha,\beta):=c(\alpha\wedge\beta)$ is a witness to $\kappa\nrightarrow[\kappa]^2_\theta$ with no rainbow triangles.
By Theorem~\ref{thma}, the existence of a special $\aleph_2$-Aronszajn tree is equivalent to one such tree $T$ that is moreover streamlined and satisfies $T\nrightarrow[\omega_2]^1_{\omega_2}$.
By \cite{MR603771}, the existence of a special $\aleph_2$-Aronszajn tree does not imply the existence of an $\aleph_2$-Souslin tree.

\subsection{Projection maps}\label{sec32}
We are grateful to Chen Meiri for pointing out the following lemma.
\begin{lemma}\label{twocolours} Let $n$ be a positive integer.
There is a map $h_n:b\rightarrow2$
from some nonzero $b<\omega$ satisfying that for every $p:n\rightarrow\omega$,
for every $l<2$, there exists an $i<b$ such that $h_n((p(k)+i)~\mathrm{mod}~b)=l$ for every $k<n$.
\end{lemma}
\begin{proof} Let $\langle q_j\mid j<n\rangle$ be an increasing sequence of prime numbers such that $q_0>n$.
Set $b:=\prod_{j=0}^{n-1}q_{j}$,
and then define $h_n:b\rightarrow2$ by letting for every $z<\omega$,
$h_n(z):=0$ iff there exists a $j<n$ such that $z$ is divisible by $q_j$.

Now, given $p:n\rightarrow\omega$, we argue as follows.
\begin{itemize}
\item[$\br$] If $l=0$, then for each $j<n$, pick $i_j<q_j$ such that
$$(p(j)+i_j)\equiv 0\pmod{q_j}.$$
By the Chinese remainder theorem, we now find an $i<b$ such that
$$\bigwedge_{j<n}\left(i\equiv i_j\pmod{q_j}\right).$$
Then, for every $k<n$, we have that $j:=k$ satisfies
$$(p(k)+i)\equiv 0\pmod{q_j},$$
and hence $h_n(p(k)+i)=0$.
So $h_n(p(k)+i)=l$ for every $k<n$.

\item[$\br$] If $l=1$, then since for every $j<n$, the set
$$\{ i<q_j\mid \exists k<n\,[(p(k)+i)\equiv 0\pmod{q_j}]\}$$
has size no more than $n$,
and since $n<q_0\le q_j$, we may find an $i_j<q_j$ outside the above set.
By the Chinese remainder theorem, find an $i<b$ such that
$$\bigwedge_{j<n}\left(i\equiv i_j\pmod{q_j}\right).$$
Then, for every $k<n$, for every $j<n$,
$$(p(k)+i)\not\equiv 0\pmod{q_j},$$
and hence $h_n(p(k)+i)=1$.
So $h_n(p(k)+i)=l$ for every $k<n$.
\end{itemize}
This completes the proof.
\end{proof}

\begin{lemma}\label{intermediate} For every positive integer $n$,
there exists a function $g_n:\omega\rightarrow\omega$
satisfying that for all $l<\omega$,
there exists an $m<\omega$ such that for every $p:n\rightarrow\omega$,
there exists an $i<m$ such that $g_n(p(k)+i)=l$ for every $k<n$.
\end{lemma}
\begin{proof}
Let $n$ be a positive integer. Let $h_n:b\rightarrow2$ be given by Lemma~\ref{twocolours}.
Define $g_n:\omega\rightarrow\omega$, as follows.
Given $z\in\mathbb N$, if there is an $l<\omega$ such that $h_n((z\div b^l)~\mathrm{mod}~b)=1$,\footnote{Here, $z \div d$ denotes the unique natural number $k$ such that $z = d \cdot k+r$ for some $r < d$.}
then let $g_n(z)$ be the least such $l$.
Otherwise, let $g_n(z):=0$.

Now, given any $l<\omega$, we claim that $m:=b^{l+1}$ is as sought.
To see this, let $p:n\rightarrow\omega$ be given.
Construct a sequence $\langle (c_j,i_j)\mid j\le l\rangle$ by recursion as follows.
We let $c_0:=0$, and (recursively) let $c_{j+1}:=i_0\cdot b^0+\cdots+i_j\cdot b^j$ for all $j<l$.
As for $i_j$, its recursive definition is divided into two, as follows:

$\br$ If $j<l$, then find some $i_j<b$ such that for every $k<n$:
$$h_n(((p(k)+c_j)\div b^j)+i_j)~\mathrm{mod}~b)=0.$$

$\br$ If $j=l$, then find some $i_j<b$ such that for every $k<n$:
$$h_n(((p(k)+c_j)\div b^j)+i_j)~\mathrm{mod}~b)=1.$$

Let $i:=\sum_{j\le l}i_j\cdot b^j$. Then $i<m$,
and our recursive construction ensured that for every $k<n$, for every $j<l$,
$$h_n(((p(k)+i)\div b^j)~\mathrm{mod}~b)=0,$$
and that for every $k<n$:
$$h_n(((p(k)+i)\div b^l)~\mathrm{mod}~b)=1.$$
Recalling the definition of $g_n$, this means that $g_n(p(k)+i)=l$ for every $k<n$,
as sought.
\end{proof}

\begin{cor}\label{lemma51}
For every positive integer $n$,
there exists a function $f_n:\mathbb Z\rightarrow\omega$
satisfying that for all $l,d<\omega$,
there exists an $m<\omega$ such that for every $p:n\rightarrow\mathbb Z$ with $\max\{|p(k)-p(k')|\mid k<k'<n\}\le d$,
there exists an $i<m$ such that $f_n(p(k)+i)=l$ for every $k<n$.
\end{cor}
\begin{proof}
Let $n$ be a positive integer. Let $g_n:b\rightarrow2$ be given by Lemma~\ref{intermediate}.
Define $f_n:\mathbb Z\rightarrow\omega$ via $f_n(z):=g_n(|z|)$.
Now, given $l,d<\omega$,
first fix $\bar m<\omega$ such that for every $p:n\rightarrow\omega$,
there exists an $i<\bar m$ such that $g_n(p(k)+i)=l$ for every $k<n$,
and then let $m:=d+2\bar m$.
To see that $m$ is as sought, let $p:n\rightarrow\mathbb Z$ be given.
There are two cases to consider:
\begin{itemize}
\item[$\br$] If $\max(\im(p))\ge -\bar m$, then we may define $p':n\rightarrow\omega$ via
$$p'(k):=p(k)+d+\bar m.$$
Now pick an $i'<\bar m$ such that $g_n(p'(k)+i')=l$ for every $k<n$.
Then $i:=d+\bar m+i'$ is smaller than $m$,
and for every $k<n$, $f_n(p(k)+i)=g_n(p'(k)+i')=l$, as sought.

\item[$\br$] If $\max(\im(p))<-\bar m$, then we may define $p':n\rightarrow\omega$ via
$$p'(k):=-p(k)-\bar m.$$
Now pick an $i'<\bar m$ such that $g_n(p'(k)+i')=l$ for every $k<n$.
Then $i:=\bar m-i'$ is smaller than $m$ (indeed, smaller than $\bar m$) and for every $k<\omega$,
$$p(k)+i=-p'(k)-\bar m+i=-p'(k)-i',$$
so since $p(k)+i<0$, we have that $|p(k)+i|=p'(k)+i'$.
Altogether, $f_n(p(k)+i)=g_n(p'(k)+i')=l$, as sought.
\end{itemize}
This complete the proof.
\end{proof}
\begin{remark}
There is no single map $f:\mathbb Z\rightarrow\omega$
satisfying the conclusion of the preceding for all $n$ simultaneously.
Indeed, otherwise, let $m$ be given with respect to $n=l=d=1$,
and then let $p:m\rightarrow\mathbb Z$ be the identity map.
Evidently, for every $i<\omega$ there exists a $k<m$ such that $f(p(k)+i)=1$.
\end{remark}

\section{$C$-sequences, walks on ordinals and club-guessing}\label{sec4}
A \emph{$C$-sequence} over a set of ordinals $S$ is a sequence $\vec C=\langle C_\delta\mid\delta\in S\rangle$
such that for every $\delta\in S$, $C_\delta$ is a closed subset of $\delta$ with $\sup(C_\delta)=\sup(\delta)$.

\begin{defn}[Todor\v{c}evi\'c]
A $C$-sequence $\vec C=\langle C_\delta\mid\delta<\kappa\rangle$ is \emph{trivial}
iff there exists a club $D\s\kappa$ such that for every $\epsilon<\kappa$, there is a $\delta<\kappa$
with $D\cap\epsilon\s C_\delta$. Otherwise, $\vec C$ is \emph{nontrivial}.
\end{defn}
The preceding generalises as follows.
\begin{defn}[\cite{paper35}] \label{c-seq_num_def}
For a $C$-sequence $\vec C=\langle C_\delta\mid\delta\in S\rangle$ over some stationary $S\s\kappa$,
$\chi(\vec C)$ denotes the least cardinal $\chi\le\kappa$
for which there exists a cofinal $\Gamma\s\kappa$ such that for every $\epsilon<\kappa$,
there is a $\Delta\in[S]^\chi$ with $\Gamma\cap\epsilon\s\bigcup_{\delta\in\Delta}C_\delta$.
\end{defn}

\begin{remark}\label{fact43}
\begin{enumerate}[(1)]
\item It can be verified that $\chi(\vec C)\le\sup(\reg(\kappa))$ and that $\vec C$ is trivial iff $\chi(\vec C)$ is finite.
\item Every successor cardinal $\mu^+$ carries a $C$-sequence $\vec C$ with $\chi(\vec C)\ge\cf(\mu)$,
and by \cite[Theorem~2.14]{paper35} this inequality is best possible.

\item By \cite[Remark~6.3.4]{MR2355670}, if $\kappa$ carries no nontrivial $C$-sequences, then it is quite high in the Mahlo hierarchy;
specifically, by \cite[Lemma~2.12]{paper35}, it is greatly Mahlo.

\item By \cite[1.10]{TodActa}, if $\kappa$ carries no nontrivial $C$-sequences, then it is weakly compact in $\mathsf{L}$.
By \cite[\S3]{paper35}, starting with a weakly compact cardinal $\kappa$, it is possible to force to increase $\chi(\vec C)$ to any prescribed regular value.
In \cite[\S7]{paper69}, this was extended to accommodate arbitary singular values as well.
\end{enumerate}
\end{remark}

We now introduce some new characteristics of $C$-sequences.

\begin{defn}\label{def44} Given a $C$-sequence $\vec C= \langle C_\delta\mid\delta<\kappa\rangle$ we attach to it the following sets:
\begin{itemize}
\item $A(\vec C):=\{\delta\in\acc(\kappa)\mid \forall\alpha\in(\delta,\kappa)\,[\delta\notin\acc(C_\alpha)]\}$;
\item $R(\vec C):=\{\delta\in\acc(\kappa)\mid \forall\alpha\in(\delta,\kappa)\,[\otp(C_\alpha\cap\delta)<\delta]\}$;
\item $V(\vec C):=\{\delta\in\acc(\kappa)\mid \forall\alpha\in(\delta,\kappa)\forall\epsilon<\delta\,[C_\delta\cap[\epsilon,\delta)\neq C_\alpha\cap[\epsilon,\delta)]\}$.
\end{itemize}
We also consider the following stronger variations:
\begin{itemize}
\item $A'(\vec C):=\{\delta\in\acc(\kappa)\mid\forall\alpha\in(\delta,\kappa)\,[\otp(C_\alpha\cap\delta)<\cf(\delta)]\}$;
\item $R'(\vec C):=\{\delta\in R(\vec C)\mid \otp(C_\delta)<\delta\}$;
\item $V'(\vec C):=\{\delta\in\acc(\kappa)\mid \forall\alpha\in(\delta,\kappa)\, [\sup(\nacc(C_\delta)\cap\nacc(C_\alpha))<\delta]\}$.
\end{itemize}
\end{defn}
\begin{remark} $A'(\vec C)\s A(\vec C)\s R(\vec C)$ and $A(\vec C)\s V'(\vec C)\s V(\vec C)$,
and if $A(\vec C)\cup R(\vec C)$ is stationary, then $\vec C$ is nontrivial.
\end{remark}

\begin{remark} Let $\vec C$ be \emph{any} $\omega$-bounded $C$-sequence on $\omega_1$. Then it is easy to see that $A'(\vec C)= \acc(\omega_1)$ and $R'(\vec C) = \acc(\omega_1\setminus\omega)$.
In particular, in this case, the intersection of all of the sets of Definition~\ref{def44} covers a club.
As we will show in Section~\ref{canonicalsect}, for regular cardinals $\kappa >\omega_1$, it is possible from the appropriate hypotheses to construct a \emph{specific} $C$-sequence $\vec C$ on $\kappa$ where some of these characteristics for $\vec C$ cover a club.
\end{remark}

\begin{defn} A $C$-sequence $\vec C= \langle C_\delta \mid \delta <\kappa\rangle$ is said to be:
\begin{itemize}
\item \emph{coherent} iff for every $\delta<\kappa$, for every $\bar\delta\in\acc(C_\delta)$, $C_\delta\cap\bar\delta=C_{\bar\delta}$;
\item \emph{weakly coherent} iff for every $\epsilon<\kappa$, $|\{ C_\delta\cap\epsilon\mid \delta<\kappa\}|<\kappa$;
\item \emph{$\mu$-bounded} iff $\otp(C_\delta)\le\mu$ for all $\delta<\kappa$;
\item \emph{regressive} iff $R'(\vec C)$ covers a club in $\kappa$;
\item \emph{lower regressive} iff $R(\vec C)$ covers a club in $\kappa$.
\end{itemize}
\end{defn}
\begin{remark} To motivate the above terminology, recall that a map $f:\kappa\rightarrow\kappa$ is \emph{regressive} if $f(\alpha)<\alpha$ for all nonzero $\alpha<\kappa$; a map $f:[\kappa]^2\rightarrow\kappa$ is \emph{lower regressive} (resp.~\emph{upper regressive}) if $f(\alpha,\beta)<\alpha$
(resp.~$f(\alpha,\beta)<\beta$) for all nonzero $\alpha<\beta<\kappa$.
\end{remark}

\begin{defn}[Jensen, \cite{jen72}]\label{def410}
\begin{itemize}
\item $\square_\mu$ asserts the existence of a coherent $\mu$-bounded $C$-sequence over $\mu^+$;
\item $\square^*_\mu$ asserts the existence of a weakly coherent $\mu$-bounded $C$-sequence over $\mu^+$.
\end{itemize}
\end{defn}

We generalise the preceding as follows:
\begin{defn}\label{def411}
\begin{itemize}
\item $\square_{<}(\kappa)$ asserts the existence of a coherent lower regressive $C$-sequence over $\kappa$;
\item $\square_{<}^*(\kappa)$ asserts the existence of a weakly coherent lower regressive $C$-sequence over $\kappa$.
\end{itemize}
\end{defn}

It is not hard to see that $\square_{<}(\mu^+)\iff \square_\mu$ and likewise $\square_{<}^*(\mu^+)\iff\square^*_\mu$.
In Section~\ref{canonicalsect}, it will be shown that the existence of a special $\kappa$-Aronszajn tree
is equivalent to $\square_{<}^*(\kappa)$ which is also equivalent to the existence of a $\square_{<}^*(\kappa)$-sequence $\vec C$ for which $V'(\vec C)\cap R'(\vec C)$ covers a club.

\subsection{Walks on ordinals} In this paper our focus is on constructing canonical trees from walks on ordinals which we now turn to discuss.

\begin{defn}[Todor{\v{c}}evi{\'c}, \cite{TodActa}]\label{defn21}
From a $C$-sequence $\langle C_\delta\mid\delta<\kappa\rangle$,
derive maps $\Tr:[\kappa]^2\rightarrow{}^\omega\kappa$,
$\rho_0:[\kappa]^2\rightarrow {}^{<\omega}\kappa$,
$\rho_1:[\kappa]^2\rightarrow \kappa$,
$\rho_2:[\kappa]^2\rightarrow \omega$,
$\tr:[\kappa]^2\rightarrow{}^{<\omega}\kappa$,
and $\lambda:[\kappa]^2\rightarrow\kappa$
as follows.
Let $(\beta,\gamma)\in[\kappa]^2$ be arbitrary.
\begin{itemize}
\item $\Tr(\beta,\gamma):\omega\rightarrow\kappa$ is defined by recursion on $n<\omega$:
$$\Tr(\beta,\gamma)(n):=\begin{cases}
\gamma,&n=0\\
\min(C_{\Tr(\beta,\gamma)(n-1)}\setminus\beta),&n>0\ \&\ \Tr(\beta,\gamma)(n-1)>\beta\\
\beta,&\text{otherwise}
\end{cases}
$$
\item $\rho_2(\beta,\gamma):=\min\{l<\omega\mid \Tr(\beta,\gamma)(l)=\beta\}$;
\item $\rho_1(\beta,\gamma):=\sup(\rho_0(\beta,\gamma))$, where
\item $\rho_0(\beta,\gamma):=\langle \otp(C_{\Tr(\beta,\gamma)(n)}\cap\beta) \mid n<\rho_2(\beta,\gamma)\rangle$;
\item $\tr(\beta,\gamma):=\Tr(\beta,\gamma)\restriction \rho_2(\beta,\gamma)$;
\item $\lambda(\beta,\gamma):=\sup\{ \sup(C_\delta\cap\beta) \mid \delta\in\im(\tr(\beta,\gamma))\}$.
\end{itemize}
\end{defn}
\begin{remark}\label{rmk412} Thanks to letting $f<_{\lex}g$ whenever $f\supseteq g$ in Definition~\ref{lexo},
one gets that $\rho_{0\gamma}$ is an order-preserving map from $(\gamma,\in)$ to $({}^{<\omega}\gamma,<_{\lex})$.
\end{remark}
For a given $C$-sequence $\vec C$, if $\rho_0, \rho_1, \rho_2$ are derived with respect to $\vec C$, then using Definition~\ref{assoctree} we can define the trees $T(\rho_0), T(\rho_1)$, and $T(\rho_2)$ which play a central role in this paper.

The next fact is well-known.
The first two implications are trivial, and the implication $(3)\implies(1)$ can be proved as in \cite[Claim~4.11.3]{paper34}
utilizing the feature that $\sup(\rho_0(\alpha,\beta))\le\alpha$ for all $\alpha<\beta<\kappa$.

\begin{fact}[Todor\v{c}evi\'c]\label{lemma41}
Suppose that $\rho_0,\rho_2$ are the characteristic functions obtained by walking along $\vec C$.
Then the following are equivalent:
\begin{enumerate}[(1)]
\item $T(\rho_0)$ is a $\kappa$-tree;
\item $T(\rho_2)$ is a $\kappa$-tree;
\item $\vec C$ is weakly coherent.
\end{enumerate}
\end{fact}

\begin{remark}\label{rmkrapid} Regardless of whether a $\vec C$ over $\kappa$ is weakly coherent, there exists a sparse enough club $D\s\kappa$ such that for every $i<3$,
for every pair $\gamma<\delta$ of successive elements of $D$, for every node $t\in T(\rho_i)\restriction[\gamma,\delta)$,
it is the case that both $\h(t)$ and the von Neumann rank of $t$ belong to the interval $[\gamma,\delta)$.
Therefore, for every $n<\omega$,
every $S\s{}^nT(\rho_i)$ of size $\kappa$ that is $\h$-rapid contains some $S'\s S$ of size $\kappa$ that is rapid,
and every $S\s{}^nT(\rho_i)$ of size $\kappa$ that is rapid contains some $S'\s S$ of size $\kappa$ that is $\h$-rapid.
This means that in the context of the trees $T(\rho_0),T(\rho_1),T(\rho_2)$, we may identify Definition~\ref{meetarrow} with the variation obtained by replacing $\h$-rapid by rapid.
Hereafter, in any context in which the two definitions coincide,
we shall opt to work with plain rapid-ness.
\end{remark}

\begin{fact}[{\cite[Lemma~2.1.11]{MR2355670}}]\label{countryfull} Suppose that $T(\rho_0)$ is obtained from walking along a $\mu$-bounded $C$-sequence over $\mu^+$.
If $\mu^{<\mu}=\mu$, then ${(T(\rho_0),<_{\lex})}$ is a $\mu^+$-Countryman line.
\end{fact}

Throughout the paper we shall repeatedly use the following fundamental feature of walks (see, e.g., \cite[Claim~3.1.2]{paper15} for a proof).
\begin{fact}\label{fact2} Whenever $\lambda(\beta,\gamma)<\alpha<\beta<\gamma<\kappa$,
$\tr(\alpha,\gamma)=\tr(\beta,\gamma){}^\smallfrown \tr(\alpha,\beta)$.
\end{fact}

A sufficient condition for $\lambda(\beta,\gamma)$ to be smaller than $\beta$ reads as follows.
\begin{fact} For all $\beta<\gamma<\kappa$, if $\beta\in A(\vec C)$, then $\lambda(\beta,\gamma)<\beta$.
\end{fact}

In \cite{paper18}, a natural variation $\lambda_2(\ldots)$ of $\lambda(\ldots)$ was considered. It satisfies
$\lambda_2(\beta, \gamma) < \beta$ outright.

\begin{defn}[{\cite[Definition~2.8]{paper18}}]
Define $\lambda_2:[\kappa]^2\rightarrow\kappa$ via
$$\lambda_2(\beta,\gamma):=\sup\{ \sup(C_\delta\cap\beta)
\mid \delta \in \im(\tr(\beta, \gamma)),\, \sup(C_\delta\cap\beta) < \beta\}.$$
\end{defn}

Put differently, $\lambda_2(\beta,\gamma)$ is $\lambda(\last{\beta}{\gamma},\gamma)$ for the following ordinal $\last{\beta}{\gamma}$.\footnote{Colloquially, the `last' ordinal in the walk from $\gamma$ to $\beta$.}

\begin{defn}[{\cite[Definition~2.10]{paper44}}]
For every $(\beta,\gamma)\in[\kappa]^2$, define an ordinal $\last{\beta}{\gamma}\in[\beta,\gamma]$ via:
$$\last{\beta}{\gamma}:=\begin{cases}
\beta,&\text{if }\lambda(\beta,\gamma)<\beta;\\
\min(\im(\tr(\beta,\gamma))),&\text{otherwise}.
\end{cases}$$
\end{defn}

\begin{fact}[{\cite[Lemma~2.11]{paper44}}]\label{last} Let $(\beta,\gamma)\in[\kappa]^2$ with $\beta>0$. Then
\begin{enumerate}
\item $\lambda(\last{\beta}{\gamma},\gamma)<\beta$;
\item If $\last{\beta}{\gamma}\neq\beta$, then $\beta\in\acc(C_{\last{\beta}{\gamma}})$. In particular, $\sup(C_{\last{\beta}{\gamma}}\cap\beta)=\sup(\beta)$;
\item $\tr(\last{\beta}{\gamma},\gamma)\sq\tr(\beta,\gamma)$.
\end{enumerate}
\end{fact}

We close this subsection by mentioning an additional useful fact. For a generalization in the language of Definition~\ref{c-seq_num_def}, see \cite[Lemma~5.8]{paper35}.
\begin{fact}[{\cite[Theorem~6.3.2]{MR2355670}}]\label{nontrivialfact}
Suppose that $\rho_2$ is obtained from walking along a $\vec C$ over $\kappa$.
Then $\vec C$ is nontrivial iff $\rho_2$ witnesses $\U(\kappa,\kappa,\omega,\omega)$.
\end{fact}

\subsection{Club-guessing}
\begin{defn}[{\cite[Definition~2.2]{paper46}}]
For subsets $S,T$ of $\kappa$ and an ordinal $\xi\le\kappa$,
$\cg_\xi(S, T)$ asserts the existence of
a $\xi$-bounded $C$-sequence $\vec C=\langle C_\delta\mid\delta \in S\rangle$ such that,
for every club $D\s\kappa$, there exists a $\delta \in S$ such that
$\sup(\nacc(C_\delta)\cap D)=\delta$.
\end{defn}
Note that the quantification over clubs above implies that the existence of a single $\delta$ is equivalent to the existence of stationarily many such $\delta$'s.

\begin{fact}[{\cite[Theorem~3.23]{paper46}}]\label{weaksquarethm}
Suppose that $\mu$ is a regular uncountable cardinal.
Then $\square^*_\mu$ holds iff there exists a $\square^*_\mu$-sequence $\vec C$ such that
$\vec C\restriction E^{\mu^+}_\mu$ witnesses $\cg_\mu(E^{\mu^+}_\mu,E^{\mu^+}_\mu)$.
\end{fact}

\begin{defn}[{\cite[Definition~4.1]{paper46}}]\label{defpi}
Let $S,T$ be stationary subsets of $\kappa$, and $\vec C= \langle C_\delta \mid \delta \in S\rangle$ a $C$-sequence.
We denote by $\Theta_1(\vec C,T)$ the set of all nonzero cardinals $\theta$ for which there exists a function $h:\kappa\rightarrow\theta$ satisfying the following.

For every club $D\s\kappa$, there exists a $\delta\in S$ such that, for every $\tau<\theta$,
$$\sup\{\gamma\in \nacc(C_\delta) \cap D \cap T\mid h(\otp(C_\delta\cap\gamma))=\tau\}= \delta.$$
\end{defn}

Note that $\Theta_1(\vec C,T)\s\kappa$,
and that $\cg_\xi(S,T)$ holds iff there exists a $\xi$-bounded $C$-sequence $\vec C$ over $S$ such that $1\in\Theta_1(\vec C,T)$.

The next lemma uses the notation $\sigma_h$ defined in Subsection~\ref{nandc}.

\begin{lemma}\label{cor62} Suppose $\vec C$ is a $C$-sequence over $\kappa$ and $\theta\in \Theta_1(\vec C\restriction A(\vec C),\kappa)$,
as witnessed by a map $h$.
Then, for every $\Upsilon\in[\kappa]^{\kappa}$,
$$\sup\{\delta<\kappa\mid \sup\{ (\rho_0(\delta,\upsilon))_{h}\mid \upsilon\in \Upsilon\setminus\delta\}<\theta\}<\kappa.$$
\end{lemma}
\begin{proof} Write $\vec C$ as $\langle C_\delta\mid\delta<\kappa\rangle$.
By Proposition~\ref{lemma44}, it suffices to prove that $(\alpha,\beta)\mapsto(\rho_0(\alpha,\beta))_{h}$ witnesses $\U(\kappa,2,\theta,3)$.
Thus, given a pairwise disjoint family $\mathcal B\s[\kappa]^2$ of size $\kappa$ and a colour $\tau<\theta$, we argue as follows.
Let $D$ be the club of $\gamma<\kappa$ such that for every $\beta<\gamma$,
there exists an $a\in\mathcal B$ with $\beta<\min(a)<\max(a)<\gamma$.
Pick $\delta\in A(\vec C)$ such that
$$\sup\{\gamma\in \nacc(C_\delta)\cap D\mid h(\otp(C_\delta\cap\gamma))=\tau\}=\delta,$$
and then pick $b\in\mathcal B$ with $\min(b)>\delta$.
As $\delta\in A(\vec C)$, $\epsilon:=\max\{\lambda(\delta,\beta)\mid \beta\in b\}$ is smaller than $\delta$.
Pick $\gamma\in \nacc(C_\delta)\cap D$ for which $\beta:=\sup(C_\delta\cap\gamma)$ is bigger than $\epsilon$ and such that $h(\otp(C_\delta\cap\gamma))=\tau$.
As $\gamma\in D$,
let us now pick $a\in\mathcal B$ with $\beta<\min(a)<\max(a)<\gamma$.
Then for all $i,j<2$, $\delta\in\tr(a(i),b(j))$ and hence
$(\rho_0(a(i),b(j)))_{h}\ge h(\otp(C_\delta\cap\gamma))=\tau$, as sought.
\end{proof}
\begin{remark} A straightforward tweak of the preceding proof shows that
if $\vec C$ is a $C$-sequence over $\kappa$ and $\theta\in \Theta_1(\vec C\restriction A(\vec C)\cap E^\kappa_{\ge\chi},\kappa)$,
as witnessed by a map $h$, then $(\alpha,\beta)\mapsto(\rho_0(\alpha,\beta))_{h}$ witnesses $\U(\kappa,\kappa,\theta,\chi)$.
\end{remark}

We now introduce the following variation of Definition~\ref{defpi}.

\begin{defn}\label{def429}
Let $S,T$ be stationary subsets of $\kappa$, and $\vec C= \langle C_\delta \mid \delta \in S\rangle$ a $C$-sequence.
Let $\mu$ be some cardinal. We denote by
$\Theta^\mu_1(\vec C, T)$ the set of all non-zero cardinals $\theta$ for which there are functions
$g:\kappa\rightarrow\mu$ and $h:\kappa\rightarrow\theta$, satisfying the following.

For every club $D\s\kappa$, for every $\nu<\mu$, there exists a $\delta\in S$ such that for every $\tau<\theta$,
$$ \sup\left\{\gamma \in \nacc(C_\delta) \cap D\cap T \,\middle|\, \begin{array}{l}g(\otp(C_\delta \cap \gamma)) = \nu\\h(\otp(C_\delta \cap \gamma)) = \tau\end{array} \right\} = \delta.$$
\end{defn}
Note that $\Theta^\mu_1(\vec C, T)\s \Theta_1(\vec C, T)$, and that for every $\theta\in\Theta_1(\vec C, T)$, we have $\theta\in\Theta^\theta_1(\vec C, T)$.
The utility of Definition~\ref{def429} is demonstrated by the following simple lemma
and its subsequent higher-dimensional variation.

\begin{lemma}\label{lemma35a}
Suppose $\vec C$ is a $C$-sequence and $\theta\in \Theta_1^{\mu}(\vec C\restriction A(\vec C),\kappa)$,
as witnessed by maps $g,h$.

For every cofinal $\Upsilon_0\s\kappa$, for every $\nu<\mu$,
there are a stationary $\Upsilon_1\s\kappa$,
an ordinal $\varepsilon<\kappa$, a sequence $\eta\in{}^{<\omega}\theta$, and some $\ell<\omega$
such that for every $\upsilon_1\in\Upsilon_1$, there is a $\upsilon_0\in\Upsilon_0$ above $\upsilon_1$ such that:
\begin{itemize}
\item for every $\delta\in(\varepsilon,\upsilon_1)$, $\tr(\delta,\upsilon_0)(\ell+1)=\upsilon_1$ and
$g(\rho_0(\delta,\upsilon_0)(\ell))=\nu$;
\item $h\circ \rho_0(\upsilon_1,\upsilon_0)=\eta$.
\end{itemize}
\end{lemma}
\begin{proof} Write $\vec C$ as $\langle C_\delta\mid\delta<\kappa\rangle$.
Given a cofinal $\Upsilon_0$ and a prescribed colour $\nu<\mu$,
let $S$ be the collection of all $\upsilon<\kappa$ for which there exists some $\beta\in\Upsilon_0$ satisfying
$\lambda(\upsilon,\beta)<\upsilon<\beta$ and $g(\rho_0(\upsilon,\beta)(\rho_2(\upsilon,\beta)-1))=\nu$.
\begin{claim}\label{cl4301} $S$ is stationary.
\end{claim}
\begin{why}
Let $D$ be an arbitrary club. Pick some $\alpha\in A(\vec C)$ such that
$$ \sup\{\upsilon \in \nacc(C_\alpha) \cap D \mid g(\otp(C_\alpha \cap \upsilon)) = \nu\} = \alpha.$$
Fix $\beta\in\Upsilon_0$ above $\alpha$.
As $\alpha\in A(\vec C)$, it is the case that $\lambda(\alpha,\beta)<\alpha$.
Pick a large enough $\upsilon\in\nacc(C_\alpha)\cap D$ such that $\upsilon^-:=\sup(C_\alpha\cap\upsilon)$ is bigger than $\lambda(\alpha,\beta)$,
and $g(\otp(C_\alpha\cap\upsilon))=\nu$.
It follows that $$\lambda(\alpha,\beta)<\lambda(\upsilon,\beta)=\upsilon^-<\upsilon<\alpha<\beta,$$
and $\rho_0(\upsilon,\beta)=\rho_0(\alpha,\beta){}^\smallfrown\langle \otp(C_\alpha\cap\upsilon)\rangle$.
So $\beta$ witnesses that $\upsilon\in S\cap D$.
\end{why}
For each $\upsilon\in S$, pick $\beta_\upsilon\in\Upsilon_0$ witnessing that $\upsilon\in S$.
Then let:
\begin{itemize}
\item $\varepsilon_\upsilon:=\lambda(\upsilon,\beta)$;
\item $\eta_\upsilon:=h\circ \rho_0(\upsilon,\beta)$;
\item $\ell_\upsilon:=\rho_2(\upsilon,\beta)-1$.
\end{itemize}

As $\theta<\kappa$, it follows that we may find $\varepsilon,\eta,\ell$ for which the following set is stationary
$$\Upsilon_1:=\{ \upsilon\in S\mid \varepsilon_\upsilon=\varepsilon, \eta_\upsilon=\eta, \ell_\upsilon=\ell\},$$
which provided the set we were after.
\end{proof}

\begin{lemma}\label{lemma35}
Suppose $\vec C$ is a $C$-sequence over $\kappa$ and $\theta\in \Theta_1^{\mu}(\vec C\restriction A(\vec C),\kappa)$,
as witnessed by maps $g,h$. Let $n$ be a positive integer.

For every rapid $T'\in[T(\rho_0^{\vec C})]^\kappa$, there exists a rapid $S'\in[{}^n T']^\kappa$ satisfying the following.
For every $\nu<\mu$, for every rapid $S\in[{}^nT(\rho_0^{\vec C})]^{\kappa}$, there are
$\varepsilon<\kappa$, $\langle (\eta^{k},\ell^k) \mid k<n\rangle$,
a cofinal $\Upsilon\s\kappa$ and a matrix $\langle \beta^k_\upsilon\mid \upsilon\in\Upsilon, k<n\rangle$
such that all of the following hold:
\begin{enumerate}
\item for every $k<n$, $\upsilon<\beta_\upsilon^k<\kappa$, $h\circ \rho_0(\upsilon,\beta_\upsilon^k)=\eta^{k}$, and $\sup(\eta^{k})=\sup(\eta^{0})$;
\item for every $\delta\in(\varepsilon,\upsilon)$, for every $k<n$,
$\tr(\delta,\beta_\upsilon^k)(\ell^k+1)=\upsilon$, and $g(\rho_0(\delta,\beta_\upsilon^k)(\ell^k))=\nu$;
\item $(\{\eta^{k}\mid k<n\},{\s})$ is an antichain.
If $S\s S'$, then $|\{\eta^{k}\mid k<n\}|=n$;
\item for all $\upsilon\neq\upsilon'$ in $\Upsilon$, there are $s\neq s'$ in $S$ such that for every $k<n$:
\begin{itemize}
\item $s(k)\s \rho_{0\beta_\upsilon^k}$ and $s'(k)\s \rho_{0\beta_{\upsilon'}^k}$,
\item $s(k)$ and $s'(k)$ are incomparable, and
\item $\dom(s(k))<\dom(s'(k))$ iff $\beta_\upsilon^k<\beta_{\upsilon'}^k$.
\end{itemize}
In particular:
\begin{itemize}
\item $s(k)\wedge s'(k)=\rho_{0\beta_\upsilon^k}\wedge \rho_{0\beta_{\upsilon'}^k}$, and
$s(k)<_{\lex}s'(k)$ iff $\rho_{0\beta_\upsilon^k}<_{\lex} \rho_{0\beta_{\upsilon'}^k}$.
\end{itemize}

\end{enumerate}
\end{lemma}
\begin{proof} Let a rapid $T'\in[T(\rho_0^{\vec C})]^\kappa$ be given.
Fix a sparse enough $Z\in[\kappa]^\kappa$ such that for every pair $\xi<\zeta$ of elements of $Z$,
there exists a $t\in T'$ such that $\rho^{\vec C}_{0\zeta}\restriction\xi\s t\s \rho^{\vec C}_{0\zeta}$.
As $A(\vec C)$ is stationary, $\vec C=\langle C_\delta\mid\delta<\kappa\rangle$ is nontrivial, and then Fact~\ref{nontrivialfact}
together with Lemma~\ref{lemma44} imply in particular that the following set is stationary:
$$\Delta:=\{\delta\in A(\vec C)\mid \sup\{ \rho_2(\delta,\zeta)\mid \zeta\in Z\setminus\delta\}=\omega\}.$$
It follows that for each $\delta\in\Delta$, we may pick an $n$-sequence $z_\delta$ of ordinals in $Z$ all of which are bigger than $\min(Z\setminus(\delta+1))$ such that $k\mapsto\rho_2(\delta,z_\delta(k))$ is injective over $n$,
and then let $\epsilon_\delta:=\max_{k<n}\lambda(\delta,z_\delta(k))$. As $\delta \in A(\vec C)$, $\epsilon_\delta < \delta$. Find a stationary $\Delta'\s\Delta$ such that the following two hold:
\begin{itemize}
\item $\delta\mapsto\epsilon_\delta$ is constant over $\Delta'$ with value, say, $\epsilon^*$;
\item for every pair $\gamma<\delta$ of ordinals from $\Delta'$, $\sup(z_\gamma)<\delta$.
\end{itemize}
For every $\delta\in\Delta'$, the definition of $Z$ implies there exists some $n$-sequence $t_\delta$ consisting of elements of $T'$ such that
for every $k<n$, $$\rho_{0z_\delta(k)}^{\vec C}\restriction\delta+1\s t_\delta(k)\s \rho_{0z_\delta(k)}^{\vec C}.$$
Clearly, $S':=\{ t_\delta\mid \delta\in\Delta'\}$ is a rapid subfamily of ${}^nT'$ of size $\kappa$.

To see that $S'$ is as sought,
suppose that we are given $\nu$ and $S$ as above.
Let $Y$ be the collection of all $\upsilon<\kappa$ for which there are $s\in S$ and a sequence $\langle \beta^k\mid k<n\rangle$ satisfying the following six clauses:
\begin{itemize}
\item for every $k<n$, $\lambda(\upsilon,\beta^k)<\upsilon<\beta^k<\kappa$;
\item for every $k<n$, $\rho_{0\beta^k}\restriction(\upsilon+1)\s s(k)\s \rho_{0\beta^k}$;
\item for every $k<n$, $g(\rho_0(\upsilon,\beta^k)(\rho_2(\upsilon,\beta^k)-1))=\nu$;
\item $k\mapsto\sup(h\circ \rho_0(\upsilon,\beta^k))$ is constant over $n$;
\item $(\{ h\circ \rho_0(\upsilon,\beta^k)\mid k<n\},{\s})$ is an antichain;
\item if $S\s S'$, then $k\mapsto\rho_2(\upsilon,\beta^k)$ is injective over $n$.
\end{itemize}
\begin{claim} $Y$ is stationary.
\end{claim}
\begin{why}
Let $D$ be an arbitrary club. Pick some $\alpha\in A(\vec C)$ above $\epsilon^*$ such that for every $\tau<\theta$,
$$ \sup\left\{\upsilon \in \nacc(C_\alpha) \cap D \,\middle|\, \begin{array}{l}g(\otp(C_\alpha \cap \upsilon)) = \nu\\h(\otp(C_\alpha \cap \upsilon)) = \tau\end{array} \right\} = \alpha.$$

$\br$ If $S\s S'$, then since $S$ has size $\kappa$, we may pick a $\delta\in\Delta'$ above $\alpha$ such that $t_\delta\in S$.
In this case, write $t_\delta$ as $s$ and $z_\delta$ as $\langle\beta^k\mid k<n\rangle$.

$\br$ Otherwise, as $S$ is a rapid family of size $\kappa$, we may find an
$s\in S$ and a sequence $\langle \beta^k\mid k<n\rangle$ of ordinals in $\kappa$
such that for every $k<n$, $\rho_{0\beta^k}\restriction(\alpha+1)\s s(k)\s \rho_{0\beta^k}$.

Next, as $\alpha\in A(\vec C)$, $\varepsilon^*:=\max\{ \lambda(\alpha,\beta^k)\mid k<n\}$ is smaller than $\alpha$.
Consider $\tau:=\max\{\sup(h\circ \rho_0(\alpha,\beta^k))\mid k<n\}$, and pick a large enough $\upsilon\in\nacc(C_\alpha)\cap D$ such that $\sup(C_\alpha\cap\upsilon)>\max\{\varepsilon^*,\epsilon^*\}$,
$g(\otp(C_\alpha\cap\upsilon))=\nu$ and $h(\otp(C_\alpha\cap\upsilon))=\tau+1$.
Clearly, $\rho_0(\upsilon,\beta^k)=\rho_0(\alpha,\beta^k){}^\smallfrown\langle \otp(C_\alpha\cap\upsilon)\rangle$ for every $k<n$.
In case that $S\s S'$, we also get that $\rho_2(\upsilon,\beta^k)=\rho_2(\delta,\beta^k)+\rho_2(\upsilon,\delta)$ for every $k<n$.
Altogether, $s$ and $\langle \beta^k\mid k<n\rangle$ witness that $\upsilon\in Y\cap D$.
\end{why}
For each $\upsilon\in Y$, pick $s_\upsilon$ and $\langle \beta^k_\upsilon\mid k<n\rangle$ witnessing that $\upsilon\in Y$. Then let:
\begin{itemize}
\item $\varepsilon_\upsilon:=\max\{\lambda(\upsilon,\beta_\upsilon^k)\mid k<n\}$;
\item $\eta^k_\upsilon:=h\circ \rho_0(\upsilon,\beta_\upsilon^k)$ for every $k<n$;
\item $\ell^k_\upsilon:=\rho_2(\upsilon,\beta_\upsilon^k)-1$ for every $k<n$.
\end{itemize}

As $\theta<\kappa$, it follows that we may find $\varepsilon$ and $\langle (\eta^{k},\ell^k) \mid k<n\rangle$
for which the following set is stationary
$$\Upsilon:=\{ \upsilon\in Y\mid \varepsilon_\upsilon=\varepsilon, \langle (\eta^{k}_\upsilon,\ell^k_\upsilon) \mid k<n\rangle=\langle (\eta^{k},\ell^k) \mid k<n\rangle\}.$$

Let $l:=\max\{\ell^k\mid k<n\}+2$. Using again that $\rho_2$ witnesses $\U(\kappa,\kappa,\omega,\omega)$,
we may find a cofinal subset $\Upsilon'\s\Upsilon$ such that every pair $\upsilon<\upsilon'$ of ordinals in $\Upsilon'$,
$\rho_2(\upsilon,\beta^k_{\upsilon'})>l$, which ensures that $\Delta(\rho_{0\beta_\upsilon^k}, \rho_{0\beta_{\upsilon'}^k})\le \upsilon $,
as in the proof of Proposition~\ref{switchtofibers}.
Then the matrix $\langle \beta^k_\upsilon\mid \upsilon\in\Upsilon', k<n\rangle$ is as sought.
\end{proof}

\begin{lemma}\label{thm421b} Suppose that $\mu=\mu^{<\mu}$ is a regular uncountable cardinal.
Then there exists a $\square^*_\mu$-sequence $\vec C=\langle C_\delta\mid \delta<\mu^+\rangle$ and a set $S\s E^{\mu^+}_\mu$ such that:
\begin{itemize}
\item for every $\delta\in S$, $C_\delta\s\acc(\delta)$;
\item $1\in\Theta_1^\mu(\vec C\restriction S,E^{\mu^+}_\mu)$, and
if $\mu$ is non-ineffable, then moreover $\omega\in\Theta_1^\mu(\vec C\restriction S,E^{\mu^+}_\mu)$.
\end{itemize}
\end{lemma}
\begin{proof} Write $\kappa:=\mu^+$. Fix two surjections $g:\kappa\rightarrow\mu$ and $h:\kappa\rightarrow\omega$ such that, for every $\iota<\kappa$,
$$\{ (g(\iota+\varsigma+1),h(\iota+\varsigma+1))\mid \varsigma<\mu\}=\mu\times\omega.$$
It follows that for each $(\alpha,\iota,\nu,n)\in\kappa\times\kappa\times\mu\times\omega$, we may fix a closed subset $y_{\alpha,\iota,\nu,n}$ of $\acc(\kappa)$ such that:
\begin{itemize}
\item $\min(y_{\alpha,\iota,\nu,n})=\alpha$;
\item $\max(y_{\alpha,\iota,\nu,n})<\alpha+\mu$;
\item $g(\iota+\otp(y_{\alpha,\iota,\nu,n}))=\nu$;
\item $h(\iota+\otp(y_{\alpha,\iota,\nu,n}))=n$.
\end{itemize}

Meanwhile, as $\mu^{<\mu}=\mu$, a $C$-sequence over $\kappa$ is a $\square^*_\mu$-sequence iff it is $\mu$-bounded.
In particular, $\square^*_\mu$ holds and we may let $\vec C=\langle C_\delta\mid\delta<\kappa\rangle$ be a $\square^*_\mu$ sequence such that $\vec C \restriction E^{\mu^+}_\mu$ satisfies $\cg_\mu(E^{\mu^+}_\mu, E^{\mu^+}_\mu)$ as provided to us by Fact~\ref{weaksquarethm}.
Using \cite[Definition~2.8]{paper46}, we may also assume that for every $\delta\in E^{\mu^+}_\mu$,
either $\nacc(C_\delta)\s E^{\mu^+}_\mu$ or $\nacc(C_\delta)\cap E^{\mu^+}_\mu=\emptyset$.
So, letting $S:=\{\delta\in E^{\mu^+}_\mu\mid C_\delta\s\acc(\delta)\}$,
it is the case that $\vec C\restriction S$ witnesses that $\cg_\mu(S,T)$ holds for $T:=E^{\mu^+}_\mu$.
For every $\delta\in S$, let $\langle\delta_i\mid i<\mu\rangle$ denote the increasing enumeration of $C_\delta$.
Next, by \cite[Proposition~4.24]{paper46}, we may fix a decomposition $S=\biguplus_{\nu<\mu}S_\nu$
such that $\vec{C}\restriction S_\nu$ witnesses $\cg_\mu(S_\nu, T)$ for every $\nu<\mu$.

$\br$ If $\mu$ is non-ineffable, then by \cite[Corollary~11.6]{paper47}, $\ubd(\ns_\mu,\omega)$ holds.
Then, for every $\nu<\mu$, by \cite[Theorem~6.4(2)]{paper46}, we may fix a sequence $\langle h_\delta:C_\delta\rightarrow\omega\mid \delta\in S_\nu\rangle$ such that
for every club $D\s\kappa$, there exists a $\delta\in S_\nu$ such that, for every $n<\omega$,
$$\sup\{\gamma\in\nacc(C_\delta)\cap D\cap T\mid h_\delta(\gamma)=n\}=\delta.$$

$\br$
If $\mu$ is ineffable, then for every $\delta\in S$, let $h_\delta:C_\delta\rightarrow\{0\}$ be a constant map.
Clearly, for every club $D\s\kappa$, for every $\nu< \mu$, there exists a $\delta\in S_\nu$ such that
$$\sup\{\gamma\in\nacc(C_\delta)\cap D\cap T\mid h_\delta(\gamma)=0\}=\delta.$$

As a next step, for all $\nu<\mu$ and $\delta\in S_\nu$,
we construct a sequence $\langle x_\delta^i\mid i<\mu\rangle$
by recursion on $i<\mu$, as follows:
\begin{itemize}
\item[$\vartriangleright$] If $\delta_{i+1}\notin T$, then set $x^i_\delta:=\{\delta_i\}$.

\item[$\vartriangleright$] If $\delta_{i+1}\in T$, then $\delta_{i+1}\ge\delta_i+\mu$, then set $x_\delta^i:=y_{\alpha,\iota,\nu,n}$,
for $\alpha:=\delta_i$, $\iota:=\otp(\bigcup_{i'<i}x_\delta^{i'})$ and $n:=h_\delta(\delta_{i+1})$.
\end{itemize}

Finally, for every $\delta<\kappa$, let
$$C_\delta^\bullet:=\begin{cases}
\bigcup\{ x_\delta^i\mid i<\mu\},&\text{if }\delta\in S;\\
C_\delta,&\text{otherwise}
\end{cases}$$

It is clear that for every $\delta\in S$, $C^\bullet_\delta$ is a closed subset of $\delta$ all of whose proper initial segments are of order-type less than $\mu$.
In addition, $C_\delta\s C^\bullet_\delta\s C_\delta\cup\acc(\delta)$, so that $C^\bullet_\delta\s\acc(\delta)$.
So $\vec{C^\bullet}=\langle C^\bullet_\delta\mid\delta<\mu^+\rangle$ is a $\mu$-bounded $C$-sequence,
hence it is $\square^*_\mu$-sequence.
\begin{claim}\label{claim4321} Let $\nu<\mu$, $\delta\in S_\nu$ and $\gamma\in\nacc(C_\delta)\cap T$. Then $\gamma\in\nacc(C_\delta^\bullet)$, $g(\otp(C_\delta^\bullet\cap\gamma))=\nu$ and $h(\otp(C_\delta^\bullet\cap\gamma))=h_\delta(\gamma)$.
\end{claim}
\begin{why} As $\gamma\in C_\delta\s C_\delta^\bullet$ and $\cf(\gamma)=\mu=\otp(C_\delta^\bullet)$, it is the case that $\gamma\in\nacc(C_\delta^\bullet).$
Let $i<\mu$ be such that $\gamma=\delta_{i+1}$. Then
$$C_\delta^\bullet\cap\gamma=\bigcup_{i'\le i}x_\delta^{i'}=\bigcup_{i'<i}x_\delta^{i'}\cup y_{\delta_i,\iota,\nu,n},$$
for $\iota:=\otp(\bigcup_{i'<i}x_\delta^{i'})$ and $n:=h_\delta(\delta_{i+1})$.
Therefore, $g(\otp(C_\delta^\bullet\cap\gamma))=\nu$ and $h(\otp(C_\delta^\bullet\cap\gamma))=h_\delta(\gamma)$.
\end{why}

To verify that $\omega\in\Theta_1^\mu(\vec C\restriction S,T)$ when $\mu$ is non-ineffable, let a club $D\s\kappa$ be given, and let $\nu<\mu$.
Fix a $\delta\in S_\nu$ such that, for every $n<\omega$,
$$\sup\{\gamma\in\nacc(C_\delta)\cap D\cap T\mid h_\delta(\gamma)=n\}=\delta.$$
So, by Claim~\ref{claim4321}, we are done. To verify that $1\in\Theta_1^\mu(\vec C\restriction S,T)$ in general is similar, so our proof is complete.
\end{proof}

\begin{thm}\label{thm421} Suppose that $\mu$ is a regular uncountable cardinal and $\square^*_\mu$ holds.
Then there exists a $\square^*_\mu$-sequence $\vec C=\langle C_\delta\mid \delta<\mu^+\rangle$ and a set $S\s E^{\mu^+}_\mu$ such that:
\begin{itemize}
\item for every $\delta\in S$, $C_\delta\s\acc(\delta)$;
\item $1\in\Theta_1^\mu(\vec C\restriction S,E^{\mu^+}_\mu)$, and
if $\mu$ is non-ineffable, then moreover $\omega\in\Theta_1^\mu(\vec C\restriction S,E^{\mu^+}_\mu)$.
\end{itemize}
\end{thm}
\begin{proof} Write $\kappa:=\mu^+$.
Fix maps $g,h$, and a sequence of closed sets $\langle y_{\alpha,\iota,\nu,n}\mid (\alpha,\iota,\nu,n)\in\kappa\times\kappa\times\mu\times\omega\rangle$
as in the proof of Lemma~\ref{thm421b}.
As $\square^*_\mu$ holds, continuing as in the proof of Lemma~\ref{thm421b},
we fix a $\square^*_\mu$-sequence $\vec C=\langle C_\delta\mid\delta<\kappa\rangle$,
a subset $S$ of $T:=E^{\mu^+}_\mu$
such that $C_\delta\s\acc(\delta)$ for every $\delta\in S$,
and a decomposition $S=\biguplus_{\nu<\mu}S_\nu$
such that $\vec{C}\restriction S_\nu$ witnesses $\cg_\mu(S_\nu, T)$ for every $\nu<\mu$.
For every $\delta\in S$, let $\psi_\delta:C_\delta\rightarrow\mu$ be the collapsing map of $C_\delta$,
and let $\langle\delta_i\mid i<\mu\rangle$ denote the increasing enumeration of $C_\delta$ so that $\psi_\delta(\delta_i) = i$ for every $i< \mu$.

By Lemma~\ref{thm421b}, we may assume that $\mu^{<\mu}>\mu$. This means that we will have to work harder to ensure that the $C$-sequence we will be producing ends up being weakly coherent.
By \cite[Corollary~10.3]{paper47} and as $\mu$ is not weakly compact, $\ubd(J^{\bd}[\mu],\omega)$ holds,
and we may fix a corresponding colouring $c:[\mu]^2\rightarrow\omega$ as in \cite[Fact~4.8(2)]{paper46}.
For every club $D\subseteq\kappa$, for all $\delta\in S$ and $\eta< \mu$, denote
$$D(\eta, \delta):=\{ m<\omega\mid \sup\{\beta\in C_\delta\mid c(\eta,\psi_\delta(\beta))=m\ \&\ \min(C_\delta\setminus(\beta+1))\in D\cap T\}=\delta\}.$$

Hereafter, let $\nu<\mu$.
By \cite[Claims 4.9.2 and 4.9.3]{paper46}, we may pick $\eta_\nu<\mu$ and a club $D_\nu\s\kappa$
such that for every club $D \s D_\nu$, there exists $\delta \in S_\nu$ such that $D(\eta_\nu,\delta)= D_\nu(\eta_\nu,\delta)$ and this set is infinite.
By possibly shrinking $D_\nu$, we may assume that $D_\nu\s\acc^+(T)$.
Next, for every $\delta\in S_\nu$, set
$$N_\delta:=\{\beta \in C_\delta\mid c(\eta_\nu,\psi_\delta(\beta))\notin D_\nu(\eta_\nu,\delta)\ \&\ \min(C_\delta\setminus(\beta+1))\in D_\nu\cap T\}.$$

\begin{claim}\label{c4322} There exists an $\varepsilon_\nu<\mu$ such that
for every club $D \s D_\nu$, there exists a $\delta\in S_\nu$ such that $D(\eta_\nu,\delta)= D_\nu(\eta_\nu,\delta)$ is an infinite set, and
$$\sup(\{\psi_\delta(\beta)\mid \beta\in N_\delta\})=\varepsilon_\nu.$$
\end{claim}
\begin{why} Suppose not. For every $\varepsilon<\mu$, fix a counterexample club $D^\varepsilon\s D_\nu$. Let $D:=\bigcap_{\varepsilon<\mu}D^\varepsilon$.
Now pick a $\delta\in S_\nu$ such that $D(\eta_\nu,\delta)=D_\nu(\eta_\nu,\delta)$ and this set is infinite.
For every $m\in\omega\setminus D(\eta_\nu,\delta)$, the set
$$\{\beta\in C_\delta\mid \eta_\nu< \psi_\delta(\beta)\ \&\ c(\eta_\nu,\psi_\delta(\beta))=m\ \&\ \min(C_\delta\setminus (\beta+1))\in D\cap T\}$$
is bounded in $\delta$. So since $\cf(\delta)=\mu>\omega$ and $D_\nu(\eta_\nu,\delta)=D(\eta_\nu,\delta)$,
the set $N_\delta$ is bounded in $\delta$ and there is some $\varepsilon<\mu$ such that
$$\sup(\{\psi_\delta(\beta)\mid \beta\in N_\delta\})=\varepsilon,$$
contradicting the fact that $D\s D^\varepsilon\s D_\nu$.
\end{why}

Fix $\delta\in S_\nu$.
Let $\varepsilon_\nu$ be given Claim~\ref{c4322}.
For every $i\le\mu$, denote:
$$\Omega_\delta^i:=\{ c(\eta_\nu,j)\mid \varepsilon_\nu<j<i,\ \eta_\nu<j,\ \delta_{j+1}\in D_\nu\cap T\}.$$

We now construct a sequence $\langle x_\delta^i\mid i<\mu\rangle$
by recursion on $i<\mu$, as follows:
\begin{itemize}
\item[$\br$] If $\delta_{i+1}\notin D_\nu$, then set $x^i_\delta:=\{\delta_i\}$.
\item[$\br$] If $\delta_{i+1}\in D_\nu$, then in particular, $\delta_{i+1}\ge\delta_i+\mu$, so we set $x_\delta^i:=y_{\alpha,\iota,\nu,n}$,
for $\alpha:=\delta_i$, $\iota:=\otp(\bigcup_{i'<i}x_\delta^{i'})$,
and $$n:=\otp(c(\eta_\nu,i)\cap \Omega^i_\delta).$$
\end{itemize}
Having completed the above construction for every $\nu<\mu$ and every $\delta\in S_\nu$,
we now turn to wrap things up.
For every $\delta<\kappa$, let
$$C_\delta^\bullet:=\begin{cases}
\bigcup\{ x_\delta^i\mid i<\mu\},&\text{if }\delta\in S;\\
C_\delta,&\text{otherwise}
\end{cases}$$

It is clear that for every $\delta\in S$, $C^\bullet_\delta$ is a closed subset of $\delta$ all of whose proper initial segments are of order-type less than $\mu$.
In addition, $C_\delta\s C^\bullet_\delta\s C_\delta\cup\acc(\delta)$, so that $C^\bullet_\delta\s\acc(\delta)$.
So $\vec{C^\bullet}=\langle C^\bullet_\delta\mid\delta<\mu^+\rangle$ is a $\mu$-bounded $C$-sequence.

\begin{claim}\label{c4332} $\vec C^\bullet$ is weakly coherent.
\end{claim}
\begin{why} Suppose not, and fix the least $\epsilon<\kappa$
such that $|\{ C_\delta^\bullet\cap\epsilon\mid \delta<\kappa\}|=\kappa$.
It follows that we may fix some $\nu<\mu$ and a cofinal subset $\Delta\s S_\nu$ such that:
\begin{itemize}
\item[(1)] $\delta\mapsto C^\bullet_\delta\cap\epsilon$ is injective over $\Delta$, but
\item[(2)] $\delta\mapsto C_\delta\cap\epsilon$ is constant over $\Delta$.
\end{itemize}

Observe first that for all $\delta,\delta'\in S_\nu$ and $\alpha\in C_\delta\cap C_{\delta'}$, if $C_\delta\cap\alpha=C_{\delta'}\cap\alpha$,
then $C_\delta^\bullet\cap\alpha=C^\bullet_{\delta'}\cap\alpha $ by the tree-like nature of $\Omega_\delta^i$ and $\Omega_{\delta'}^i$.
So we may fix an $\alpha<\epsilon$ such that $\sup(C_\delta\cap\epsilon)=\alpha$ for all $\delta\in\Delta$.
By minimality of $\epsilon$, and by possibly shrinking $\Delta$ further, we may also assume that
\begin{itemize}
\item[(3)] $\delta\mapsto C^\bullet_\delta\cap\alpha$ is constant over $\Delta$.
\end{itemize}
It thus follows that the map $\delta\mapsto C^\bullet_\delta\cap[\alpha,\epsilon)$ is injective over $\Delta$.
However, for every $\delta\in\Delta$, $C^\bullet_\delta\cap[\alpha,\epsilon)$
is equal to $y_{\alpha,\iota,\nu,n}\cap\epsilon$
for $\iota:=\otp(C^\bullet_\delta\cap\alpha)$
and some $n<\omega$, or just $\{\alpha\}$.
Recalling Clause~(3), there exists an $\iota<\kappa$ such that:
$$\{ C^\bullet_\delta\cap[\alpha,\epsilon)\mid \delta\in\Delta\}\s\{y_{\alpha,\iota,\nu,n}\cap\epsilon\mid n<\omega\}\cup \{\alpha\},$$
contradicting the fact that the set on the right hand size is countable.
\end{why}

To see that $\omega\in \Theta_1^\mu(\vec{C^\bullet}\restriction S,T)$, let $\nu<\mu$,
and let $D$ be a club in $\kappa$. By possibly shrinking $D$, we may assume that $D\s D_\nu$.
Pick $\delta \in S_\nu$ such that $D(\eta_\nu,\delta)= D_\nu(\eta_\nu,\delta)$ is infinite, and
$$\sup(\{\psi_\delta(\beta)\mid \beta\in N_\delta\})=\varepsilon_\nu.$$
For any $D'\in\{D,D_\nu\}$,
$$D'(\eta_\nu,\delta)=\{ m<\omega\mid \sup\{j<\mu\mid c(\eta_\nu,j)=m\ \&\ \delta_{j+1}\in D'\cap T\}=\mu\}.$$
So, since $D(\eta_\nu,\delta)= D_\nu(\eta_\nu,\delta)$, the definition of $\varepsilon_\nu$ implies that
$$D(\eta_\nu,\delta)=\{ c(\eta_\nu,j)\mid \varepsilon_\nu<j<\mu,\ \eta_\nu<j,\ \delta_{j+1}\in D_\nu\cap T\}=\Omega_\delta^{\mu}.$$
In particular, $\Omega_\delta^{\mu}$ is infinite.
Now, given an $n<\omega$ and some $\epsilon<\delta$,
we shall find a $\bar\delta\in\nacc(C^\bullet_\delta)\cap D\cap T$ above $\epsilon$ such that
$g(\otp(C^\bullet_\delta\cap\bar\delta))=\nu$ and $h(\otp(C^\bullet_\delta\cap\bar\delta))=n$.
Fix the unique $m\in \Omega^\mu_\delta$ such that $\otp(\Omega^\mu_\delta\cap m)=n$.
Note that for a tail of $i<\mu$, $\Omega_\delta^\mu\cap m=\Omega_\delta^i\cap m$.

As $m\in D(\eta_\nu,\delta)$, there are cofinally many $\beta\in C_\delta$ such that $\min(C_\delta\setminus(\beta+1))\in D\cap T$,
$\eta<\psi_\delta(\beta)$ and $c(\eta_\nu,\psi_\delta(\beta))=m$. So, we may find a large enough $i<\mu$ such that:
\begin{itemize}
\item $\delta_{i+1}\in D\cap T$,
\item $\eta_\nu<i$,
\item $c(\eta_\nu,i)=m$,
\item $\epsilon<\delta_i$, and
\item $\Omega_\delta^\mu\cap m=\Omega_\delta^i\cap m$.
\end{itemize}

So
$$\otp(c(\eta_\nu,i)\cap \Omega^i_\delta)=\otp(m\cap \Omega^i_\delta)=\otp(m\cap \Omega^\mu_\delta)=n.$$
Put $\beta:=\max(x_\delta^i)$ so that
$\delta_i\le\beta<\delta_{i+1}=\min(C^\bullet_\delta\setminus(\beta+1))$.
Since $\delta_{i+1}\in D\s D_\nu$, we know that
$x_\delta^i=y_{\alpha,\iota,\nu,n}$,
for $\alpha:=\delta_i$ and $\iota:=\otp(\bigcup_{i'<i}x_\delta^{i'})$.
Consequently,
\begin{itemize}
\item $g(\otp(C^\bullet_\delta\cap\delta_{i+1}))=g(\otp(\bigcup\nolimits_{i'\le i}x^\delta_{i'}))=\nu$, and
\item $h(\otp(C^\bullet_\delta\cap\delta_{i+1}))=h(\otp(\bigcup\nolimits_{i'\le i}x^\delta_{i'}))=n$,
\end{itemize}
as sought.
\end{proof}

\section{\texorpdfstring{$T(\rho_0)$}{T(rho0)}}\label{sec5}

A useful fact concerning walks along an $\omega$-bounded $C$-sequence over $\omega_1$ is that
for all $\alpha<\beta<\omega_1$ such that $\rho_{0\alpha}$ and $\rho_{0\beta}$ are incomparable, $\Delta(\rho_{0\alpha},\rho_{0\beta})\in\nacc(\omega_1)$.
One may hope that this feature generalises,
however, it is not hard to construct an $\omega_1$-bounded $C$-sequence $\vec C$ over $\omega_2$
such that, for stationarily many $\delta<\omega_2$, there are $\alpha<\beta<\omega_2$ such that $\Delta(\rho_{0\alpha},\rho_{0\beta})=\delta$.

The correct generalization of this feature above $\omega_1$ is given by the next lemma.
To this end, fix a $C$-sequence $\vec C=\langle C_\delta\mid\delta<\kappa\rangle$, and throughout this section let $\rho_0,\rho_2,\Tr,\tr$ denote the corresponding characteristic functions.
\begin{lemma}\label{lemma2444}
For all $\alpha<\beta<\kappa$ such that $\rho_{0\alpha}$ and $\rho_{0\beta}$ are incomparable,
$\Delta(\rho_{0\alpha},\rho_{0\beta})\notin V(\vec C)$.
\end{lemma}
\begin{proof} We prove a slightly more general assertion. Let $h:\kappa\rightarrow\kappa$ be any map
that is \emph{good} in the sense that for all $\alpha,\beta,\gamma<\kappa$ with $\gamma\in\acc(\kappa)$,
if $h(\alpha+\iota)=h(\beta+\iota)$ for all $\iota<\gamma$ then $h(\alpha+\gamma)=h(\beta+\gamma)$.
Of course, the identity map is good.

For the sake of this proof, define $\rho_h:[\kappa]^2\rightarrow{}^{<\omega}\kappa$ via $\rho_h(\alpha,\beta):=h\circ\rho_0(\alpha,\beta)$.
Now, let $(\alpha, \beta) \in [\kappa]^2$ be such that $\rho_{h\alpha}$ and $\rho_{h\beta}$ are incomparable, and denote $\delta:= \Delta(\rho_{h\alpha},\rho_{h\beta})$.
We shall show that $\delta \notin V(\vec C)$. To this end, we may assume that $\delta \in \acc(\kappa)$.
Let $\bar \alpha: =\last \delta \alpha$ and $\bar \beta:= \last\delta \beta$, so that $\sup(C_{\bar\alpha}\cap\delta)=\delta=\sup(C_{\bar\beta}\cap\delta)$.

Let $\epsilon:= \sup\{\lambda_2(\delta, \alpha), \lambda_2(\delta, \beta)\}$,
so that, for every $\xi \in (\epsilon, \delta)$,
\begin{enumerate}
\item $\tr(\xi, \alpha) = \tr(\bar \alpha, \alpha){}^\smallfrown\tr(\xi, \bar \alpha)$, and
\item $\tr(\xi, \beta) = \tr(\bar \beta, \beta){}^\smallfrown\tr(\xi, \bar \beta)$.
\end{enumerate}

\begin{claim}\label{claim24441} $\rho_2(\bar\alpha,\alpha)=\rho_2(\bar\beta,\beta)$.
\end{claim}
\begin{why} Suppose not. Without loss of generality, $n:=\rho_2(\bar\alpha,\alpha)$ is smaller than $\rho_2(\bar\beta,\beta)$.
Since $\delta= \Delta(\rho_{h\alpha},\rho_{h\beta})$,
in particular, $\rho_{2\alpha}\restriction\delta = \rho_{2\beta}\restriction\delta$.
Together with Clause~(i), this yields that for every $\xi\in (\epsilon,\delta)\cap C_{\bar\alpha}$,
$$\rho_2(\xi,\beta)=\rho_2(\xi,\alpha)=n+1\le\rho_2(\bar\beta,\beta).$$
By Clause~(ii) this means that $\sup(C_{\Tr(\bar\beta,\beta)(n)}\cap\delta)=\delta$ and hence $\rho_2(\bar\beta,\beta)=n$.
This is a contradiction.
\end{why}

It thus follows from Clauses (i) and (ii) that $\rho_h(\bar\alpha,\alpha)=\rho_h(\bar\beta,\beta)$.

Now, letting $n$ denote the unique element of $\{\rho_2(\bar\alpha,\alpha),\rho_2(\bar\beta,\beta)\}$,
and letting $\varrho$ denote the unique element of $\{\rho_{h\alpha}\restriction \delta,\rho_{h\beta}\restriction\delta\}$,
we get that
$$C_{\bar\alpha}\cap(\epsilon,\delta)=\{ \xi\in(\epsilon,\delta)\mid |\varrho(\xi,\alpha)|=n+1\}=C_{\bar\beta}\cap(\epsilon,\delta).$$

In addition, since $h(\otp(C_{\bar\alpha}\cap\xi))=\varrho(\xi)(n)=h(\otp(C_{\bar\beta}\cap\xi))$ for every $\xi\in(\epsilon,\delta)$,
the fact that $h$ is good implies that $h(\otp(C_{\bar\alpha}\cap\delta))=h(\otp(C_{\bar\beta}\cap\delta))$.

Now, there are four options to consider:
\begin{itemize}
\item[$\br$] If $\bar\beta=\delta<\bar\alpha$, then $C_{\bar\alpha}\cap(\epsilon,\delta)=C_{\delta}\cap(\epsilon,\delta)$, and hence $\delta \notin V(\vec C)$.
\item[$\br$] If $\bar\alpha=\delta<\bar\beta$, then $C_\delta\cap(\epsilon,\delta)=C_{\bar\beta}\cap(\epsilon,\delta)$, and hence $\delta \notin V(\vec C)$.
\item[$\br$] If $\bar\alpha=\delta=\bar\beta$, then $\rho_{h}(\delta,\alpha)=\rho_h(\bar\alpha,\alpha)=\rho_h(\bar\beta,\beta)=\rho_h(\delta,\beta)$, contradicting the choice of $\delta$.
\item[$\br$] If $\min\{\bar\alpha,\bar\beta\}>\delta$, then
$$\begin{aligned}\rho_{h}(\delta,\alpha)
=&\ \rho_h(\bar\alpha,\alpha){}^\smallfrown\langle h(\otp(C_{\bar\alpha}\cap\delta))\rangle\\
=&\ \rho_h(\bar\beta,\beta){}^\smallfrown\langle h(\otp(C_{\bar\beta}\cap\delta))\rangle=\rho_h(\delta,\beta),
\end{aligned}$$
contradicting the choice of $\delta$.
\qedhere
\end{itemize}
\end{proof}

An alternative generalization of the feature known at the level of $\omega_1$ is given by the next lemma.
\begin{lemma}\label{hasudroff} Let $\alpha<\beta<\kappa$ such that $\rho_{0\alpha}$ and $\rho_{0\beta}$ are incomparable.

If $\Delta(\rho_{0\alpha},\rho_{0\beta})\neq\Delta(\rho_{2\alpha},\rho_{2\beta})$, then $\Delta(\rho_{0\alpha},\rho_{0\beta})\in\nacc(\kappa)$.
\end{lemma}
\begin{proof}
We prove the contrapositive. Suppose that $\delta:=\Delta(\rho_{0\alpha},\rho_{0\beta})$ is in $\acc(\kappa)$.
Denote $\epsilon:=\max\{\lambda_2(\delta,\alpha),\lambda_2(\delta,\beta)\}$.
\begin{claim} Let $\gamma\in\{\alpha,\beta\}$. For every $i<\rho_2(\delta,\gamma)$:
\begin{enumerate}
\item For every $\xi\in (\epsilon,\delta)$,
$\Tr(\xi,\gamma)(i)=\Tr(\delta,\gamma)(i)$;
\item $\rho_0(\delta,\gamma)(i)=\sup\{ \rho_0(\xi,\gamma)(i) \mid \xi\in (\epsilon,\delta)\}$.
\end{enumerate}
\end{claim}
\begin{why} (i) By Fact~\ref{fact2}.

(ii) By Clause~(i).
\end{why}

It follows that for every $i<\min\{\rho_2(\delta,\beta),\rho_2(\delta,\alpha)\}$:
$$\begin{aligned}\rho_0(\delta,\beta)(i)=&\ \sup\{ \rho_0(\xi,\beta)(i) \mid \xi\in (\epsilon,\delta)\}\\
=&\ \sup\{ \rho_0(\xi,\alpha)(i) \mid \xi\in (\epsilon,\delta)\}=\rho_0(\delta,\alpha)(i),\end{aligned}$$
so since $\rho_0(\delta,\beta)\neq\rho_0(\delta,\alpha)$, it must be the case that $\rho_2(\delta,\beta)\neq \rho_2(\delta,\alpha)$.
In addition, trivially, $\Delta(\rho_{0\alpha},\rho_{0\beta})\le\Delta(\rho_{2\alpha},\rho_{2\beta})$.
Thus, altogether $\Delta(\rho_{0\alpha},\rho_{0\beta})=\Delta(\rho_{2\alpha},\rho_{2\beta})$.
\end{proof}

A special case of an old theorem of Todor\v{c}evi\'c asserts that by walking along an $\omega$-bounded $C$-sequence over $\omega_1$,
$T(\rho_0)$ is a special Aronszajn tree.
More recently, Peng \cite{peng} proved
that in the same context, $T(\rho_2)$ is special.
Here it is established that this is not a coincidence.
\begin{lemma}\label{lemma511}
The following are equivalent:
\begin{enumerate}[(1)]
\item $T(\rho_0)$ is a special $\kappa$-tree;
\item $T(\rho_2)$ is a special $\kappa$-tree.
\end{enumerate}
\end{lemma}
\begin{proof} $(2)\implies(1)$: Suppose that $T(\rho_2)$ is a special $\kappa$-tree,
and fix a good colouring $d:T(\rho_2)\restriction D\rightarrow\kappa$.
By Fact~\ref{lemma41}, $T(\rho_0)$ is a $\kappa$-tree.
For every sequence $\sigma$, let $\ell(\sigma)$ denote its length. Next, define a colouring $c:T(\rho_0)\restriction D\rightarrow\kappa$ via
$$c(x):=d(\ell\circ x).$$
It is clear that $c(x)<\dom(x)$ for every $x\in T(\rho_0)\restriction D$.
Finally, given a pair $x\subsetneq y$ of nodes in $T(\rho_0)$,
it is the case that $\ell\circ x\subsetneq \ell\circ y$, and hence $c(x)\neq c(y)$.

$(1)\implies(2)$: Suppose that $T(\rho_0)$ is a special $\kappa$-tree,
and fix a good colouring $d:T(\rho_0)\restriction D\rightarrow\kappa$.
By Fact~\ref{lemma41}, $T(\rho_2)$ is a $\kappa$-tree.
Fix a bijection $\pi:3\times\omega\times T(\rho_2)\times \kappa\leftrightarrow\kappa$,
and consider the club $$C:=\{\gamma\in D\cap\acc(\kappa)\mid \pi[3\times\omega\times(T(\rho_2)\restriction\gamma)\times \gamma]= \gamma\}.$$

We shall define a colouring $c:T(\rho_2)\restriction C\rightarrow\kappa$, as follows.
Given $x\in T(\rho_2)\restriction C$, set $\gamma:=\dom(x)$ and fix the least $\gamma'\in[\gamma,\kappa)$ such that $x=\rho_{2\gamma'}\restriction\gamma$.
Now, consider the following four options:
\begin{description}
\item[Case 0] If $\gamma'=\gamma$, then let
$$c(x):=\pi(0,0,\emptyset,d(\rho_{0\gamma})).$$
\item[Case 1] If $\gamma'>\gamma$ and $\lambda(\gamma,\gamma')<\gamma$, then denote $\alpha:=\lambda(\gamma,\gamma')+1$ and let:
$$c(x):=\pi(1,\rho_2(\gamma,\gamma'),(\rho_{2\gamma}\restriction\alpha),d(\rho_{0\gamma})).$$
\item[Case 2] If $\gamma'>\gamma$ and $\lambda(\gamma,\gamma')=\gamma$, then denote $\alpha:=\lambda_2(\gamma,\gamma')+1$ and let:
$$c(x):=\pi(2,\rho_2(\gamma,\gamma'),(\rho_{2\last{\gamma}{\gamma'}}\restriction\alpha),d(\rho_{0\last{\gamma}{\gamma'}}\restriction\gamma)).$$
\end{description}

It is clear that $c(x)<\dom(x)$ for every $x\in T(\rho_2)\restriction C$.
Next, suppose that $x\subsetneq y$ is a pair of nodes in $T(\rho_2)\restriction C$.
Towards a contradiction, suppose that $c(x)=c(y)$,
say it is equal to $\pi(i,n,t,\epsilon)$.
Denote $\gamma:=\dom(x)$ and fix the least $\gamma'\in[\gamma,\kappa)$ such that $x=\rho_{2\gamma'}\restriction\gamma$.
Likewise, denote $\delta:=\dom(y)$ and fix the least $\delta'\in[\delta,\kappa)$ such that $y=\rho_{2\delta'}\restriction\delta$.

$\br$ If $i=0$, then $\rho_{2\gamma}=x\subsetneq y=\rho_{2\delta}$.
But then,
$$C_\gamma=\{ \beta<\gamma\mid \rho_{2\gamma}(\beta)=1\}=\{ \beta<\gamma\mid \rho_{2\delta}(\beta)=1\}=C_\delta\cap\gamma.$$
As $\sup(C_\gamma)=\gamma$, it altogether follows that $\rho_{0\gamma}=\rho_{0\delta}\restriction\gamma$,
which must mean that $d(\rho_{0\gamma})\neq d(\rho_{0\delta})$, contradicting the fact that $c(x)=c(y)$.

$\br$ If $i=1$, then $\alpha$ lies in between $\lambda(\gamma,\gamma')$ and $\gamma$,
so for every $\beta\in[\alpha,\gamma)$, $$x(\beta)=\rho_{2\gamma'}(\beta)=\rho_{2\gamma'}(\gamma)+\rho_{2\gamma}(\beta)=n+\rho_{2\gamma}(\beta).$$
Likewise, for every $\beta\in[\alpha,\delta)$, $y(\beta)=n+\rho_{2\delta}(\beta)$.
So from $x=y\restriction\gamma$, it follows that $\rho_{2\gamma}\restriction [\alpha,\gamma)=\rho_{2\delta}\restriction [\alpha,\gamma)$.
In addition, $\rho_{2\gamma}\restriction\alpha=t=\rho_{2\delta}\restriction\alpha$.
So $\rho_{2\gamma}=\rho_{2\delta}\restriction\gamma$, which yields a contradiction, as in the previous case.

$\br$ If $i=2$, then $\alpha$ lies in between $\lambda_2(\gamma,\gamma')$ and $\gamma$,
for every $\beta\in[\alpha,\gamma)$,
$$x(\beta)=\rho_{2\gamma'}(\beta)=\rho_2(\gamma,\gamma')-1+\rho_2(\beta,\last{\gamma}{\gamma'})=n-1+\rho_2(\beta,\last{\gamma}{\gamma'}).$$
Likewise, for every $\beta\in[\alpha,\delta)$, $y(\beta)=n-1+\rho_2(\beta,\last{\delta}{\delta'})$.
So from $x=y\restriction\gamma$, it follows that $\rho_{2\last{\gamma}{\gamma'}}\restriction [\alpha,\gamma)=\rho_{2\last{\delta}{\delta'}}\restriction [\alpha,\gamma)$.
In addition, $\rho_{2\last{\gamma}{\gamma'}}\restriction\alpha=t=\rho_{2\last{\delta}{\delta'}}\restriction\alpha$.
So $\rho_{2\last{\gamma}{\gamma'}}=\rho_{2\last{\delta}{\delta'}}\restriction\gamma$, and then
$$C_{\last{\gamma}{\gamma'}}\cap\gamma=\{ \beta<\gamma\mid \rho_{2\last{\gamma}{\gamma'}}(\beta)=1\}=\{ \beta<\gamma\mid \rho_{2\last{\delta}{\delta'}}(\beta)=1\}=C_{\last{\delta}{\delta'}}\cap\gamma.$$
As $\sup(C_{\last{\gamma}{\gamma'}}\cap\gamma)=\gamma$, it altogether follows that $\rho_{0\last{\gamma}{\gamma'}}\restriction \gamma=\rho_{0\last{\delta}{\delta'}}\restriction\gamma$,
which must mean that $d(\rho_{0\last{\gamma}{\gamma'}}\restriction \gamma)\neq d(\rho_{0\last{\delta}{\delta'}}\restriction \delta)$,
contradicting the fact that $c(x)=c(y)$.
\end{proof}

By a theorem of Mart{\'{\i}}nez-Ranero \cite[Theorem~2.213]{MR2995913},
if $\vec C$ is an $\omega$-bounded $C$-sequence over $\omega_1$,
then $T(\rho_0)$ is isomorphic to an $E^{\omega_1}_\omega$-coherent tree.
Using Lemma~\ref{lemma2444}, we can generalise this result to higher cardinals while also respecting the $<_{\lex}$-ordering, as follows.

\begin{lemma}\label{l54} If $V(\vec C)$ covers a club,
then there exists an injection $\psi:T(\rho_0)\rightarrow{}^{<\kappa}({}^{<\omega}\kappa)$ such that:
\begin{itemize}
\item for every $t\in T(\rho_0)$, $\dom(\psi(t))=\dom(t)$;
\item for all $s,t\in T(\rho_0)$, $s\s t$ iff $\psi(s)\s\psi(t)$;
\item for all $s,t\in T(\rho_0)$, $s<_{\lex}t$ iff $\psi(s)<_{\lex}\psi(t)$;
\item for all $s,t\in T(\rho_0)$, $\Delta(s,t)=\Delta(\psi(s),\psi(t))$;
\item for some club $D\s\kappa$, the image of $\psi$ is an $(A(\vec C)\cap D)$-coherent streamlined tree.
\end{itemize}
\end{lemma}
\begin{proof} Fix a club $E$ in $\kappa$ that is a subset of $V(\vec C)$.
For every $\gamma<\kappa$, for every $t:\gamma\rightarrow{}^{<\omega}\kappa$, we shall define a corresponding $t^*:\gamma\rightarrow{}^{<\omega}\kappa$,
following the basic idea of \cite[Theorem~4.27]{peng}.
Given $\beta<\gamma$, consider $\alpha:=\sup(E\cap\beta)$.

$\br$ If $\alpha=\beta$, then let $t^*(\beta):=\emptyset$.

$\br$ If $\alpha<\beta$, then let $\eta:=t(\alpha)\wedge t(\beta)$ be the longest initial segment of both $t(\alpha)$ and $t(\beta)$,
let $\sigma$ be such that $t(\beta)=\eta{}^\smallfrown\sigma$,
and finally let $$t^*(\beta):=\langle|t(\alpha)|-|\eta|\rangle{}^\smallfrown \sigma.$$

\begin{claim}\label{c541} Suppose $s:\gamma\rightarrow{}^{<\omega}\kappa$ and $t:\delta\rightarrow{}^{<\omega}\kappa$ in $T(\rho_0)$, with $\gamma<\delta$.
\begin{enumerate}
\item If $s\s t$, then $s^*\s t^*$;
\item If $s$ is incomparable with $t$, then $\Delta(s^*,t^*)=\Delta(s,t)$
and $s<_{\lex}t$ iff $s^*<_{\lex}t^*$.
\end{enumerate}
\end{claim}
\begin{why} (i) This is trivial.

(ii) By Lemma~\ref{lemma2444}, $\beta:=\Delta(s,t)$ does not belong to the club $E$, so that $\alpha:=\sup(E\cap\beta)$ is smaller than $\beta$,
which in turn implies that $s(\alpha)=t(\alpha)$.
Clearly, $s^*\restriction\beta=t^*\restriction\beta$.
Now denote $s^*(\beta)$ as $\langle k_s\rangle{}^\smallfrown\sigma_s$ and $\eta_s:=s(\alpha)\wedge s(\beta)$,
and likewise, denote $t^*(\beta)$ as $\langle k_t\rangle{}^\smallfrown\sigma_t$ and $\eta_t:=t(\alpha)\wedge t(\beta)$.
Consider the following options:
\begin{itemize}
\item[$\br$] If $k_s<k_t$, then $\Delta(s^*,t^*)=\beta$ and $s^*<_{\lex}t^*$.
It also follows that $\eta_t\subsetneq \eta_s$, so that $s(\beta)(|\eta_t|)=s(\alpha)(|\eta_t|)=t(\alpha)(|\eta_t|)$.
Now, by Remark~\ref{rmk412}, $t(\alpha)<_{\lex}t(\beta)$,
and hence either $t(\beta)=\eta_t$ or $t(\alpha)(|\eta_t|)<t(\beta)(|\eta_t|)$.

$\br\br$ If $t(\beta)=\eta_t$, then $t(\beta)\subsetneq \eta_s \s s(\beta)$, and hence, $s(\beta)<_{\lex}t(\beta)$.
So $s<_{\lex}t$.

$\br\br$ Otherwise, $s(\beta)(|\eta_t|)=t(\alpha)(|\eta_t|)<t(\beta)(|\eta_t|)$,
and then $s(\beta)<_{\lex}t(\beta)$.
So $s<_{\lex}t$.

\item[$\br$] If $k_t<k_s$, then a similar analysis yields that $\Delta(s^*,t^*)=\beta$, $t^*<_{\lex}s^*$ and $t<_{\lex}s$.
\item[$\br$] If $k_s=k_t$, then $\eta_s=\eta_t$, so since $s(\beta)\neq t(\beta)$, it follows that $\sigma_s\neq\sigma_t$ and hence $\Delta(s^*,t^*)=\beta$.
Furthermore $s<_{\lex}t$ iff $s(\beta)<_{\lex}t(\beta)$ iff $\eta_s{}^\smallfrown\sigma_s<_{\lex}\eta_t{}^\smallfrown\sigma_t$ iff
$\sigma_s<_{\lex}\sigma_t$ iff $\langle k_s\rangle{}^\smallfrown\sigma_s<_{\lex}\langle k_t\rangle{}^\smallfrown\sigma_t$
iff $s^*(\beta)<_{\lex}t^*(\beta)$ iff $s^*<_{\lex}t^*$.\qedhere
\end{itemize}
\end{why}

Set $T:=\{ t^*\mid t\in T(\rho_0)\}$, and note that we have established that the map
$t\stackrel{\psi}{\mapsto}t^*$
constitutes an isomorphism from
$(T(\rho_0),{\s},{<_{\lex}})$ to $(T,{\s},{<_{\lex}})$.
\begin{claim}\label{c542} $T$ is $(A(\vec C)\cap D)$-coherent, for $D:=\acc(E)$.
\end{claim}
\begin{why} By the definition of $T$, it suffices to prove the following. For every $\gamma\in A(\vec C)\cap D$, for every $\delta\in(\gamma,\kappa)$,
$$\sup\{\beta<\gamma\mid (\rho_{0\gamma})^*(\beta)\neq (\rho_{0\delta}\restriction\gamma)^*(\beta)\}<\gamma.$$
To this end, let $\gamma<\delta$ be as above. As $\gamma\in A(\vec C)\cap D$, we have that $\epsilon:=\min(E\setminus\lambda(\gamma,\delta)+1)$ is smaller than $\gamma$.
Now, for every $\alpha\in[\epsilon,\gamma)$, it is the case that $\rho_{0\delta}(\alpha)=\rho_0(\gamma,\delta){}^\smallfrown \rho_{0\gamma}(\alpha)$.
So, for every $\beta\in(\epsilon,\gamma)$, if $\alpha:=\sup(E\cap\beta)$ is smaller than $\beta$,
then letting $\eta:=\rho_{0\gamma}(\alpha)\wedge\rho_{0\gamma}(\beta)$ and letting $\sigma$ be such that $\rho_{0\gamma}(\beta)=\eta{}^\smallfrown\sigma$,
we get that

$$\begin{aligned}
(\rho_{0\gamma})^*(\beta)=&\ \langle|\rho_{0\gamma}(\alpha)|-|\eta|\rangle{}^\smallfrown\sigma\\
=&\ \langle|\rho_0(\gamma,\delta){}^\smallfrown \rho_{0\gamma}(\alpha)|-|\rho_0(\gamma,\delta){}^\smallfrown \eta|\rangle{}^\smallfrown\sigma=(\rho_{0\delta})^*(\beta),
\end{aligned}$$
as sought.
\end{why}
This completes the proof.
\end{proof}

Motivated by Fact~\ref{lemma41} and the preceding result, we now characterise when there exists a $C$-sequence $\vec C$ such that $T(\rho_0^{\vec C})$ is a $\kappa$-tree and $V(\vec C)$ covers a club.
The next result provides a characterization of when this is possible.
Note that in the context of walks on ordinals, it is common to add
to the very definition of $C$-sequences
the requirement that $C_{\gamma+1}=\{\gamma\}$ for all $\gamma$'s (see \cite[\S2.1]{MR2355670}).
In the upcoming, we allow a slightly greater generality in order to support scenarios such as \cite[Definition~5.13]{MR4199849}.

\begin{cor}\label{STP} The following are equivalent:
\begin{enumerate}[(1)]
\item $\kappa$-$\stp$ fails;\footnote{Here $\kappa$-$\stp$ asserts that the \emph{subtle tree property} holds at $\kappa$ (see \cite[p.~13]{Chris10}).}
\item there is a weakly coherent $C$-sequence $\vec C$ over $\kappa$ such that $V'(\vec C)$ covers a club;
\item there is a weakly coherent $C$-sequence $\vec C$ over $\kappa$ such that $V(\vec C)$ covers a club;
\item for every finite $F\s\omega$, there is a weakly coherent $C$-sequence $\vec C$ over $\kappa$ such that:
\begin{itemize}
\item for every $\gamma\in\acc(\kappa)$, $C_{\gamma+1}=F\cup\{\gamma\}$;
\item $\{\Delta(\rho^{\vec C}_{0\alpha},\rho^{\vec C}_{0\beta})\mid \alpha<\beta<\kappa\ \&\ \rho_{0\alpha}\neq\rho_{0\beta}\restriction\alpha\}$ is nonstationary.
\end{itemize}
\item there are a finite $F\s\omega$ and a weakly coherent $C$-sequence $\vec C$ over $\kappa$ such that:
\begin{itemize}
\item for every $\gamma\in\acc(\kappa)$, $C_{\gamma+1}=F\cup\{\gamma\}$;
\item $\{\Delta(\rho^{\vec C}_{0\alpha},\rho^{\vec C}_{0\beta})\mid \alpha<\beta<\kappa\ \&\ \rho_{0\alpha}\neq\rho_{0\beta}\restriction\alpha\}$ is nonstationary.
\end{itemize}
\end{enumerate}
\end{cor}
\begin{proof} $(1)\implies(2)$: By \cite[Lemma~2.17 and Corollary~2.19]{paper58}.

$(2)\implies(3)$: This is trivial.

$(3)\implies(4)$: By Lemma~\ref{lemma2444}.

$(4)\implies(5)$: This is trivial.

$(5)\implies(1)$: Let $F\s\omega$ and $\vec C$ be as in Clause~(5). Fix a club $D\s\acc(\kappa)$
disjoint from $\{\Delta(\rho^{\vec C}_{0\alpha},\rho^{\vec C}_{0\beta})\mid \alpha<\beta<\kappa\ \&\ \rho_{0\alpha}\neq\rho_{0\beta}\restriction\alpha\}$.
Towards a contradiction, suppose that $\kappa$-$\stp$ holds. Then we may pick a pair $(\delta,\delta')\in[D]^2$ such that $C_\delta= C_{\delta'}\cap\delta$.
Put $\alpha:=\delta+1$ and $\beta:=\delta'+1$. Note:
\begin{itemize}
\item For every $\xi\le\max(F)$, as $C_\alpha\cap(\xi+1)=C_\beta\cap(\xi+1)$, it is the case that $\rho_0(\xi,\alpha)=\rho_0(\xi,\beta)$.
\item For every $\xi<\delta$ above $\max(F)$, as $C_{\delta'}\cap(\xi+1)=C_\delta\cap(\xi+1)$,
it is the case that
$$\rho_0(\xi,\beta)=\langle|F|\rangle{}^\smallfrown\rho_0(\xi,\delta')=\langle|F|\rangle{}^\smallfrown\rho_0(\xi,\delta)=\rho_0(\xi,\alpha).$$
\item $\rho_0(\delta,\alpha)=\langle|F|\rangle$.
\item $\rho_0(\delta,\beta)=\langle |F|,\otp(C_{\delta'}\cap\delta)\rangle$.
\end{itemize}
Thus $\Delta(\rho_{0\alpha},\rho_{0\beta})=\delta\in D$.
This is a contradiction.
\end{proof}

\section{\texorpdfstring{$T(\rho_1)$}{T(rho1)}}\label{sec6}

Lemma~6.2.1 of \cite{MR2355670} is stated in the context of walking along a $\mu$-bounded $C$-sequence over $\mu^+$ but remains true in the general case, as follows.
\begin{lemma}\label{lemma19} For every $i<\kappa$, for every $\delta<\kappa$, $|\{\alpha<\delta\mid \rho_{1\delta}(\alpha)\le i\}|\leq |i| +\aleph_0$.
\end{lemma}
\begin{proof} Suppose not, and let $i$ be a counterexample, and let $j$ be the least infinite cardinal above $|i|^+$. Let $\delta$ be the least ordinal less than $\kappa$ for which $A:=\{\alpha<\delta\mid \rho_{1\delta}(\alpha)\le i\}$ has size $\ge j$.
In particular, for every $\alpha\in A$, $\otp(C_\delta\cap\alpha)\le i$. Fix $i_0\le i$ for which $A_0:=\{\alpha<\delta\mid \otp(C_\delta\cap\alpha)=i_0\}$ has size $\ge j$.
It follows that $\alpha\mapsto\min(C_\delta\setminus\alpha)$ is constant over $A_0$ with value, say, $\gamma$.
Let $A_1:=A_0\setminus\{\gamma\}$, so that $A_1\s\gamma$ and $|A_1|\ge j$.
Next, by the minimality of $\delta$, and since $\gamma<\delta$, it is the case that $|\{\alpha<\gamma\mid \rho_{1\gamma}(\alpha)\le i\}|<j\le |A_1|$,
so we may pick some $\alpha\in A_1$ such that $\rho_{1\gamma}(\alpha)>i$.
As $\alpha<\gamma=\min(C_\delta\setminus\alpha)$, it is the case that $\rho_{1\delta}(\alpha)=\max\{\otp(C_\delta\cap\alpha),\rho_{1\gamma}(\alpha)\}$.
But $\otp(C_\delta\cap\alpha)=i_0\le i<\rho_{1\gamma}(\alpha)$,
which means that $\rho_{1\delta}(\alpha)>i$, contradicting the fact that $\alpha\in A_1\s A_0\s A$.
\end{proof}

Recalling Proposition~\ref{switchtofibers} and Remark~\ref{rmkrapid}, we thus get the following.

\begin{cor}\label{coro62}\label{lemma43b}
Suppose that $\rho_1$ is obtained by walking along a $\mu$-bounded $C$-sequence over $\kappa:=\mu^+$.
Then $\rho_1$ witnesses $\U(\kappa,\kappa,\mu,\mu)$.
In particular, for every $\zeta<\cf(\mu)$, for every rapid $S\s{}^\zeta T(\rho_1)$ of size $\kappa$,
there exists a rapid $X\s {}^\zeta\kappa$ of size $\kappa$ such that
for all $\langle \alpha_\iota\mid \iota<\zeta\rangle\neq \langle \beta_\iota\mid \iota<\zeta\rangle$ in $X$,
there are $x\neq y$ in $S$ such that for every $\iota<\zeta$:
\begin{itemize}
\item $x(\iota)\s \rho_{1\alpha_\iota}$ and $y(\iota)\s \rho_{1\beta_\iota}$;
\item $x(\iota)$ and $y(\iota)$ are incomparable;
\item $\dom(x(\iota))<\dom(y(\iota))$ iff $\alpha_\iota<\beta_\iota$. \qed
\end{itemize}
\end{cor}

The following is a straightforward generalization of \cite[Lemma~6.2.2]{MR2355670}.
\begin{lemma}\label{lemmaB}
$T(\rho_1)$ is $A'(\vec C)$-coherent.
\end{lemma}
\begin{proof} Let $\delta\in A'(\vec C)$. Towards a contradiction, suppose that there exists an $\alpha\ge\delta$ such that
$$\sup\{\xi<\delta\mid \rho_1(\xi,\delta)\neq \rho_1(\xi,\alpha)\}<\delta,$$
and let $\alpha$ be the least such. Clearly, $\alpha>\delta$.
As $\delta\in A'(\vec C)$, $\sup(C_\alpha\cap\delta)<\delta$, so that
the following set is cofinal in $\delta$:
$$X:=\{\xi\in(\sup(C_\alpha\cap\delta),\delta)\mid \rho_1(\xi,\delta)\neq \rho_1(\xi,\alpha)\}.$$
Set $\bar\alpha:=\min(C_\alpha\setminus\delta)$.
As $\delta\in A'(\vec C)$, $i:=\otp(C_\alpha\cap\delta)$ is smaller than $\cf(\delta)$ and then Lemma~\ref{lemma19} implies that
$X':=\{\xi\in X\mid \rho_1(\xi,\bar\alpha)>i\}$ is cofinal in $\delta$.
As $\bar\alpha\in[\delta,\alpha)$, the minimality of $\alpha$ implies that we may pick a $\xi\in X'$ such that $\rho_1(\xi,\delta)=\rho_1(\xi,\bar\alpha)$.
Now, $\rho_1(\xi,\alpha)=\max\{\otp(C_\alpha\cap\xi),\rho_1(\xi,\bar\alpha)\}$,
and $\otp(C_\alpha\cap\xi)=\otp(C_\alpha\cap\delta)=i<\rho_1(\xi,\bar\alpha)$,
so that $\rho_1(\xi,\alpha)=\rho_1(\xi,\bar\alpha)=\rho_1(\xi,\delta)$, contradicting the fact that $\xi\in X'\s X$.
\end{proof}

As mentioned in the introduction,
in his book \cite[Question~2.2.18]{MR2355670}, Todor\v{c}evi\'c asks what is a condition
one needs to put on a $C$-sequence $\vec C$ on $\omega_1$ in order to guarantee that the corresponding $T(\rho_1^{\vec C})$ is a special Aronszajn tree.\footnote{To compare,
conditions on $\vec C$ that make $T(\rho_1^{\vec C})$ nonspecial and almost-Souslin were known \cite{MR2194042}. See \cite[Theorem~B]{paper29} for an application at successors of singulars.}
Indeed, unlike $T(\rho^{\vec C}_0)$ and $T(\rho^{\vec C}_2)$ that are outright special for any $\vec C$ over $\omega_1$ that is $\omega$-bounded,
it was not known whether $\zfc$ proves the existence of a special $\aleph_1$-Aronszajn tree of the form $T(\rho_1^{\vec C})$.

The next lemma answers the question of Todor\v{c}evi\'c by providing a condition that applies to all successors of regulars, and the subsequent lemma shows that the requisite $C$-sequences can be constructed as well.
\begin{lemma}\label{lemma2} Suppose that $\mu$ is an infinite regular cardinal, and $\rho_1$ is obtained by walking along a weakly coherent $\mu$-bounded $C$-sequence $\vec C= \langle C_\delta \mid \delta <\mu^+\rangle$.
Suppose also that for some club $D\s\mu^+$, for every $\delta\in E^{\mu^+}_\mu\cap D$,
for every pair $\alpha<\beta$ of successive elements of $C_\delta$, $\otp(C_\beta\cap\alpha)>\otp(C_\delta\cap\alpha)+1$.
Then $T(\rho_1)$ is special.
\end{lemma}
\begin{proof}
As $\vec C$ is weakly coherent, Fact~\ref{lemma41} implies that $T(\rho_0)$ is a $\mu^+$-tree, and hence so is $T(\rho_1)$.
As $\vec C$ is $\mu$-bounded, $E^{\mu^+}_\mu\s A'(\vec C)$ and hence Lemma~\ref{lemmaB} implies that $T(\rho_1)$ is $E^{\mu^+}_\mu$-coherent.
Let $d:\{\rho_{1\delta}\mid\delta\in D\cap E^{\mu^+}_{\mu}\}\rightarrow\mu^+$ be the function defined via $d(\rho_{1\delta}):=\min(C_\delta)$.
By Lemma~\ref{lemma15}, it suffices to prove that the map $d$ is injective on chains.

To this end, let $\gamma<\delta$ be a pair of ordinals in $D\cap E^{\mu^+}_\mu$ such that $\rho_{1\gamma}\subsetneq\rho_{1\delta}$.
Let $\langle \gamma_i \mid i<\mu\rangle$ be the increasing enumeration of $C_\gamma$, and $\langle \delta_i \mid i< \mu\rangle$ that of $C_\delta$.
Towards a contradiction, suppose that $d(\rho_{1\gamma})=d(\rho_{1\delta})$, so that $\gamma_0=\delta_0$.
Let $k<\mu$ be the least ordinal such that $\gamma_k\neq \delta_k$.
Then $k>0$ and since $C_\gamma$ and $C_\delta$ are closed,
$k$ is a successor ordinal, say $k=j+1$.
Now, there are two cases:

$\br$ If $\gamma_k<\delta_k$, then
$$\gamma_0=\delta_0< \cdots <\gamma_j=\delta_j< \gamma_k<\delta_k<\delta,$$
and hence
$$\begin{aligned}\rho_{1\delta}(\gamma_k)&=\max\{\otp(C_\delta\cap \gamma_k),\rho_{1\delta_k}(\gamma_k)\}\\
&\ge\rho_{1\delta_k}(\gamma_k)\\
&\ge\otp(C_{\delta_k}\cap\gamma_k)\\
&\ge\otp(C_{\delta_k}\cap\delta_j)\\
&>\otp(C_{\delta}\cap\delta_j)+1\\
&=j+1=k=\rho_{1\gamma}(\gamma_k).
\end{aligned}$$
This is a contradiction.

$\br$ If $\delta_k<\gamma_k$, then a similar reasoning yields a contradiction.
\end{proof}

\begin{lemma}\label{lemma66}
Suppose that $\mu$ is an infinite regular cardinal, and there exists a special $\mu^+$-Aronszajn tree.
Then there exists a weakly coherent $\mu$-bounded $C$-sequence $\vec C= \langle C_\delta \mid \delta <\mu^+\rangle$
and a club $D\s\mu^+$ such that for every $\delta\in D$,
for every pair $\alpha<\beta$ of successive elements of $C_\delta$, $\otp(C_\beta\cap\alpha)>\otp(C_\delta\cap\alpha)+1$.
In particular, $T(\rho_1)$ is special.
\end{lemma}
\begin{proof} Let $Z$ be a nonstationary cofinal subset of $E^{\mu^+}_\mu$, e.g., $Z:=\nacc(E^{\mu^+}_\mu)$.
Let $\langle Z_i\mid i<\mu\rangle$ be a partition of $Z$ into cofinal sets.
Let $C$ be a subclub of $\bigcap_{i<\mu}\acc^+(Z_i)$ that is disjoint from $Z$.
Denote $D:=\acc(C)$.
\begin{claim} There exists a weakly coherent $\mu$-bounded $C$-sequence $\vec C=\langle C_\delta\mid\delta<\mu^+\rangle$
such that $C_\delta\s C$ for every $\delta\in D$.
\end{claim}
\begin{why} If $\mu=\omega$, then let $\langle C_\delta\mid\delta<\omega_1\rangle$ be any $\omega$-bounded ladder system
such that $C_\delta\s C$ for every $\delta\in D$.

Next, suppose that $\mu>\omega$.
Using \cite[Theorem~6.1.14]{MR2355670},
fix a weakly coherent $\mu$-bounded $C$-sequence $\langle c_\delta\mid\delta<\mu^+\rangle$.
Define $\langle C_\delta\mid\delta<\mu^+\rangle$ via
$$C_\delta:=\begin{cases}
\{ \min(C\setminus\gamma)\mid \gamma\in c_\delta\},&\text{if }\delta\in D;\\
c_\delta,&\text{otherwise}.
\end{cases}$$
A standard verification (cf.~\cite[Lemma~2.8]{paper32}) establishes that $\langle C_\delta\mid\delta<\mu^+\rangle$ is as sought.
\end{why}

Fix a function $f:\mu\times\mu^+\times\mu^+\rightarrow\mu^+$ such that for every $j<\mu$ and for every pair $\beta<\gamma$ of ordinals in $\mu^+$,
$$f(j,\beta,\gamma):=\begin{cases}
\min(Z_i\cap[\beta,\gamma)),&\text{if }j= i+1\text{ and }Z_i\cap[\beta,\gamma)\neq\emptyset;\\
\min(Z_j\cap[\beta,\gamma)),&\text{if }j=\sup(j)\text{ and }Z_j\cap[\beta,\gamma)\neq\emptyset;\\
\beta,&\text{otherwise},
\end{cases}$$
and note that $\beta\le f(j,\beta,\gamma)<\gamma$.
Let $\vec C=\langle C_\delta\mid\delta<\mu^+\rangle$ be given by the claim,
and then, for every $\delta<\mu^+$, set
$$C_\delta^\bullet:=\begin{cases}
\{ f(\otp(C_\delta\cap\gamma),\sup(C_\delta\cap\gamma),\gamma)\mid \gamma\in C_\delta\},&\text{if }\delta\in D;\\
(i+2)\cup C_\delta,&\text{if }\delta\in Z_i;\\
C_\delta,&\text{otherwise}.
\end{cases}$$

It is clear that for every $\delta\in D$, $\sup(C_\delta^\bullet)=\sup(C_\delta)=\delta$, $\otp(C_\delta^\bullet)=\otp(C_\delta)$ and $\acc(C_\delta^\bullet)=\acc(C_\delta)$.
Thus, altogether $\vec C^\bullet:=\langle C_\delta^\bullet\mid \delta<\mu^+\rangle$ is a $\mu$-bounded $C$-sequence.
\begin{claim} $\vec C^\bullet$ is weakly coherent.
\end{claim}
\begin{why} Suppose not. As $\vec C$ is weakly coherent, this must mean that there exists some $\epsilon<\mu^+$ such that
$\{ C_\delta^\bullet\cap\epsilon\mid \delta\in D\}$ has size $\mu^+$.
By passing to the least such $\epsilon$, we moreover get that
$\{ C_\delta^\bullet\cap\epsilon\mid \delta\in D, \sup(C_\delta^\bullet\cap\epsilon)=\epsilon\}$ has size $\mu^+$.
Thus, fix a cofinal $\Delta\s D\setminus(\epsilon+1)$ such that:
\begin{itemize}
\item $\delta\mapsto C_\delta^\bullet\cap\epsilon$ is injective over $\Delta$;
\item $\sup(C_\delta^\bullet\cap\epsilon)=\epsilon$ for every $\delta\in\Delta$;
\item $\delta\mapsto C_\delta\cap\epsilon$ is constant over $\Delta$, with value say $c$.
\end{itemize}

As $\acc(C_\delta^\bullet)=\acc(C_\delta)$ for every $\delta\in D$, it follows that $\sup(c)=\epsilon$.
But then, for every $\delta\in\Delta$, $C_\delta^\bullet\cap\epsilon=\{f(\otp(c\cap\gamma),\sup(c\cap\gamma),\gamma)\mid \gamma\in c\}$,
contradicting the first bullet.
\end{why}
\begin{claim} Let $\delta\in D$.
Suppose $\alpha<\beta$ are successive elements of $C^\bullet_\delta$. Then $\otp(C^\bullet_\beta\cap\alpha)>\otp(C^\bullet_\delta\cap\alpha)+1$.
\end{claim}
\begin{why}
Let $i:=\otp(C_\delta^\bullet\cap\alpha)$, and note that $\otp(C_\delta^\bullet\cap\beta)=i+1$.
Let $\gamma$ be the unique element of $C_\delta$ such that $\otp(C_\delta\cap\gamma)=i+1$.
Then $\beta=f(i+1,\sup(C_\delta\cap\gamma),\gamma)$.
Since $\delta\in D$, it is the case that $C_\delta\s C\s \acc^+(Z_i)$,
and hence $\beta\in Z_i$. So $i+2\s C_\beta^\bullet$.
Let $\bar\gamma$ be the unique element of $C_\delta$ such that $\otp(C_\delta\cap\bar\gamma)=i$.
Then $\alpha=f(i,\sup(C_\delta\cap\bar\gamma),\bar\gamma)$ and hence either $\alpha=\bar\gamma\in C_\delta\s C\s\mu^+\setminus\mu$
or $\alpha\in\bigcup_{j\le i}Z_j\s E^{\mu^+}_\mu$, so that $\alpha\ge\mu$.
Altogether, $\otp(C_\beta^\bullet\cap\alpha)\ge \otp(C_\beta^\bullet\cap\mu)\ge i+2>i+1=\otp(C_\delta^\bullet\cap\alpha)+1$.
\end{why}
By Lemma~\ref{lemma2}, $T(\rho_1^{\vec C^\bullet})$ is a special $\mu^+$-Aronszajn tree.
\end{proof}

We now arrive at Theorem~\ref{thmd}.

\begin{cor}\label{cor66}
Suppose that $\mu$ is an infinite regular cardinal. If there exists a special $\mu^+$-Aronszajn tree,
then there exists a $C$-sequence $\vec C$ over $\mu^+$ such that $T(\rho_0^{\vec C}),T(\rho_1^{\vec C}),T(\rho_2^{\vec C})$
are all special $\mu^+$-Aronszajn trees.
\end{cor}
\begin{proof} Using Lemma~\ref{lemma66},
fix a weakly coherent $\mu$-bounded $C$-sequence $\vec C$ over $\mu^+$ such that $T(\rho^{\vec C}_1)$ is a special $\mu^+$-Aronszajn tree.
By \cite[Theorem~6.1.14]{MR2355670},
since $\vec C$ is a weakly coherent and $\mu$-bounded $C$-sequence, $T(\rho^{\vec C}_0)$ is as well a special $\mu^+$-Aronszajn tree.\footnote{This also follows from Lemma~\ref{lemma511} and the upcoming Lemma~\ref{lemma57}.\label{footnotecor66}}
Then, by Lemma~\ref{lemma511}, $T(\rho^{\vec C}_2)$ is a special $\mu^+$-Aronszajn tree, as well.
\end{proof}

\section{\texorpdfstring{$T(\rho_2)$}{T(rho2)}}\label{sec7}

\begin{lemma}
For all $\alpha<\beta<\kappa$ such that $\rho_{2\alpha}$ and $\rho_{2\beta}$ are incomparable,
$\Delta(\rho_{2\alpha},\rho_{2\beta})\notin V(\vec C)$.
\end{lemma}
\begin{proof} Just run the proof of Lemma~\ref{lemma2444} with $h$ a constant map.
\end{proof}

Likewise, the proof of Lemma~\ref{l54}
yields the following generalization of \cite[Theorem~4.27(1)]{peng}.

\begin{lemma}
If $V(\vec C)$ covers a club, then there exists an injection $\psi:T(\rho_2)\rightarrow{}^{<\kappa}\mathbb Z$ such that:
\begin{itemize}
\item for every $t\in T(\rho_2)$, $\dom(\psi(t))=\dom(t)$;
\item for all $s,t\in T(\rho_2)$, $s\s t$ iff $\psi(s)\s\psi(t)$;
\item for all $s,t\in T(\rho_2)$, $s<_{\lex}t$ iff $\psi(s)<_{\lex}\psi(t)$;
\item for all $s,t\in T(\rho_2)$, $\Delta(s,t)=\Delta(\psi(s),\psi(t))$;
\item for some club $D\s\kappa$, the image of $\psi$ is an $(A(\vec C)\cap D)$-coherent streamlined tree.\qed
\end{itemize}
\end{lemma}

Our next lemma generalises \cite[Lemma~6.7]{MR3620068}.

\begin{lemma}\label{lemma57}
Suppose $\rho_2$ is obtained by walking along a $C$-sequence $\vec C$ over $\kappa$.
Then there exists a function $c:T(\rho_2)\restriction(R(\vec C)\cap D)\rightarrow\kappa$ such that:
\begin{itemize}
\item $D$ is a club in $\kappa$;
\item $c(x)<\dom(x)$ for every $x\in\dom(c)$;
\item for every pair $x\subsetneq y$ of nodes in $\dom(c)$, $c(x)\neq c(y)$.
\end{itemize}
\end{lemma}
\begin{proof} Write $\vec C=\langle C_\delta \mid \delta <\kappa\rangle$.
Fix a bijection $\pi:\kappa^4\leftrightarrow\kappa$,
and let $D:=\{\gamma\in\acc(\kappa)\mid \pi[\gamma^4]=\gamma\}$.
We now define a colouring $c:T(\rho_2)\restriction (R(\vec C)\cap D)\rightarrow\kappa$, as follows.
Given $x\in T(\rho_2)\restriction (R(\vec C)\cap D)$, set $\gamma:=\dom(x)$ and fix the least $\gamma'\in[\gamma,\kappa)$ such that $x=\rho_{2\gamma'}\restriction\gamma$.
There are three main possibilities:
\begin{description}
\item[Case 0] If $\gamma'=\gamma$, then let
$$c(x):=\begin{cases}\pi(0,0,0,\otp(C_\gamma)),&\text{if }\otp(C_\gamma)<\gamma;\\
\pi(0,0,0,0),&\text{otherwise}.
\end{cases}$$
\item[Case 1] If $\gamma'>\gamma$ and $\lambda(\gamma,\gamma')<\gamma$, then denote $\alpha:=\lambda(\gamma,\gamma')+1$ and let:
$$c(x):=\begin{cases}
\pi(1,\rho_2(\gamma,\gamma'),\alpha,\otp(C_\gamma\cap[\alpha,\gamma))),&\text{if }\otp(C_\gamma\cap[\alpha,\gamma))<\gamma;\\
\pi(1,\rho_2(\gamma,\gamma'),\alpha,0),&\text{otherwise}.
\end{cases}$$
\item[Case 2] If $\gamma'>\gamma$ and $\lambda(\gamma,\gamma')=\gamma$, then denote $\alpha:=\lambda_2(\gamma,\gamma')+1$ and let:
$$c(x):=\pi(2,\rho_2(\gamma,\gamma'),\alpha,\otp(C_{\last{\gamma}{\gamma'}}\cap[\alpha,\gamma))).$$
\end{description}

Note that in Case~2 we indeed have $c(x)<\gamma$, since $\gamma\in R(\vec C)$.
Towards a contradiction, suppose that $x\subsetneq y$ is a pair of nodes from $\dom(c)$
such that $c(x)=c(y)$, say, it is equal to $\pi(i,n,\alpha,\epsilon)$.
There are three main possibilities:
\begin{itemize}
\item[$\br$] If $i=0$, then $\gamma'=\gamma$, $\delta'=\delta$ and $\rho_{2\gamma}=x\restriction\gamma=y\restriction\gamma=\rho_{2\delta}\restriction\gamma$.
Consequently,
$$C_\gamma=\{ \beta<\gamma\mid \rho_{2\gamma}(\beta)=1\}=\{ \beta<\gamma\mid \rho_{2\delta}(\beta)=1\}=C_\delta\cap\gamma.$$
Now, there are two options:
\begin{itemize}
\item[$\br\br$] If $\epsilon=0$, then $\otp(C_\delta\cap\gamma)=\otp(C_\gamma)=\gamma$, contradicting the fact that $\gamma\in R(\vec C)$.

\item[$\br\br$] If $\epsilon>0$, then $\epsilon=\otp(C_\gamma)=\otp(C_\delta\cap\gamma)<\otp(C_\delta)=\epsilon$,
where the strict inequality comes from the fact that $\sup(C_\delta)=\delta$.
This is again a contradiction.
\end{itemize}
\item[$\br$] If $i=1$, then $\alpha$ lies in between $\lambda(\gamma,\gamma')$ and $\gamma$,
so for every $\beta\in[\alpha,\gamma)$, $$x(\beta)=\rho_{2\gamma'}(\beta)=\rho_{2\gamma'}(\gamma)+\rho_{2\gamma}(\beta)=n+\rho_{2\gamma}(\beta).$$
Likewise, for every $\beta\in[\alpha,\delta)$, $y(\beta)=n+\rho_{2\delta}(\beta)$.
So from $x=y\restriction\gamma$, it follows that $\rho_{2\gamma}\restriction [\alpha,\gamma)=\rho_{2\delta}\restriction [\alpha,\gamma)$.
But then
$$\begin{aligned}C_\gamma\cap[\alpha,\gamma)=&\ \{ \beta\in[\alpha,\gamma)\mid \rho_{2\gamma}(\beta)=1\}\\
=&\ \{ \beta\in[\alpha,\gamma)\mid \rho_{2\delta}(\beta)=1\}=C_\delta\cap[\alpha,\gamma).\end{aligned}$$
Now, there are two options:
\begin{itemize}
\item[$\br\br$] If $\epsilon=0$, then $\otp(C_\delta\cap\gamma)\ge\otp(C_\gamma\cap[\alpha,\gamma))=\gamma$, contradicting the fact that $\gamma\in R(\vec C)$.

\item[$\br\br$] If $\epsilon>0$, then we arrive at the following contradiction:
$$\epsilon=\otp(C_\gamma\cap[\alpha,\gamma))=\otp(C_\delta\cap[\alpha,\gamma))<\otp(C_\delta\cap[\alpha,\delta))=\epsilon.$$
\end{itemize}
\item[$\br$] If $i=2$, then $\alpha$ lies in between $\lambda_2(\gamma,\gamma')$ and $\gamma$,
so for every $\beta\in[\alpha,\gamma)$,
$$x(\beta)=\rho_{2\gamma'}(\beta)=\rho_2(\gamma,\gamma')-1+\rho_2(\beta,\last{\gamma}{\gamma'})=n-1+\rho_2(\beta,\last{\gamma}{\gamma'}).$$
Likewise, for every $\beta\in[\alpha,\delta)$, $y(\beta)=n-1+\rho_2(\beta,\last{\delta}{\delta'})$.
So from $x=y\restriction\gamma$, it follows that
$$\begin{aligned}C_{\last{\gamma}{\gamma'}}\cap[\alpha,\gamma)=&\ \{ \beta\in[\alpha,\gamma)\mid \rho_{2\gamma'}(\beta)=n\}\\
=&\ \{ \beta\in[\alpha,\gamma)\mid \rho_{2\delta'}(\beta)=n\}=C_{\last{\delta}{\delta'}}\cap[\alpha,\gamma),\end{aligned}$$
which means that
$$\begin{aligned}\epsilon=\otp(C_{\last{\gamma}{\gamma'}}\cap[\alpha,\gamma))=&\ \otp(C_{\last{\delta}{\delta'}}\cap[\alpha,\gamma))\\
<&\ \otp(C_{\last{\delta}{\delta'}}\cap[\alpha,\delta))=\epsilon,\end{aligned}$$
where the strict inequality comes from the fact that $\sup(C_{\last{\delta}{\delta'}}\cap\delta)=\delta$. This is a contradiction.\qedhere
\end{itemize}
\end{proof}

The upcoming four corollaries were known to Todor\v{c}evi\'c.

\begin{cor}\label{cor74}
Suppose that $\rho_2$ is obtained by walking along a $\square_{<}^*(\kappa)$-sequence.
Then $T(\rho_2)$ is a special $\kappa$-Aronszajn tree.
\end{cor}
\begin{proof} By Fact~\ref{lemma41}, $T(\rho_2)$ is a $\kappa$-tree.
The rest follows from Lemma~\ref{lemma57} and Lemma~\ref{fact13}.
\end{proof}
\begin{remark} Regressiveness of $\vec C$ is not necessary for the corresponding $T(\rho_2)$ to be special.
\end{remark}

\begin{cor} Suppose $\rho_2$ is obtained by walking along a $\square^*_\mu$-sequence.
Then $T(\rho_2)$ is a special $\mu^+$-Aronszajn tree. \qed
\end{cor}
\begin{cor} Suppose $\rho_2$ is obtained by walking along a $\square_\mu$-sequence.
Then $(T(\rho_2),<_{\lex})$ is a $\mu^+$-Countryman line. \qed
\end{cor}
\begin{cor}[Peng, {\cite[Theorem~4.27(2)]{peng}}] Suppose $\rho_2$ is obtained by walking along an $\omega$-bounded $C$-sequence over $\omega_1$. Then $T(\rho_2)$ is a special $\aleph_1$-Aronszajn tree
and $(T(\rho_2),<_{\lex})$ is an $\aleph_1$-Countryman line. \qed
\end{cor}

By \cite[Lemma~4.12]{paper35}, the following generalises \cite[Lemma~4.5]{Lucke}.
\begin{lemma}
Suppose $\rho_2$ is obtained by walking along a $C$-sequence $\vec C$ over $\kappa$.
Then $T(\rho_2)$ does not admit an ascending path of width less than $\chi(\vec C)$.
\end{lemma}
\begin{proof} Towards a contradiction, suppose that $\vec f=\langle f_\gamma\mid\gamma<\kappa\rangle$ is a $\theta$-ascending path through $T(\rho_2)$ with $\theta<\chi(\vec C)$.
For all $\gamma<\kappa$ and $i<\theta$, fix $\gamma_i\in[\gamma,\kappa)$ such that $f_\gamma(i)=\rho_{2\gamma_i}\restriction\gamma$.
As $\chi(\vec C)\le\sup(\reg(\kappa))$, the set $E^\kappa_{>\theta}$ is stationary.
Thus, recalling Fact~\ref{last}, we may fix $\alpha<\beta<\kappa$ for which the following set is stationary:
$$\Gamma:=\{\gamma\in E^\kappa_{>\theta}\mid \sup\{\lambda_2(\gamma,\gamma_i)\mid i<\theta\}=\alpha\ \&\ \beta\in\bigcap\nolimits_{i<\theta}C_{\last{\gamma}{\gamma_i}}\}.$$
\begin{claim} Let $\epsilon<\kappa$. There exists a $\Delta\in[\kappa]^{\le\theta}$ such that $\Gamma\cap\epsilon\s\bigcup_{\delta\in\Delta}C_\delta$.
\end{claim}
\begin{why} Fix $\delta\in\Gamma$ above $\epsilon$ and set $\Delta:=\{\last{\delta}{\delta_i}\mid i<\theta\}$.
Let $\gamma\in\Gamma\cap\epsilon$ be arbitrary, and we shall show that $\gamma\in\bigcup_{\eth\in\Delta}C_\eth$.

First, as $\vec f$ is a $\theta$-ascending path, we may fix $i,j<\theta$ such that $f_\gamma(i)\s f_{\delta}(j)$.
For every $\xi\in(\alpha,\gamma)$,
we have that:
\begin{itemize}
\item $f_\gamma(i)(\xi)=\rho_2(\xi,\gamma_i)=\rho_2(\last{\gamma}{\gamma_i},\gamma_i)+\rho_2(\xi,\last{\gamma}{\gamma_i})$, and
\item $f_{\delta}(j)(\xi)=\rho_2(\xi,\delta_{j})=\rho_2(\last{\delta}{\delta_{j}},\delta_{j})+\rho_2(\xi,\last{\delta}{\delta_{j}})$.
\end{itemize}
As $\beta\in C_{\last{\gamma}{\gamma_i}}\cap C_{\last{\delta}{\delta_{j}}}\cap(\alpha,\gamma)$
and $f_\gamma(i)(\beta)=f_{\delta}(j)(\beta)$, we infer that $\rho_2(\last{\gamma}{\gamma_i},\gamma_i)=\rho_2(\last{\delta}{\delta_{j}},\delta_{j})$.
Consequently, for every $\xi\in C_{\last{{\gamma}}{\gamma_i}}\cap(\alpha,{\gamma})$,
$$\rho_2(\xi,\delta)=f_{\delta}(j)(\xi)=f_\gamma(i)(\xi)=\rho_2(\last{\gamma}{\gamma_i},\gamma_i)+1=\rho_2(\last{\delta}{\delta_{j}},\delta_{j})+1,$$
which means that $\xi\in C_{\last{\delta}{\delta_{j}}}$. So $\gamma$ is an accumulation point of the club $C_{\last{\delta}{\delta_{j}}}$.
Recalling the definition of $\Delta$, this shows that $\gamma\in\bigcup_{\eth\in\Delta}C_\eth$.
\end{why}
It follows that $\Gamma$ witnesses $\chi(\vec C)\le\theta$. This is a contradiction.
\end{proof}
\begin{remark} It follows that in \cite[Theorem~3.4]{paper35}, the conjunction of Clauses (1),(4) and (5) implies Clause~(2).
\end{remark}

\section{Evading a Countryman line the easy way}\label{easysec}

The following is a straightforward generalization of \cite[Lemma~3.3]{MR2444284}.

\begin{lemma}\label{lemma42} Suppose that $\vec C$ is a weakly coherent $C$-sequence over $\kappa$,
and $\nu$ is a (possibly finite) nonzero cardinal such that $\mu^{\nu}<\kappa$ for all $\mu<\kappa$.
For every rapid $X\s {}^\nu\kappa$ of size $\kappa$,
there exists a club $D_X$ in $\kappa$ such that, for every $\gamma\in D_X\cap R(\vec C)\cap E^\kappa_{>\nu}$,
for every matrix $\langle \beta^k_i\mid i<m, k<\nu\rangle$
such that $i\mapsto \langle \beta^k_i\mid k<\nu\rangle$ is an injection from a positive integer $m$ to $\{x\in X\mid \min(x)\ge\gamma\}$,
for every ordinal $\epsilon$ with $\sup_{i<m,k<\nu}\lambda(\gamma,\beta^k_i)\le \epsilon<\gamma$,
there is a $\zeta\in(\epsilon,\gamma)$ and a matrix $\langle \alpha^k_i\mid i<m, k<\nu\rangle$ such that:
\begin{itemize}
\item $i\mapsto \langle \alpha^k_i\mid k<\nu\rangle$ is an injection from $m$ to $\{x\in X\mid \sup(x)<\gamma\}$;
\item for all $i<m$ and $k<\nu$, $\Delta(\rho_{0\alpha^k_i},\rho_{0\beta^k_i})=\zeta$ and $\rho_{0\alpha^k_i}<_\lex\rho_{0\beta^k_i}$.
\end{itemize}
\end{lemma}
\begin{proof} Let a rapid $X\s {}^\nu\kappa$ of size $\kappa$ be given.
If $R(\vec C)$ is nonstationary, then let $D_X$ be a club disjoint from it, and we are done.

Hereafter, suppose that $R(\vec C)$ is stationary, so that $\vec C$ is nontrivial.
By Facts \ref{lemma41} and \ref{nontrivialfact}, then, $T:=T(\rho_0^{\vec C})$ is a $\kappa$-Aronszajn tree.
Fix a continuous $\in$-chain $\langle M_\gamma\mid\gamma<\kappa\rangle$ of elementary submodels of $H_{\kappa^+}$,
each of size less than $\kappa$ such that $\{\vec C,X\}\in M_0$ and ${}^\nu M_\gamma\s M_{\gamma+1}$ for every $\gamma<\kappa$.
Let $D_X:=\{\gamma\in\acc(\kappa)\mid M_\gamma\cap\kappa=\gamma\}$.
Note that ${}^\nu M_\gamma\s M_{\gamma}$ for every $\gamma\in D_X\cap E^\kappa_{>\nu}$.

Now, given $\gamma$, $\langle \beta^k_i\mid i<m, k<\nu\rangle$ and $\epsilon$ as in the statement of the lemma, set
\begin{itemize}
\item $\sigma^k_i:=\rho_0(\gamma,\beta^k_i)$ for all $i<m$ and $k<\nu$;
\item $t^k_i:=\rho_{0\beta^k_i}\restriction(\epsilon+1)$ for all $i<m$ and $k<\nu$;
\item $t:= \rho_{0\gamma}\restriction(\epsilon+1)$.
\end{itemize}

Since $T$ is a $\kappa$-tree, $T\restriction\gamma\s M_\gamma$,
so that $t$ and the matrix $\langle t^k_i\mid i<m, k<\nu\rangle$ are both in $M_\gamma$.
Since $\gamma\in R(\vec C)$, $\langle \sigma^k_i\mid i<m, k<\nu\rangle$ is in $M_\gamma$, as well.

Altogether, the set $P$ of all pairs $(\bar\gamma,\langle \alpha^k_i\mid i<m, k<\nu\rangle)$ satisfying all of the following is in $M_\gamma$ and has $(\gamma, \langle \beta^k_i \mid i< m, k< \nu\rangle)$ as an element:
\begin{enumerate}
\item $\bar\gamma\in \acc(\kappa\setminus\epsilon)$;
\item $\langle \alpha^k_i\mid i<m, k<\nu\rangle$ is a matrix
such that $i\mapsto \langle \alpha^k_i\mid k<\nu\rangle$ is an injection from $m$ to $\{x\in X\mid \min(x)\ge\bar\gamma\}$;
\item $\sup_{i<m,k<\nu}\lambda(\bar\gamma,\alpha^k_i)\le \epsilon<\bar\gamma$;
\item for all $i<m$ and $k<\nu$, $t^k_i\sq \rho_{0\alpha^k_i}$ and $\rho_0(\bar\gamma,\alpha^k_i)=\sigma^k_i$;
\item $\rho_{0\bar\gamma} \restriction (\epsilon+1) = t$.
\end{enumerate}
\begin{claim}\label{c811} There is a pair $(\bar\gamma,\langle \alpha^k_i\mid i<m, k<\nu\rangle)\in P\cap M_\gamma$ such that $\rho_{0\bar\gamma}<_{\lex}\rho_{0\gamma}$.
\end{claim}
\begin{why} Consider $\Gamma:=\{\bar\gamma\mid (\bar\gamma,\langle \alpha^k_i\mid i<m, k<\nu\rangle)\in P\}$ which is a subset of $\kappa$
lying in $M_\gamma$. It suffices to find a $\bar\gamma\in \Gamma\cap M_\gamma$ such that $\rho_{0\bar\gamma}<_{\lex}\rho_{0\gamma}$.
If we cannot, then $M_\gamma$ satisfies that for every $\varepsilon<\kappa$, there is some $\gamma^*\in\Gamma$ such that
$\rho_{0\gamma^*}<_{\lex}\rho_{0\bar\gamma}$ for every $\bar\gamma\in\Gamma\cap\varepsilon$.
Going outside the model, this allows to embed $(\kappa,{\ni})$ into $(T,<_{\lex})$,
contradicting Fact~\ref{flippedorder}.
\end{why}

Pick a pair $(\bar\gamma,\langle \alpha^k_i\mid i<m, k<\nu\rangle)\in P\cap M_\gamma$ as in the claim.
Let $i<m$ and $k<\nu$. Then:
\begin{enumerate}[(1)]
\item $t^k_i\sq \rho_{0\alpha^k_i}$ and $t^k_i\sq \rho_{0\beta^k_i}$;
\item for every $\xi\in(\epsilon,\bar\gamma]$, $\rho_0(\xi,\beta^k_i)=\rho_0(\gamma,\beta^k_i){}^\smallfrown \rho_0(\xi,\gamma)=\sigma^k_i{}^\smallfrown \rho_0(\xi,\gamma)$;
\item for every $\xi\in(\epsilon,\bar\gamma]$, $\rho_0(\xi,\alpha^k_i)=\rho_0(\bar\gamma,\alpha^k_i){}^\smallfrown \rho_0(\xi,\bar\gamma)=\sigma^k_i{}^\smallfrown \rho_0(\xi,\bar\gamma)$.
\end{enumerate}
In addition:
\begin{itemize}
\item[(4)] $\rho_{0\bar\gamma}\restriction(\epsilon+1)=t=\rho_{0\gamma}\restriction(\epsilon+1)$.
\end{itemize}

As $\rho_{0\bar\gamma}<_{\lex}\rho_{0\gamma}$, the two are incomparable, so we consider $\xi:=\Delta(\rho_{0\bar\gamma},\rho_{0\gamma})$.
By Clause~(4), $\epsilon<\xi<\bar\gamma$.
Let $i<m$ and $k<\nu$.
By Clause~(1), $\Delta(\rho_{0\alpha^k_i},\rho_{0\beta^k_i})>\epsilon$.
Then, by Clauses (2) and (3), $\Delta(\rho_{0\alpha^k_i},\rho_{0\beta^k_i})=\Delta(\rho_{0\bar\gamma},\rho_{0\gamma})=\xi$.
It thus follows that
$$\rho_{0\alpha^k_i}<_\lex\rho_{0\beta^k_i}\iff\sigma^k_i{}^\smallfrown \rho_0(\xi,\bar\gamma)<_\lex\sigma^k_i{}^\smallfrown \rho_0(\xi,\gamma)\iff\rho_{0\bar\gamma}<_\lex\rho_{0\gamma},$$
as sought.
\end{proof}

\begin{thm}\label{easyway} Suppose that:
\begin{itemize}
\item $\kappa=\mu^+$ for some cardinal $\mu=\mu^{\aleph_1}$;
\item $\vec C$ is a $\square^*_\mu$-sequence such that $V(\vec C)$ covers a club;
\item $A(\vec C)\cap E^{\kappa}_{>\omega_1}$ is stationary and $1\in\Theta_1^{\omega_1}(C\restriction A(\vec C),\kappa)$.
\end{itemize}

Then there is a special $\kappa$-Aronszajn line with no $\kappa$-Countryman subline.
\end{thm}
\begin{proof} Towards a contradiction, suppose that every special $\kappa$-Aronszajn line contains a $\kappa$-Countryman line.
Write $\langle C_\delta\mid\delta<\kappa\rangle$ for $\vec C$.
Fix a club $E$ in $\kappa$ such that $E\s V(\vec C)$, and
fix a colouring $g:\kappa\rightarrow\omega_1$, witnessing that $1\in\Theta_1^{\omega_1}(C\restriction A(\vec C),\kappa)$.

To every function $t\in{}^{<\kappa}({}^{<\omega}\kappa)$, we attach a finite sequence $t^*\in{}^{<\omega}\kappa$ via
$$t^*:=\begin{cases}
t(\sup(E \cap \dom(t))),&\text{if }\sup(E\cap\dom(t))\in\dom(t);\\
\emptyset,&\text{otherwise}.
\end{cases}$$

Using \cite[Proposition~7.7]{paper47}, fix a colouring $c:[\omega_1]^2\rightarrow2$ witnessing $\onto(\{\omega\},J^{\bd}[\omega_1],2)$.
For every $n<\omega$, define a colouring $c_n:T(\rho_0^{\vec C})\rightarrow2$ via $c_n(t):=c(n,(t^*)_g)$.\footnote{Recall that by Subsection~\ref{nandc}, $(t^*)_g$ stands for $\sup\{ g(t^*(i))\mid i\in\dom(t^*)\}$.}
Similarly to Fact~\ref{flippedorder}, for every $n<\omega$, we define an ordering $<_n$ of $T(\rho_0^{\vec C})$
by letting for all $s,t\in T(\rho_0^{\vec C})$:
\begin{itemize}
\item[$\br$] If $c_n(s\wedge t)=0$, then $s<_n t$ iff $s<_{\lex}t$;
\item[$\br$] If $c_n(s\wedge t)\neq0$, then $s<_n t$ iff $t<_{\lex}s$.
\end{itemize}

As $\vec C$ is a $\square^*_\mu$-sequence, by \cite[Theorem~6.1.14]{MR2355670} (or see Footnote~\ref{footnotecor66}) $(T(\rho_0^{\vec C}),<_{\lex})$ is a special $\kappa$-Aronszajn line,
and hence we may fix some $T\in[T(\rho_0^{\vec C})]^\kappa$ such that $(T,{<_{\lex}})$ is a $\kappa$-Countryman line.
For every $n<\omega$, $(T,<_n)$ is a special $\kappa$-Aronszajn line,
so we may fix some $T^n\in[T]^\kappa$ such that $(T^n,<_n)$ is a $\kappa$-Countryman line.
Let $d:T\times T\rightarrow\mu$ be a decomposition of $(T,{<_{\lex}})$ into chains,
and for every $n<\omega$, let $d_n:T^n\times T^n\rightarrow\mu$ be a decomposition of $(T^n,<_{n})$ into chains.

For each $\upsilon<\kappa$, fix a sequence $\langle (t_{\upsilon}^n,\beta_{\upsilon}^n)\mid n<\omega\rangle$ in $\prod_{n < \omega}(T^n\times \kappa)$ such that
for every $n<\omega$, $\dom(t_{\upsilon}^n)>\upsilon$ and $t_{\upsilon}^n\s \rho_{0\beta_{\upsilon}^n}$.
Fix $\lambda^*<\kappa$, $\nu^*<\mu$, $j^*<\omega_1$,
and a stationary $\Upsilon\s A(\vec C)\cap E^\kappa_{>\omega}$ such that for every $\upsilon\in\Upsilon$:
\begin{itemize}
\item $\lambda^*=\sup\{ \lambda(\upsilon,\beta_{\upsilon}^n)\mid n<\omega\}$;
\item $\nu^*=\sup\{\rho_1(\upsilon,\beta_{\upsilon}^n)\mid n<\omega\}$;
\item $j^*=\sup\{ (\rho_0(\upsilon,\beta_{\upsilon}^n))_g\mid n<\omega\}$;
\item for every $\bar\upsilon\in\Upsilon\cap\upsilon$, $\sup\{\beta_{\bar\upsilon}^n\mid n<\omega\}<\upsilon$.
\end{itemize}

By Corollary~\ref{coro62}, $\rho_1$ witnesses $\U(\kappa,\kappa,\mu,\omega_1)$,
so we may fix a cofinal $\Upsilon'\s\Upsilon$ such that, for every pair $\bar\upsilon<\upsilon$ of ordinals in $\Upsilon'$,
$\rho_1(\bar\upsilon,\beta_{\upsilon}^n)>\nu^*$ for every $n<\omega$.
In particular, for every pair $\bar\upsilon<\upsilon$ of ordinals in $\Upsilon'$,
$$t_{\beta_{\bar\upsilon}^n}\wedge t_{\beta_{\upsilon}^n}=\rho_{0\beta_{\bar\upsilon}^n}\wedge \rho_{0\beta_{\upsilon}^n}.$$

By Lemma~\ref{lemma35a}, for every $j<\omega_1$, we may fix a stationary $\Gamma_j\s\kappa$
and an ordinal $\varepsilon_j<\kappa$ such that for every $\gamma\in\Gamma_j$,
there is a $\upsilon_j(\gamma)\in \Upsilon'\setminus(\gamma+1)$
such that for every $\xi\in(\varepsilon_j,\gamma)$,
$\gamma\in\im(\tr(\xi,\upsilon_j(\gamma)))$ and $(\rho_0(\xi,\upsilon_j(\gamma)))_g\ge j$.
Set $\varepsilon:=\min(E\setminus\sup\{\lambda^*,\mu,\varepsilon_j\mid j<\omega_1\}+1)$.
For all $j<\omega_1$ and $\delta<\kappa$, let $\gamma(j,\delta):=\min(\Gamma_j\setminus(\delta+1))$.
As $\mu^{\aleph_1}=\mu$, we may fix a stationary $\Delta\s A(\vec C)\cap E^\kappa_{>\omega_1}$ such that for every pair $\bar\delta<\delta$ of ordinals in $\Delta$,
for all $n<\omega$ and $j<\omega_1$, the following three hold:
\begin{itemize}
\item $\beta_{\upsilon_j(\gamma(j,\bar\delta))}^n<\delta$;
\item $d(t_{\upsilon_0(\gamma(0,\bar\delta))}^n,t_{\upsilon_j(\gamma(j,\bar\delta))}^n)=d(t_{\upsilon_0(\gamma(0,\delta))}^n,t_{\upsilon_j(\gamma(j,\delta))}^n)$;
\item $d_n(t_{\upsilon_0(\gamma(0,\bar\delta))}^n,t_{\upsilon_j(\gamma(j,\bar\delta))}^n)=d_n(t_{\upsilon_0(\gamma(0,\delta))}^n,t_{\upsilon_j(\gamma(j,\delta))}^n)$.
\end{itemize}

In particular, $X:=\{\langle \beta_{\upsilon_j(\gamma(j,\delta))}^n\mid (j,n)\in \omega_1\times\omega\rangle\mid \delta\in\Delta\}$ is a rapid subfamily of ${}^{\omega_1\times\omega}\kappa$ of size $\kappa$.
Let $D_X$ be the club given by Lemma~\ref{lemma42}.
Fix $\delta\in D_X\cap\Delta$ above $\varepsilon$. Since $\vec C$ is $\mu$-bounded and $\varepsilon>\mu$, we get that $\delta\in R(\vec C)$.
In addition, as $\delta \in A(\vec C) \cap E^{\kappa}_{>\omega_1}$, $\epsilon:=\sup\{\varepsilon,\lambda(\delta,\beta_{\upsilon_j(\gamma(j,\delta))}^n)\mid (j,n)\in \omega_1\times\omega\}$
is smaller than $\delta$. Thus, an application of $\delta\in D_X\cap R(\vec C)\cap E^\kappa_{>\omega_1}$ (with $m=1$)
yields a $\zeta\in(\epsilon,\delta)$ and another ordinal $\bar\delta\neq\delta$ such that
for every $(j,n)\in \omega_1\times\omega$,
$\Delta(\rho_{0\beta_{\upsilon_j(\gamma(j,\bar\delta))}^n},\rho_{0\beta_{\upsilon_j(\gamma(j,\delta))}^n})=\zeta$.
Consider $\xi:=\sup(\zeta\cap E)$, and note that by Lemma~\ref{lemma2444}, $\varepsilon\le\xi<\zeta$.
\begin{claim} Let $j\in[j^*,\omega_1)$ and $n<\omega$.
Then
$$(\rho_0(\xi,\beta_{\upsilon_j(\gamma(j,\delta))}^n))_g=(\rho_0(\xi,\upsilon_j(\gamma(j,\delta))))_g\ge j.$$
\end{claim}
\begin{why} As $\lambda(\upsilon_j(\gamma(j,\delta)),\beta_{\upsilon_j(\gamma(j,\delta))}^n)\le\lambda^*<\varepsilon\le\xi$,
we get that $\upsilon_j(\gamma(j,\delta))\in\im(\tr(\xi,\beta_{\upsilon_j(\gamma(j,\delta))}^n))$.
Therefore,
$$\tr(\xi,\beta_{\upsilon_j(\gamma(j,\delta))}^n)=\tr(\upsilon_j(\gamma(j,\delta)),\beta_{\upsilon_j(\gamma(j,\delta))}^n)
{}^\smallfrown \tr(\xi,\upsilon_j(\gamma(j,\delta))),$$
from which it follows that
$$\rho_0(\xi,\beta_{\upsilon_j(\gamma(j,\delta))}^n)=\rho_0(\upsilon_j(\gamma(j,\delta)),\beta_{\upsilon_j(\gamma(j,\delta))}^n){}^\smallfrown \rho_0(\xi,\upsilon_j(\gamma(j,\delta))).$$
So,
$$(\rho_0(\xi,\beta_{\upsilon_j(\gamma(j,\delta))}^n))_g=\max\{(\rho_0(\upsilon_j(\gamma(j,\delta)),\beta_{\upsilon_j(\gamma(j,\delta))}^n))_g,(\rho_0(\xi,\upsilon_j(\gamma(j,\delta))))_g\}.$$
However, $(\rho_0(\upsilon_j(\gamma(j,\delta)),\beta_{\upsilon_j(\gamma(j,\delta))}^n))_g\le j^*\le j\le(\rho_0(\xi,\upsilon_j(\gamma(j,\delta))))_g$, so we are done.
\end{why}

It follows that $B:=\{ (\rho_0(\xi,\upsilon_j(\gamma(j,\delta))))_g\mid j^*\le j<\omega_1\}$ is an unbounded subset of $\omega_1$.
By the choice of the colouring $c$, we may now fix some $n<\omega$ such that $c[\{n\}\circledast B]=2$.
For each $i<2$, pick $j_i\in[j^*,\omega_1)$ such that
$c(n,(\rho_0(\xi,\upsilon_{j_i}(\gamma(j_i,\delta))))_g)=i$,
and note that
$$\begin{aligned}
c_n(t^n_{\upsilon_{j_i}(\gamma(j_i,\bar\delta))}\wedge t^n_{\upsilon_{j_i}(\gamma(j_i,\delta))})&=c_n(\rho_{0\beta^n_{\upsilon_{j_i}(\gamma(j_i,\bar\delta))}}\wedge \rho_{0\beta^n_{\upsilon_{j_i}(\gamma(j_i,\delta))}})\\
&=c(n,((\rho_{0\beta^n_{\upsilon_{j_i}(\gamma(j_i,\delta))}}\restriction\zeta)^*)_g)\\
&=c(n, (\rho_{0}(\xi,\beta^n_{\upsilon_{j_i}(\gamma(j_i,\delta))}))_g)\\
&=c(n,(\rho_0(\xi,\upsilon_{j_i}(\gamma(j_i,\delta)))_g))\\
&=i.
\end{aligned}$$

Without loss of generality, $t^n_{\upsilon_0(\gamma(0,\bar\delta))}<_{\lex}t^n_{\upsilon_0(\gamma(0,\delta))}$.
For each $i<2$, from $d(t^n_{\upsilon_0(\gamma(0,\bar\delta))},t^n_{\upsilon_{j_i}(\gamma(j_i,\bar\delta))})=d(t^n_{\upsilon_0(\gamma(0,\delta))},t^n_{\upsilon_{j_i}(\gamma(j_i,\delta))})$,
we infer that $$t^n_{\upsilon_{j_i}(\gamma(j_i,\bar\delta))}<_{\lex}t^n_{\upsilon_{j_i}(\gamma(j_i,\delta))}.$$
So, by the definition of $<_{n}$ we get that $t^n_{\upsilon_{j_0}(\gamma(j_0,\bar\delta))}<_n t^n_{\upsilon_{j_0}(\gamma(j_0,\delta))}$
and $t^n_{\upsilon_{j_1}(\gamma(j_1,\delta))}<_n t^n_{\upsilon_{j_1}(\gamma(j_1,\bar\delta))}$.
On the other hand, $d_n(t^n_{\upsilon_0(\gamma(0,\bar\delta))},t^n_{\upsilon_{j_i}(\gamma(j_i,\bar\delta))})=d_n(t^n_{\upsilon_0(\gamma(0,\delta))},t^n_{\upsilon_{j_i}(\gamma(j_i,\delta)})$
for every $i<2$. Taking $i=0$ this tells us that
$$t^n_{\upsilon_{0}(\gamma(0,\bar\delta))}<_n t^n_{\upsilon_{0}(\gamma(0,\delta))},$$
and taking $i=1$ this tells us that
$$t^n_{\upsilon_{0}(\gamma(0,\delta))}<_n t^n_{\upsilon_{0}(\gamma(0,\bar\delta))}.$$
This is a contradiction.
\end{proof}

\section{Strong colourings of trees}\label{sec8}

The following is a straightforward generalization of \cite[Lemma~2.8]{MR2444284}.

\begin{lemma}\label{lemma61}
There exists a map $\varphi:{}^{<\omega}\omega\rightarrow\mathbb Z$ satisfying the following.
For every function $h:\kappa\rightarrow\omega$,
for every $\delta\in A(\vec C)$, for every finite matrix $\langle \beta^k_i\mid i<m, k<n\rangle$ of ordinals in the interval $(\delta,\kappa)$,
for every sequence $\langle \upsilon_i\mid i<m\rangle$ such that all of the following hold:
\begin{enumerate}
\item for every $i<m$, $\upsilon_i\in \bigcap_{k<n}\im(\tr(\delta,\beta_i^k))$;
\item for all $i<m$ and $k<n$, $h\circ\rho_0(\upsilon_i,\beta_i^k)=h\circ\rho_0(\upsilon_0,\beta_0^k)$;
\item for all $i<m$ and $k<n$, $(\rho_0(\upsilon_i,\beta_i^k))_{h}=(\rho_0(\upsilon_0,\beta_0^0))_{h}$;
\item for all $i<j<m$, $(\rho_0(\upsilon_0,\beta_0^0))_{h}<(\rho_0(\delta,\upsilon_i))_{h}<(\rho_0(\delta,\upsilon_j)_{h}$,
\end{enumerate}
there are $\varepsilon<\delta$ and $\tau<\omega$ such that, for every $\xi\in(\varepsilon,\delta)$,
if $h(\otp(C_\delta\cap\xi))=\tau$, then for all $k,k'<n$ and $i<m$:
\begin{itemize}
\item $\varphi(h\circ \rho_0(\xi,\beta^k_i))=\varphi(h\circ \rho_0(\xi,\beta^k_0))+i$, and
\item $\varphi(h\circ \rho_0(\xi,\beta^k_i))-\varphi(h\circ \rho_0(\xi,\beta^{k'}_i))=\varphi(h\circ \rho_0(\upsilon_0,\beta_0^k))-\varphi(h\circ \rho_0(\upsilon_0,\beta_0^{k'}))$.
\end{itemize}
\end{lemma}
\begin{proof}
Let $\langle d_\tau\mid \tau<\omega\rangle$ be the injective enumeration of some dense subset of the Tychonoff space ${}^\omega\mathbb Z$.
Define a map $\varphi:{}^{<\omega}\omega\rightarrow\mathbb Z$ by letting $\varphi(\emptyset):=0$,
and for every nonempty finite sequence $\sigma:l+1\rightarrow\omega$, define
$$\varphi(\sigma):=\sum_{j=0}^l d_{\sigma(j)}(\sup(\sigma\restriction j)).$$

To see this works,
let $h$, $\delta$, $\langle \beta^k_i\mid i<m, k<n\rangle$ and $\langle \upsilon_i\mid i<m\rangle$ be as above.
Denote $\mathfrak h:=(\rho_0(\upsilon_0,\beta_0^0))_{h}$.
For every $i<m$, denote $\mathfrak h_i:=(\rho_0(\delta,\upsilon_i))_{h}$.
For every $k<n$, denote $\eta^k:=h\circ \rho_0(\upsilon_0,\beta_0^k)$.
With this notation, we have:
\begin{itemize}
\item for all $i<m$ and $k<n$:
\begin{itemize}
\item $h\circ \rho_0(\upsilon_i,\beta_i^k)=\eta^k$, and
\item $\sup(\eta^k)=\mathfrak h<\mathfrak h_i=(\rho_0(\delta,\upsilon_i))_{h}=(\rho_0(\delta,\beta^k_i))_{h}$.
\end{itemize}
\item $i\mapsto\mathfrak h_i$ is strictly increasing over $m$.
\end{itemize}
Also define a function $g:m\rightarrow\mathbb Z$ via:
$$g(i):=\sum_{j=0}^{\rho_2(\delta,\upsilon_i)-1} d_{h(\rho_0(\delta,\upsilon_i)(j))}(\max\{\mathfrak h,h(\rho_0(\delta,\upsilon_i)(\iota))\mid \iota<j\}).$$

Next, as $\delta\in A(\vec C)$, we know that $\varepsilon:=\sup\{ \lambda(\delta,\beta^k_i)\mid i<m, k<n\}$ is smaller than $\delta$.
Consider $H:=\{ \mathfrak h_i\mid i<m\}$, which is an $m$-sized subset of $\omega$.
Define a function $f:H\rightarrow\mathbb Z$ by letting for every $i<m$,
$$f(\mathfrak h_i):=-g(i)+i+g(0).$$

By the choice of our dense set, there are infinitely many $\tau<\omega$ such that $f\s d_{\tau}$,
so let us pick such a $\tau$ above $\max(H)$. To see that $\varepsilon$ and $\tau$ are as sought,
suppose that $\xi\in(\varepsilon,\delta)$ is such that $h(\otp(C_\delta\cap\xi))=\tau$.

Let $i<m$ and $k<n$.
Since $\lambda(\delta,\beta^k_i)\le\varepsilon<\xi$, we have
$$\begin{aligned}\rho_0(\xi,\beta^k_i)=&\ \rho_0(\delta,\beta^k_i){}^\smallfrown\rho_0(\xi,\delta)\\
=&\ \rho_0(\delta,\beta^k_i){}^\smallfrown\langle\otp(C_\delta\cap\xi)\rangle{}^\smallfrown\rho_0(\xi,\min(C_\delta\setminus\xi))\\
=&\ \rho_0(\upsilon_i,\beta^k_i){}^\smallfrown\rho_0(\delta,\upsilon_i){}^\smallfrown\langle\otp(C_\delta\cap\xi)\rangle{}^\smallfrown\rho_0(\xi,\min(C_\delta\setminus\xi)).
\end{aligned}$$
Letting $\sigma:=h\circ\rho_0(\xi,\beta^k_i)$ and $l:=\rho_2(\xi,\beta^k_i)-1$, it is the case that
$$\begin{aligned}\varphi(h\circ\rho_0(\xi,\beta^k_i))=&\sum_{j=0}^l d_{\sigma(j)}(\sup(\sigma\restriction j))\\
=&\sum_{j=0}^{\rho_2(\delta,\beta^k_i)-1} d_{\sigma(j)}(\sup(\sigma\restriction j))+d_{\sigma(\rho_2(\delta,\beta^k_i))}(\sup(\sigma\restriction \rho_2(\delta,\beta^k_i)))\\
&\quad +\sum_{j=\rho_2(\delta,\beta^k_i)+1}^{l} d_{\sigma(j)}(\sup(\sigma\restriction j))\\
=&\ \varphi(h\circ\rho_0(\delta,\beta^k_i))+d_{h(\otp(C_\delta\cap\xi))}((\rho_0(\delta,\beta^k_i))_h)\\
&\quad +\sum_{j=\rho_2(\delta,\beta^k_i)+1}^{l} d_{\sigma(j)}(\sup(\sigma\restriction j))\\
=&\ \varphi(h\circ\rho_0(\delta,\beta^k_i))+d_\tau((\rho_0(\delta,\beta^k_i))_h)+\sum_{j=\rho_2(\delta,\beta^k_i)+1}^{l} d_{\sigma(j)}(\sup(\sigma\restriction j))\\
=&\ \varphi(h\circ\rho_0(\delta,\beta^k_i))+f(\mathfrak h_i)+\sum_{j=\rho_2(\delta,\beta^k_i)+1}^{l} d_{\sigma(j)}(\sup(\sigma\restriction j)).
\end{aligned}$$

We now analyze the first and third components of the above equation. Starting with the first,
as $h\circ \rho_0(\delta,\beta_i^k)=\eta^k{}^\smallfrown(h\circ\rho_0(\delta,\upsilon_i))$
and $\sup(\eta^k)=\mathfrak h$, we get that
$$\begin{aligned}
\varphi(h\circ\rho_0(\delta,\beta^k_i))=&\ \varphi(\eta^k)+\sum_{j=0}^{\rho_2(\delta,\upsilon_i)-1} d_{h(\rho_0(\delta,\upsilon_i)(j))}(\max\{\mathfrak h,h(\rho_0(\delta,\upsilon_i)(\iota))\mid \iota<j\})\\
=&\ \varphi(\eta^k)+g(i).
\end{aligned}$$

To analyze the third component, note that for $j=\rho_2(\delta,\beta^k_i)$, it is the case that
$$\max\{{h}(\rho_0(\xi,\beta^k_i)(\iota))\mid \iota<j\}=(\rho_0(\delta,\beta^k_i))_h\le\max(H)<\tau=h(\rho_0(\xi,\beta^k_i)(j)),$$
and hence
$$\begin{aligned}
\sum_{j=\rho_2(\delta,\beta^k_i)+1}^{l} d_{\sigma(j)}(\sup(\sigma\restriction j))=&\ \sum_{j=\rho_2(\delta,\beta^k_i)+1}^{l} d_{h(\rho_0(\xi,\beta^k_i)(j))}(\max\{{h}(\rho_0(\xi,\beta^k_i)(\iota))\mid \iota<j\})\\
=&\ \sum_{j=\rho_2(\delta,\beta^k_i)+1}^{l} d_{h(\rho_0(\xi,\beta^k_i)(j))}(\max\{{h}(\rho_0(\xi,\beta^k_i)(\iota))\mid \rho_2(\delta,\beta^k_i)\le \iota<j\})\\
=&\ \sum_{j=1}^{\rho_2(\xi,\delta)-1} d_{h(\rho_0(\xi,\delta)(j))}(\max\{{h}(\rho_0(\xi,\delta)(\iota))\mid \iota<j\})\\
=&\ \sum_{j=1}^{\rho_2(\xi,\delta)-1} d_{h(\rho_0(\xi,\delta)(j))}((\rho_0(\xi,\delta)\restriction j)_h).
\end{aligned}$$

Altogether,
$$\varphi(h\circ\rho_0(\xi,\beta^k_i))=\varphi(\eta^k)+g(i)+f(\mathfrak h_i)+\sum_{j=1}^{\rho_2(\xi,\delta)-1} d_{h(\rho_0(\xi,\delta)(j))}((\rho_0(\xi,\delta)\restriction j)_h).$$
Therefore, for all $i<m$ and $k<n$:
$$\begin{aligned}
\varphi(h\circ\rho_0(\xi,\beta^k_i))-\varphi(h\circ\rho_0(\xi,\beta^k_0))=&\ (g(i)+f(\mathfrak h_i))-(g(0)+f(\mathfrak h_0))\\
=&\ g(i)+f(\mathfrak h_i)-g(0)-0=\\
=&\ g(i)-g(i)+i+g(0)-g(0)=i.
\end{aligned}$$

It also follows that for all $k<k'<n$ and $i<m$:
$$\varphi(h\circ\rho_0(\xi,\beta^k_i))-\varphi(h\circ\rho_0(\xi,\beta^{k'}_i))=\varphi(\eta^k)-\varphi(\eta^{k'}),$$
as sought.
\end{proof}

\begin{theorem}\label{thm52}
Suppose that $\vec C$ is a weakly coherent $C$-sequence over $\kappa$ such that $V(\vec C)$ covers a club in $\kappa$,
and $\omega\in \Theta_1^\mu(\vec C\restriction A(\vec C), R(\vec C))$ for a given cardinal $\mu$.
Then $T(\rho_0^{\vec C})\nmeetarrow[\kappa]^{n,1}_\mu$ holds for every positive $n<\omega$.
\end{theorem}
\begin{proof}
Let $\varphi:{}^{< \omega}\omega\rightarrow \mathbb Z$ be given by Lemma~\ref{lemma61}.
Let a positive $n<\omega$ be given,
and then let $f_n:\mathbb Z\rightarrow\omega$ be the corresponding map given by Corollary~\ref{lemma51}.
For every $\Sigma\in{}^n({}^{<\omega}\omega)$, set $$\diam(\Sigma):=\sup\{ |\varphi(\Sigma(k))-\varphi(\Sigma(k'))|\mid k<k'<n\}.$$
Let $\langle \Sigma_l\mid l<\omega\rangle$ be an injective enumeration of ${}^n({}^{<\omega}\omega\setminus\{\emptyset\})$.

Write $\vec C$ as $\langle C_\delta\mid\delta<\kappa\rangle$.
Let $g,h$ witness together that $\omega\in\Theta_1^\mu(\vec C\restriction A(\vec C),R(\vec C))$.
As $|{}^n\mu|=\mu$, we may assume that $g$ is a map from $\kappa$ onto ${}^n\mu$.
Fix a club $E$ in $\kappa$ such that $E\s V(\vec C)$.
To every function $t\in{}^{<\kappa}({}^{<\omega}\kappa)$, we attach a finite sequence $t^*\in{}^{<\omega}\kappa$ via
$$t^*:=\begin{cases}
t(\sup(E \cap \dom(t))),&\text{if }\sup(E\cap\dom(t))\in\dom(t);\\
\emptyset,&\text{otherwise}.
\end{cases}$$

We now define our colouring $c: T(\rho_0^{\vec C})\rightarrow\mu$, as follows.
Given $t\in T(\rho_0^{\vec C})$, consider $l:=f_n(\varphi(h\circ t^*))$.
If there exists a $k<n$ such that $h\circ t^*$ extends $\Sigma_l(k)$,
then let $c(t):=g(t^*(\dom(\Sigma_l(k))-1))(k)$ for the least such $k$. Otherwise, let $c(t):=0$.

To prove that $c$ witnesses $T(\rho_0^{\vec C}) \nmeetarrow [\kappa]^{n,1}_\mu$,
and recalling Clause~(2) of Definition~\ref{meetarrow},
first let $T'\in[T]^\kappa$ be given. Then let $S'\in[{}^nT']^\kappa$
be the rapid family given by Lemma~\ref{lemma35} when invoked with $T'$.
Finally, let $S\s {}^nT(\rho_0^{\vec C})$ be any rapid family of size $\kappa$,\footnote{Recall Remark~\ref{rmkrapid}.}
and let $\nu\in{}^n\mu$;
our goal is to find $x\neq y$ in $S$ such that for every $k<n$:
\begin{itemize}
\item if $S\s S'$, then $c(x(k) \wedge y(k)) = \nu(k)$;
\item if $S\nsubseteq S'$, then $c(x(k) \wedge y(k)) = \nu(0)$;
\item $x(k)<_{\lex}y(k)$ iff $\dom(x(k))<\dom(y(k))$;
\item $\dom(x(k) \wedge y(k))=\dom(x(0) \wedge y(0))$.
\end{itemize}

For simplicity, in case that $S\nsubseteq S'$, we may switch $\nu$ with a constant map, and then the first two bullets coincide.
Returning to Lemma~\ref{lemma35} with the above $S$ and $\nu$,
we fix a cofinal $\Upsilon\s\kappa$,
$\varepsilon_0<\kappa$, and $\langle (\eta^k,\ell^k)\mid k<n\rangle$
such that for every $\upsilon\in\Upsilon$, there is a $\langle \beta_\upsilon^k\mid k<n\rangle$ such that:
\begin{enumerate}
\item for every $k<n$, $\upsilon<\beta_\upsilon^k<\kappa$, $h\circ \rho_0(\upsilon,\beta_\upsilon^k)=\eta^{k}$ and $\sup(\eta^{k})=\sup(\eta^0)$;
\item for every $\delta\in(\varepsilon_0,\upsilon)$, for every $k<n$,
$\tr(\delta,\beta_\upsilon^k)(\ell^k+1)=\upsilon$ and
$g(\rho_0(\delta,\beta_\upsilon^k)(\ell^k))=\nu$;
\item $(\{\eta^{k}\mid k<n\},{\s})$ is an antichain.
If $S\s S'$, then $|\{\eta^{k}\mid k<n\}|=n$;
\item for all $\upsilon\neq\upsilon'$ in $\Upsilon$, there are $x\neq y$ in $S$ such that for every $k<n$:
\begin{itemize}
\item $\dom(x(k))<\dom(y(k))$ iff $\beta_\upsilon^k<\beta_{\upsilon'}^k$,
\item $x(k)\wedge y(k)=\rho_{0\beta_\upsilon^k}\wedge \rho_{0\beta_{\upsilon'}^k}$, and
\item $x(k)<_{\lex}y(k)$ iff $\rho_{0\beta_\upsilon^k}<_{\lex} \rho_{0\beta_{\upsilon'}^k}$.
\end{itemize}
\end{enumerate}

By possibly shrinking $\Upsilon$, we may assume that $X:=\{ \langle \beta_\upsilon^k\mid k<n\rangle\mid \upsilon\in\Upsilon\}$ is rapid.
Fix the unique $l<\omega$ such that $\Sigma_l(k)=\eta^{k}$ for every $k<n$.
Fix an $m<\omega$ such that for every $p:n\rightarrow\mathbb Z$,
with $\max\{|p(k)-p(k')|\mid k<k'<n\}\le\diam(\Sigma_l)$,
there exists an $i<m$ such that $f_n(p(k)+i)=l$ for every $k<n$.

Consider the club $D:=D_X\cap \acc(E\setminus\varepsilon_0)$,
where $D_X$ is given by Lemma~\ref{lemma42} with respect to the family $X$ (where $n$ plays the role of $\nu$).
Let $\Delta$ be the (stationary) set of all $\delta\in A(\vec C)$ such that for every $\tau<\omega$,
$$ \sup\{\gamma \in \nacc(C_\delta) \cap D\cap R(\vec C) \mid h(\otp(C_\delta \cap \gamma)) = \tau\} = \delta.$$

By Lemma~\ref{cor62}, we may now fix a $\delta\in\Delta$ such that
$$\sup\{ (\rho_0(\delta,\upsilon))_{h}\mid \upsilon\in \Upsilon\setminus\delta\}=\omega.$$

Pick an increasing sequence $\langle \upsilon_i\mid i<m\rangle$ of elements of $\Upsilon\setminus\delta$ such that:
\begin{itemize}
\item for all $i<j<m$, $\sup(\eta^0)<(\rho_0(\delta,\upsilon_i))_{h}<(\rho_0(\delta,\upsilon_j))_{h}$.
\end{itemize}
For simplicity, for each $i$, write $\langle \beta^k_i\mid k<n\rangle$ for $\langle \beta^k_{\upsilon_i}\mid k<n\rangle$.

For every $\xi<\delta$, define a function $p_\xi:n\rightarrow\mathbb Z$ via $$p_\xi(k):=\varphi(h\circ \rho_0(\xi,\beta^k_0)).$$

Appealing to the choice of $\varphi$ with $\langle \beta^k_i\mid i<m, k<n\rangle$,
let $\varepsilon< \delta$ and $\tau < \omega$
be such that for every $\xi \in (\varepsilon, \delta)$, if $h(\otp(C_\delta\cap \xi))=\tau$, then for all $k,k'<n$ and $i<m$:
\begin{itemize}
\item $\varphi(h\circ \rho_0(\xi,\beta^k_i))=p_\xi(k)+i$, and
\item $p_\xi(k)-p_\xi(k')=\varphi(h\circ \rho_0(\upsilon_0,\beta_0^k))-\varphi(h\circ \rho_0(\upsilon_0,\beta_0^{k'}))=\varphi(\eta^k)-\varphi(\eta^{k'})$.
\end{itemize}
In particular, in this case, $\max\{|p_\xi(k)-p_\xi(k')|\mid k<k'<n\}=\diam(\Sigma_l)$.

Note that since $\delta\in\Delta\s A(\vec C)$, we may also assume that $\max\{\lambda(\delta, \beta^k_i),\varepsilon_0 \mid i<m, k<n\} < \varepsilon$.

By the choice of $\delta$ we can now find a $\gamma\in\nacc(C_\delta)$ such that:
\begin{itemize}
\item $\gamma \in D_X \cap \acc(E)\cap R(\vec C)$,
\item $h(\otp(C_\delta \cap \gamma)) = \tau$, and
\item $\gamma^{-}= \sup(C_\delta \cap \gamma)$ is bigger than $\varepsilon$.
\end{itemize}

Evidently, $\otp(C_\delta\cap\xi)=\otp(C_\delta\cap\gamma)$ for every $\xi\in(\gamma^-,\gamma]$.
Let $\epsilon:= \min(E \setminus(\gamma^{-}+1))$, so that for all $i<m$ and $k<n$,
$$\varepsilon_0<\varepsilon < \gamma^{-} <\epsilon < \gamma < \delta < \beta^k_i.$$

Since
\begin{itemize}
\item $\gamma \in D_X\cap R(\vec C)$, whereas $D_X$ was given by Lemma~\ref{lemma42},
\item $\langle \beta^k_i \mid i<m, k<n\rangle$ is a matrix of the form required by Lemma~\ref{lemma42}, and
\item $\lambda(\gamma,\beta^k_i)=\gamma^-<\epsilon$ for all $i<m$ and $k<n$,
\end{itemize}
we may obtain a matrix $\langle \alpha^k_i \mid i<m, k<n\rangle$ with entries in $X \cap{}^n\gamma$ and a $\zeta \in (\epsilon, \gamma)$ such that for all $i<m$ and $k<n$:
\begin{itemize}
\item $\Delta(\rho_{0\alpha^k_i},\rho_{0\beta^k_i})=\zeta$, and
\item $\rho_{0\alpha^k_i}<_{\lex}\rho_{0\beta^k_i}$.
\end{itemize}

By Lemma~\ref{lemma2444}, $\zeta\notin E$, so we consider $\xi:=\max(E \cap \zeta)$, noting that $\xi\ge\epsilon$.

Now, for all $i<m$ and $k<n$, letting $t^k_i:= \rho_{0\alpha^k_i} \wedge \rho_{0\beta^k_i}$, we have that
\begin{itemize}
\item $\xi \in \dom(t^k_i)$, and in fact
\item $\xi= \max(E \cap \dom(t^k_i))$, and so
\item $(t_i^k)^*=t^k_i(\xi) = \rho_{0}(\xi,\beta^k_i)$.
\end{itemize}

For all $k<n$ and $i< m$, since
$$\max\{\lambda(\delta, \beta^k_i),\varepsilon_0\} < \varepsilon < \gamma^{-} < \epsilon \le\xi<\zeta< \gamma< \delta<\beta^k_i,$$
we have that:
\begin{itemize}
\item $(t_i^k)^*=\rho_{0}(\delta,\beta^k_i){}^\smallfrown \rho_{0}(\xi,\delta)$;
\item $\tr(\delta,\beta_i^k)(\ell^k+1)=\upsilon_i$;
\item $g(\rho_0(\delta,\beta_i^k)(\ell^k))=\nu$;
\item $h(\otp(C_\delta\cap \xi))=h(\otp(C_\delta\cap \gamma))=\tau$, and hence
\item $\varphi(h\circ (t_i^k)^*)=p_\xi(k)+i$.
\end{itemize}

By the choice of $m$, we may fix an $i< m$ such that $f_n(p_\xi(k)+i) = l$ for all $k<n$.

Let $k<n$.
As $f_n(\varphi(h\circ (t_i^k)^*)) = l$, the definition of $c(t_i^k)$ requires that we consult $\Sigma_l$.
Now, $\Sigma_l(k)=\eta^{k}=h\circ \rho_0(\upsilon_i,\beta^k_i)$ is an initial segment of $h\circ (t^k_i)^*$.
As $\im(\Sigma_l)$ is an antichain, it follows that for the least $\bar k\le k$ such that $\Sigma_l(\bar k)=\eta^k$, we have
$$\begin{aligned}
c(t^k_i)=&\ g((t^k_i)^*(\dom(\Sigma_l(\bar k)-1))(\bar k)\\
=&\ g((t^k_i)^*(\dom(\eta^k)-1))(\bar k)\\
=&\ g((t^k_i)^*(\ell^k))(\bar k)=g(\rho_0(\delta,\beta^k_i)(\ell^k))(\bar k)=\nu(\bar k).
\end{aligned}$$
Additionally, if $\nu$ is not constant, then $S\s S'$ and $|\{\eta^{k}\mid k<n\}|=n$, so that $\bar k=k$.
This means that $c(t^k_i)=\nu(k)$ in both cases.

Finally, recalling Clause~(iv), we may pick $x\neq y$ in $S$ such that for every $k<n$:
\begin{itemize}
\item $\dom(x(k))<\dom(y(k))$ (since $\alpha_i^k<\beta_i^k$),
\item $x(k)\wedge y(k)=\rho_{0\alpha^k_i} \wedge \rho_{0\beta^k_i}=t^k_i$ and $\dom(t^k_i)=\zeta$, and
\item $x(k)<_{\lex} y(k)$ (since $\rho_{0\alpha^k_i}<_{\lex}\rho_{0\beta^k_i}$).
\end{itemize}

Altogether, $c(x(k)\wedge y(k))=\nu(k)$ and $x(k)<_{\lex} y(k)$ iff $\dom(x(k))<\dom(y(k))$.
\end{proof}

In Theorem~\ref{thm52}, we obtained a strong colouring with $\mu$ colours,
assuming $\omega\in \Theta_1^\mu(\ldots)$. In the next theorem, we shall want to get a strong colouring with $\mu^+$ colours from the same hypothesis.
To compare, it is indeed the case that $\mu^+\nrightarrow[\mu^+]^2_\mu$ implies $\mu^+\nrightarrow[\mu^+]^2_{\mu^+}$,
and the standard argument takes a sequence $\langle e_\beta:\mu\rightarrow\beta+1\mid \beta<\mu^+\rangle$ of surjections
and a colouring $c:[\mu^+]^2\rightarrow\mu$ and then stretch it to a colouring $c^+:[\mu^+]^2\rightarrow\mu^+$ via $c^+(\alpha,\beta):=e_\beta(c(\alpha,\beta))$.
Unfortunately, neither this simple method nor the more advanced ones of \cite[\S3]{paper50} are applicable here,
since our colouring gets as an input the meet $(\rho_{0\alpha}\restriction\bar\alpha)\wedge(\rho_{0\beta}\restriction\bar\beta)$ of two restricted fibers rather then the pair $((\rho_{0\alpha}\restriction\bar\alpha),(\rho_{0\beta}\restriction\bar\beta))$,
let alone the pair $(\alpha,\beta)$.
This forces us to somehow determine the stretched colour just from $\rho_{0\beta}\restriction\zeta$ for a (hopefully large enough) $\zeta<\min\{\bar\alpha,\bar\beta\}$,
which may explain the convoluted nature of the colouring given in the upcoming proof.

On the other hand, the next theorem illustrates the power of three cardinal constellations $\kappa_0<\kappa_1<\kappa_2$ as discussed
in the paper's introduction, since
by the results of the next section, as soon as $\mu>\omega$,
many triples of the form $(\kappa_0,\kappa_1,\kappa_2):=(\omega,\mu,\mu^+)$ satisfy the hypotheses as soon as $\square^*_\mu$ holds.

\begin{theorem}\label{thm84}
Suppose that $\vec C$ is a $\square_\mu^*$-sequence such that $V(\vec C)$ covers a club in $\mu^+$,
and $\omega\in \Theta_1^\mu(\vec C\restriction A(\vec C),\mu^+)$.
Then $T(\rho_0^{\vec C})\nmeetarrow[\mu^+]^{n,1}_{\mu^+}$ holds for every positive $n<\omega$.
\end{theorem}
\begin{proof}
Let $\varphi:{}^{< \omega}\omega\rightarrow \mathbb Z$ be given by Lemma~\ref{lemma61}.
Let a positive $n<\omega$ be given,
and then let $f_n:\mathbb Z\rightarrow\omega$ be the corresponding map given by Corollary~\ref{lemma51}.
For every $\Sigma\in{}^n({}^{<\omega}\omega)$, set $$\diam(\Sigma):=\sup\{ |\varphi(\Sigma(k))-\varphi(\Sigma(k'))|\mid k<k'<n\}.$$
Let $\langle \Sigma_l\mid l<\omega\rangle$ be an injective enumeration of ${}^n({}^{<\omega}\omega)$.

Write $\kappa:=\mu^+$ and $\vec C$ as $\langle C_\delta\mid\delta<\kappa\rangle$.
Let $g,h$ witness together that $\omega\in\Theta_1^\mu(\vec C\restriction A(\vec C), \kappa)$.
As $|{}^{<\omega}\mu|=\mu$, we may assume that $g$ is a map from $\kappa$ onto ${}^{<\omega}\mu$.
Fix a club $E$ in $\kappa$ such that $E\s V(\vec C)$.
To every function $t\in{}^{<\kappa}({}^{<\omega}\kappa)$, we attach a finite sequence $t^*\in{}^{<\omega}\kappa$ via
$$t^*:=\begin{cases}
t(\sup(E \cap \dom(t))),&\text{if }\sup(E\cap\dom(t))\in\dom(t);\\
\emptyset,&\text{otherwise}.
\end{cases}$$

Fix a bijection $\pi:\omega\leftrightarrow\omega\times\omega$.
Fix a surjection $\psi:\kappa\rightarrow{}^n\kappa$ such that every element in ${}^n \kappa$ is enumerated cofinally often.
We now define our colouring $c: T(\rho_0^{\vec C})\rightarrow\kappa$, as follows.

Given $t\in T(\rho_0^{\vec C})$, consider $(l,\ell):=\pi(f_n(\varphi(h\circ t^*)))$.
If there exists a $k<n$ such that $h\circ t^*$ extends $\sigma:=\Sigma_l(k)$,
and $\dom(t^*)>\dom(\sigma)+\ell$,
and there exists a unique $\chi\in\dom(t)$ such that
$$t(\chi)=(t^*\restriction\dom(\sigma)){}^\smallfrown g(t^*(\dom(\sigma)+\ell)),$$
then let $c(t):=\psi(\chi)(k)$ for the least such $k$. Otherwise,
let $c(t):=0$.

To prove that $c$ witnesses $T(\rho_0^{\vec C}) \nmeetarrow [\kappa]^{n,1}_\kappa$,
and recalling Clause~(2) of Definition~\ref{meetarrow},
first let $T'\in[T]^\kappa$ be given. Then let $S'\in[{}^nT']^\kappa$
be the rapid family given by Lemma~\ref{lemma35} when invoked with $T'$.
Finally, let $S\s {}^nT(\rho_0^{\vec C})$ be any rapid family of size $\kappa$,
and let $\varkappa\in{}^n\kappa$.
Our goal is to find $x\neq y$ in $S$ such that for every $k<n$:
\begin{itemize}
\item if $S\s S'$, then $c(x(k) \wedge y(k)) = \varkappa(k)$;
\item if $S\nsubseteq S'$, then $c(x(k) \wedge y(k)) = \varkappa(0)$;
\item $x(k)<_{\lex}y(k)$ iff $\dom(x(k))<\dom(y(k))$;
\item $\dom(x(k) \wedge y(k))=\dom(x(0) \wedge y(0))$.
\end{itemize}

For simplicity, in case that $S\nsubseteq S'$, we switch $\varkappa$ with a constant map, and then the first two bullets coincide.
Returning to Lemma~\ref{lemma35} with the above $S$,
we fix a cofinal $\Upsilon_0\s\kappa$,
$\varepsilon_0<\kappa$, and $\langle (\eta^k,\ell^k)\mid k<n\rangle$
such that for every $\upsilon\in\Upsilon_0$, there is a $\langle \beta^{k,\upsilon}\mid k<n\rangle$ such that:
\begin{enumerate}
\item for every $k<n$, $\upsilon<\beta^{k,\upsilon}<\kappa$, $h\circ \rho_0(\upsilon,\beta^{k,\upsilon})=\eta^{k}$ and $\sup(\eta^{k})=\sup(\eta^0)$;
\item for every $\delta\in(\varepsilon_0,\upsilon)$, for every $k<n$,
$\tr(\delta,\beta^{k,\upsilon})(\ell^k)=\upsilon$;
\item $(\{\eta^{k}\mid k<n\},{\s})$ is an antichain.
If $S\s S'$, then $|\{\eta^{k}\mid k<n\}|=n$;
\item for all $\upsilon\neq\upsilon'$ in $\Upsilon_0$, there are $x\neq y$ in $S$ such that for every $k<n$:
\begin{itemize}
\item $\dom(x(k))<\dom(y(k))$ iff $\beta^{k,\upsilon}<\beta^{k, \upsilon'}$,
\item $x(k)\wedge y(k)=\rho_{0\beta^{k,\upsilon}}\wedge \rho_{0\beta^{k, \upsilon'}}$, and
\item $x(k)<_{\lex}y(k)$ iff $\rho_{0\beta^{k,\upsilon}}<_{\lex} \rho_{0\beta^{k, \upsilon'}}$.
\end{itemize}
\end{enumerate}

By possibly increasing $\varepsilon_0$, we may assume that $\mu<\varepsilon_0<\mu^+$.
Fix the unique $l<\omega$ such that $\Sigma_l(k)=\eta^{k}$ for every $k<n$.
Also fix $\chi\in(\varepsilon_0,\kappa)$ such that $\psi(\chi)=\varkappa$.
By possibly shrinking $\Upsilon_0$, we may assume the existence of $\nu\in{}^{<\omega}\mu$
such that for every $\upsilon\in\Upsilon_0$, ($\chi<\upsilon$ and) $\rho_0(\chi,\upsilon)=\nu$.
By another round of shrinking $\Upsilon_0$, we may also assume that $X:=\{ \langle \beta^{k,\upsilon}\mid k<n\rangle\mid \upsilon\in\Upsilon_0\}$ is rapid.

Appealing now to Lemma~\ref{lemma35a}, fix a stationary $\Upsilon_1\s\kappa$,
$\varepsilon_1<\kappa$, and $(\eta^*,\ell^*)$
such that for every $\upsilon_1\in\Upsilon_1$, there is a $\upsilon_0\in\Upsilon_0$ above $\upsilon_1$ such that:
\begin{itemize}
\item for every $\delta\in(\varepsilon_1,\upsilon_1)$, $\tr(\delta,\upsilon_0)(\ell^*+1)=\upsilon_1$ and
$g(\rho_0(\delta,\upsilon_0)(\ell^*))=\nu$;
\item $h\circ \rho_0(\upsilon_1,\upsilon_0)=\eta^*$.
\end{itemize}

By possibly increasing $\varepsilon_1$, we may assume that $\varepsilon_1>\chi$.
It altogether follows that for every $\upsilon\in\Upsilon_1$,
there is a $\langle \beta^k_\upsilon\mid k<n\rangle\in X$ all of whose entries are above $\upsilon$ such that:
\begin{itemize}
\item for every $\delta\in(\varepsilon_0,\upsilon)$, for every $k<n$, $\tr(\delta,\beta_\upsilon^k)(\ell^k)\in\Upsilon_0$;
\item for every $\delta\in(\varepsilon_1,\upsilon)$, for every $k<n$, $\tr(\delta,\beta_\upsilon^k)(\ell^k+\ell^*+1)=\upsilon$ and $g(\rho_0(\delta,\beta_\upsilon^k)(\ell^k+\ell^*))=\nu$;
\item for every $k<n$, $h\circ \rho_0(\upsilon,\beta_\upsilon^k)=\eta^{k}{}^\smallfrown \eta^*$.
\end{itemize}

Denote $d:=\diam(\langle\eta^{k}{}^\smallfrown \eta^*\mid k<n\rangle)$.
Fix an $m<\omega$ such that for every $p:n\rightarrow\mathbb Z$,
with $\max\{|p(k)-p(k')|\mid k<k'<n\}\le d$,
there exists an $i<m$ such that $f_n(p(k)+i)=\pi^{-1}(l,\ell^*)$ for every $k<n$.

Consider the club $D:=D_X\cap \acc(E\setminus\varepsilon_1)$,
where $D_X$ is given by Lemma~\ref{lemma42} with respect to $X$ (where $n$ plays the role of $\nu$).
Let $\Delta$ be the stationary set of all $\delta\in A(\vec C)$ such that for every $\tau<\omega$,
$$ \sup\{\gamma \in \nacc(C_\delta) \cap D \mid h(\otp(C_\delta \cap \gamma)) = \tau\} = \delta.$$

By Lemma~\ref{cor62}, we may now fix a $\delta\in\Delta$ such that
$$\sup\{ (\rho_0(\delta,\upsilon))_{h}\mid \upsilon\in \Upsilon_1\setminus\delta\}=\omega.$$

Pick an increasing sequence $\langle \upsilon_i\mid i<m\rangle$ of elements of $\Upsilon_1\setminus(\delta+1)$ such that:
\begin{itemize}
\item for all $i<j<m$, $\sup(\eta^0{}^\smallfrown\eta^*)<(\rho_0(\delta,\upsilon_i))_{h}<(\rho_0(\delta,\upsilon_j))_{h}$.
\end{itemize}
For simplicity, for each $i$, write $\langle \beta^k_i\mid k<n\rangle$ for $\langle \beta^k_{\upsilon_i}\mid k<n\rangle$.

For every $\xi<\delta$, define a function $p_\xi:n\rightarrow\mathbb Z$ via $$p_\xi(k):=\varphi(h\circ \rho_0(\xi,\beta^k_0)).$$
Appealing to the choice of $\varphi$ with $\langle \beta^k_i\mid i<m, k<n\rangle$,
let $\varepsilon< \delta$ and $\tau < \omega$
be such that for every $\xi \in (\varepsilon, \delta)$, if $h(\otp(C_\delta\cap \xi))=\tau$,
then for all $k,k'<n$ and $i<m$:
\begin{itemize}
\item $\varphi(h\circ \rho_0(\xi,\beta^k_i))=p_\xi(k)+i$, and
\item $p_\xi(k)-p_\xi(k')=\varphi(h\circ \rho_0(\upsilon_0,\beta_0^k))-\varphi(h\circ \rho_0(\upsilon_0,\beta_0^{k'}))=\varphi(\eta^{k}{}^\smallfrown \eta^*)-\varphi(\eta^{k'}{}^\smallfrown \eta^*)$.
\end{itemize}
In particular, in this case, $\max\{|p_\xi(k)-p_\xi(k')|\mid k<k'<n\}=d$.

Note that since $\delta\in\Delta\s A(\vec C)$, we may also assume that $\max\{\lambda(\delta, \beta^k_i),\varepsilon_1 \mid i<m, k<n\} < \varepsilon$.

By the choice of $\delta$ we can find a $\gamma\in\nacc(C_\delta)$ such that
\begin{itemize}
\item $\gamma \in D_X \cap \acc(E)$,
\item $h(\otp(C_\delta \cap \gamma)) = \tau$, and
\item $\gamma^{-}= \sup(C_\delta \cap \gamma)$ is bigger than $\varepsilon$.
\end{itemize}

Let $\epsilon:= \min(E \setminus(\gamma^{-}+1))$, so that for all $i<m$ and $k<n$,
$$\mu<\varepsilon_0<\chi<\varepsilon_1<\varepsilon < \gamma^{-} <\epsilon < \gamma < \delta < \beta^k_i.$$

As $\vec C$ is $\mu$-bounded, it follows that $\gamma\in R(\vec C)$.
So, since
\begin{itemize}
\item $\gamma \in D_X\cap R(\vec C)$, whereas $D_X$ was given by Lemma~\ref{lemma42},
\item $\langle \beta^k_i \mid i<m, k<n\rangle$ is a matrix of the form required by Lemma~\ref{lemma42}, and
\item $\lambda(\gamma,\beta^k_i)=\gamma^-<\epsilon$ for all $i<m$ and $k<n$,
\end{itemize}
we may obtain a matrix $\langle \alpha^k_i \mid i<m, k<n\rangle$ with entries in $X \cap{}^n\gamma$ and a $\zeta \in (\epsilon, \gamma)$ such that for all $i<m$ and $k<n$:
\begin{itemize}
\item $\Delta(\rho_{0\alpha^k_i},\rho_{0\beta^k_i})=\zeta$, and
\item $\rho_{0\alpha^k_i}<_{\lex}\rho_{0\beta^k_i}$.
\end{itemize}

By Lemma~\ref{lemma2444}, $\zeta\notin E$, so we let $\xi:=\max(E \cap \zeta)$, noting that $\xi\ge\epsilon$.
Now, for all $i<m$ and $k<n$, letting $t^k_i:= \rho_{0\alpha^k_i} \wedge \rho_{0\beta^k_i}$, we infer that
$$(t_i^k)^*=t^k_i(\xi) = \rho_{0}(\xi,\beta^k_i).$$

For all $k<n$ and $i< m$, since
$$\varepsilon_0<\chi<\max\{\lambda(\delta, \beta^k_i),\varepsilon_1\} < \varepsilon < \gamma^{-} < \epsilon \le\xi<\zeta< \gamma< \delta<\beta^k_i,$$
we have that:
\begin{itemize}
\item $(t_i^k)^*=\rho_{0}(\delta,\beta^k_i){}^\smallfrown \rho_{0}(\xi,\delta)$;
\item $\tr(\delta,\beta^k_i)(\ell^k)=\tr(\chi,\beta^k_i)(\ell^k)$ is in $\Upsilon_0$;
\item $g(\rho_0(\delta,\beta_i^k)(\ell^k+\ell^*))=\nu$;
\item $h(\otp(C_\delta\cap \xi))=h(\otp(C_\delta\cap \gamma))=\tau$, and hence
\item $\varphi(h\circ (t_i^k)^*)=p_\xi(k)+i$.
\end{itemize}

By the choice of $m$, we may fix an $i< m$ such that $f_n(p_\xi(k)+i)=\pi^{-1}(l,\ell^*)$ for all $k<n$.
Let $k<n$.
Following the definition of $c(t_i^k)$,
as $\pi(f_n(\varphi(h\circ (t_i^k)^*)))=(l,\ell^*)$,
and as $\Sigma_l(k)=\eta^{k}=h\circ (\rho_0(\xi,\beta^k_i)\restriction\ell^k)$ is an initial segment of $h\circ (t^k_i)^*$,
we look at $(t^k_i)^*\restriction\ell^k$,
which by the second bullet above, coincides with $\rho_0(\chi,\beta^k_i)\restriction \ell^k$.
As $\tr(\chi,\beta^k_i)(\ell^k)$ is in $\Upsilon_0$,
it follows that
$$\rho_0(\chi,\beta^k_i)=(\rho_0(\chi,\beta^k_i)\restriction\ell^k){}^\smallfrown\nu=((t^k_i)^*\restriction\ell^k){}^\smallfrown\nu.$$
In addition, $g((t^k_i)^*(\ell^k+\ell^*))=g(\rho_0(\delta,\beta^k_i)(\ell^k+\ell^*))=\nu$,
so that
$$t_i^k(\chi)=\rho_0(\chi,\beta^k_i)=((t^k_i)^*\restriction\dom(\eta^k)){}^\smallfrown g((t^k_i)^*(\dom(\eta^k)+\ell^*)),$$
so since $t_i^k$ is injective and since $\im(\Sigma_l)$ is an antichain, $c(t^k_i)=\psi(\chi)(\bar k)=\varkappa(\bar k)$ for the least $\bar k\le k$ such that $\Sigma_l(\bar k)=\eta^k$.
Additionally, if $\varkappa$ is not constant, then $S\s S'$ and $|\{\eta^{k}\mid k<n\}|=n$, so that $\bar k=k$.
This means that $c(t^k_i)=\varkappa(k)$ in both cases.

Finally, recalling Clause~(iv), we may pick $x\neq y$ in $S$ such that for every $k<n$:
\begin{itemize}
\item $\dom(x(k))<\dom(y(k))$ (since $\alpha_i^k<\beta_i^k$),
\item $x(k)\wedge y(k)=\rho_{0\alpha^k_i} \wedge \rho_{0\beta^k_i}=t^k_i$ and $\dom(t^k_i)=\zeta$, and
\item $x(k)<_{\lex} y(k)$ (since $\rho_{0\alpha^k_i}<_{\lex}\rho_{0\beta^k_i}$).
\end{itemize}

Altogether, $c(x(k)\wedge y(k))=\varkappa(k)$ and $x(k)<_{\lex} y(k)$ iff $\dom(x(k))<\dom(y(k))$.
\end{proof}

\section{Improve your $C$-sequence!}\label{canonicalsect}

Works of Jensen \cite{jen72} and Todor\v{c}evi\'c \cite{TodActa} provided characterizations of
the existence of special Aronszajn trees in terms of the existence of $C$-sequences.
This was then generalised by Krueger \cite{MR3078820} who proved that
the existence of a special $\kappa$-Aronszajn tree is equivalent to the existence of weakly coherent $C$-sequences $\vec C=\langle C_\delta\mid\delta<\kappa\rangle$ of various sorts,
e.g., one in which $\otp(C_\delta)<\delta$ for club many $\delta<\kappa$.
Here, it is shown that the former is equivalent to both seemingly weaker as well as stronger postulates.
Our motivation comes from the results of the previous section,
and so the main results here are pump-up theorems for
Definitions \ref{def410} and \ref{def411}, yielding
witnessing $C$-sequences with improved characteristics in the sense of Definitions \ref{def44} and \ref{def429}.

\begin{thm}\label{thm91} All of the following are equivalent:
\begin{enumerate}[(1)]
\item There is a $\square_{<}^*(\kappa)$-sequence $\vec C$ such that $R'(\vec C)\cap V'(\vec C)$ covers a club;
\item $\square_{<}^*(\kappa)$ holds;
\item There is a special $\kappa$-Aronszajn tree.
\end{enumerate}
\end{thm}
\begin{proof} $(1)\implies(2)$: This is trivial.

$(2)\implies(3)$: By Corollary~\ref{cor74}.

$(3)\implies(1)$:
By \cite[Theorem~2.24]{paper58}, if there is a special $\kappa$-Aronszajn tree, then we may fix one $\mathbf T=(T,<_T)$ such that $V(\mathbf T)$, \emph{the set of vanishing levels of $\mathbf T$}, covers a club. Here (see \cite[Definition~1.2]{paper58}) $V(\mathbf T)$ is the set of $\alpha \in \acc(\kappa)$ such that for every node $x$ of $\mathbf T$ of height less than $\alpha$, there is an $\alpha$-branch $B$ of $\mathbf T$ that contains $x$ but the branch $B$ is not bounded in $\mathbf T$.

Without loss of generality, our tree is ordinal-based, i.e., $\mathbf T=(\kappa,<_T)$. Let $C:=\{\delta<\kappa \mid T\restriction\delta=\delta\}$.
Now, fix a sparse enough club $D\s\acc(C)\cap V(\mathbf T)$ for which there exists a good colouring (see Lemma~\ref{fact13}) $f:T\restriction D\rightarrow\kappa$, that is:
\begin{itemize}
\item $f(t)<\h_{\mathbf T}(t)$ for every $t \in T\restriction D$;
\item for every pair $s<_T t$ in $T\restriction D$, $f(s)\neq f(t)$.
\end{itemize}

For every $\delta\in D$, as $\delta\in V(\mathbf T)$, we may let $B_\delta$ be a vanishing $\delta$-branch of $T$.
It follows that $B_\delta$ is a cofinal subset of $(\delta,{\in})$. In addition, for every pair $\gamma<\delta$ of ordinals in $D$, $\sup(B_\gamma\cap B_\delta)<\gamma$.
Note that for all $\gamma, \xi \in B_\delta$, if $\h_{\mathbf T}(\gamma)$ and $\h_{\mathbf T}(\xi)$ belong to $C$,
then $\gamma <_T \xi$ iff $\gamma< \xi$. Also note that $f$ is injective over $B_\delta\cap (T\restriction D)$.

Set $\Delta:=\acc(D)\cap E^\kappa_{>\omega}$.
For every $\delta\in \Delta$,
for every $\gamma\in B_\delta\cap (T\restriction D)$, we have $f(\gamma)<\h_{\mathbf T}(\gamma)$,
so using Fodor's lemma we may fix an $\varepsilon_\delta<\delta$ for which
$$B_\delta':=\{\gamma\in B_\delta\cap (T\restriction D)\mid f(\gamma)<\varepsilon_\delta\}\setminus\varepsilon_\delta.$$ is cofinal in $(\delta,{\in})$.
Now we run the thinning-out procedure from Krueger's \cite[Theorem~2.5]{MR3078820}.

Let $\delta\in \Delta$. We recursively define an increasing continuous sequence of ordinals below $\delta$, $\langle \gamma^\delta_i \mid i< \sigma_\delta\rangle$,
where $\sigma_\delta$ denotes the length of our recursion.

We start by letting $\gamma^\delta_0:=\min(B_\delta')$.
Next, suppose that $j =i+1$ for some $i$ such that $\gamma^\delta_{i}$ has already been defined.
Then $\gamma^\delta_j$ is taken to be the unique $\gamma \in B_\delta'$ such that
\begin{itemize}
\item $\gamma^\delta_{i} <_T \gamma$ (equivalently, $\gamma^\delta_{i} < \gamma$),
\item $f(\gamma)$ is minimal for satisfying the above condition.
\end{itemize}

Finally, suppose that $j$ is a nonzero limit ordinal and we have constructed $\langle \gamma^\delta_i \mid i< j\rangle$ so far.
Then either $\sup\{\gamma^\delta_i \mid i< j\} = \delta$ and then we take $\sigma_\delta:=j$ and terminate the recursion, or else we simply take $\gamma^\delta_j:=\sup\{\gamma^\delta_i \mid i< j\} $.
So the recursive procedure has been described, with the resulting sequence $\langle \gamma^\delta_i \mid i< \sigma_\delta\rangle$ clearly a club in $\delta$.

We now define a $C$-sequence $\vec C:=\langle C_\delta\mid\delta<\kappa\rangle$, as follows:

$\br$ $C_0:=\emptyset$.

$\br$ For every $\gamma<\kappa$, $C_{\gamma+1}:=\{\gamma\}$.

$\br$ For every $\delta\in \acc(D)\cap E^\kappa_{>\omega}$, let $C_\delta:=\{\gamma^\delta_j\mid j<\sigma_\delta\}$.

$\br$ For every $\delta\in \acc(D)\cap E^\kappa_\omega$, let $C_\delta$ be a cofinal subset of $(B_\delta,\in)$ of order-type $\omega$.

$\br$ For every $\delta\in\acc(\kappa)\setminus\acc(D)$, let $C_\delta$ be a cofinal subset of $(\delta,\in)$ of order-type $\cf(\delta)$ such that $\min(C_\delta)=\sup(D\cap\delta)+1$.

\begin{claim} $\acc(D)\s R'(\vec C)$.
\end{claim}
\begin{why} Let $\gamma\in\acc(D)$, and let $\delta< \kappa$. We need to verify that $\otp(C_\delta\cap\gamma)<\gamma$.

$\br$ If $\delta<\gamma$, then $\otp(C_\delta\cap\gamma)\le\delta<\gamma$.

$\br$ If $\gamma=\delta$ and $\cf(\delta)=\omega$, then $\otp(C_\delta)=\omega<\delta$.

$\br$ If $\gamma=\delta$ and $\cf(\delta)>\omega$, then $C_\delta=\{\gamma^\delta_j\mid j<\sigma_\delta\}$, so that $\otp(C_\delta)=\sigma_\delta$. Then our recursive construction ensures that $i\mapsto f(\gamma_{i+1}^\delta)$ is a strictly increasing map from $\sigma_\delta$ to $\varepsilon_\delta$,
and hence $\sigma_\delta\le\varepsilon_\delta<\delta$.

$\br$ If $\delta>\gamma$ and $\delta\notin\acc(D)$, then $\min(C_\delta)=\sup(D\cap\delta)+1$,
so that $C_\delta\cap\gamma$ is empty.

$\br$ If $\delta>\gamma$ and $\delta\in\acc(D)\cap E^\kappa_\omega$, then $\otp(C_\delta)=\omega$, so that $C_\delta\cap\gamma$ is finite.

$\br$ If $\delta>\gamma$ and $\delta\in\acc(D)\cap E^\kappa_{>\omega}$, then $\otp(C_\delta)=\sigma_\delta\le\varepsilon_\delta\le\min(C_\delta)$,
so if $\otp(C_\delta\cap\gamma)=\gamma$, then $\min(C_\delta)<\gamma<\otp(C_\delta)$, which yields a contradiction.
\end{why}
\begin{claim} $\vec C$ is weakly coherent.
\end{claim}
\begin{why} Suppose not, and fix the least $\epsilon<\kappa$ such that $|\{ C_\delta\cap\epsilon\mid \delta<\kappa\}|=\kappa$.
A moment's reflection makes it clear that $|\{ C_\delta\cap\epsilon\mid \sup(C_\delta\cap\epsilon)=\epsilon, \delta\in \Delta\}|=\kappa$.
Recalling the definition of $\vec C$, it thus follows that we may fix a set $\Delta_0\in[\Delta]^\kappa$ such that:
\begin{itemize}
\item $\epsilon\in\acc(C_\delta)$ for every $\delta\in\Delta_0$, and
\item $\delta\mapsto C_\delta\cap\epsilon$ is injective over $\Delta_0$.
\end{itemize}
Of course, it also follows that $\delta\mapsto \nacc(C_\delta)\cap\epsilon$ is injective over $\Delta_0$.
As $\nacc(C_\delta)\s B_\delta'\s B_\delta\setminus\varepsilon_\delta$ for every $\delta\in \Delta$,
we may fix some $\Delta_1\in[\Delta_0]^\kappa$ such that $\delta\mapsto\varepsilon_\delta$ is constant over $\Delta_1$.
For every $\delta\in\Delta_1$, there is a unique node $\beta_\delta\in B_\delta$ with $\h_{\mathbf T}(\beta_\delta)=\epsilon$. As $\mathbf T$ is a $\kappa$-tree,
we may fix a $\Delta_2\in[\Delta_1]^\kappa$ such that $\delta\mapsto \beta_\delta$ is constant over $\Delta_2$.
Consequently, $\delta\mapsto B_\delta'\cap\epsilon$
is also constant over $\Delta_2$.

Fix $\delta\neq \delta'$ in $\Delta_2$, so that $C_\delta\cap\epsilon\neq C_{\delta'}\cap\epsilon$,
while $B_\delta'\cap\epsilon=B_{\delta'}'\cap\epsilon$.
Clearly, $\gamma^\delta_0=\min(B_\delta',\in)=\min(B_{\delta'}',\in)=\gamma^{\delta'}_0$.
It thus follows there exists an $i$ such that $\gamma^\delta_{i+1}\neq\gamma^{\delta'}_{i+1}$,
so let $i$ be the least such. This ensures that $\gamma^\delta_{i}=\gamma^{\delta'}_{i}$, say $\gamma^*$.
As $\delta,\delta'\in\Delta_0$, we know that $\gamma^\delta_{i+1}$ and $\gamma^{\delta'}_{i+1}$ are both ordinals smaller than $\epsilon$.

So $\gamma^\delta_{i+1}$ is the unique element of $B_\delta'\cap(\gamma^*,\epsilon)$ that has the least $f$-value,
and $\gamma^{\delta'}_{i+1}$ is the unique element of $B_{\delta'}'\cap(\gamma^*,\epsilon)$ that has the least $f$-value.
As $B_\delta'\cap(\gamma^*,\epsilon)=B_{\delta'}'\cap(\gamma^*,\epsilon)$ we arrive at a contradiction.
\end{why}
\begin{claim} $\acc(D)\s V'(\vec C)$.
\end{claim}
\begin{why} Let $\delta\in\acc(D)$. Then $\nacc(C_\delta)\s B_\delta$. Towards a contradiction, suppose that $\alpha>\delta$ and $\sup(\nacc(C_\delta)\cap\nacc(C_\alpha))=\delta$.
In particular, $\alpha\in\acc(\kappa)$.

$\br$ If $\alpha\in\acc(D)$, then $\nacc(C_\alpha)\s B_\alpha$,
but $\sup(B_\delta\cap B_\alpha)<\alpha$, because $B_\delta$ is a vanishing $\delta$-branch,
and $B_\alpha$ is an $\alpha$-branch.

$\br$ If $\alpha\notin\acc(D)$, then $\min(C_\alpha)=\sup(D\cap\alpha)+1\ge\delta+1$.
So $C_\delta\cap C_\alpha=\emptyset$.
\end{why}

Altogether, $\vec C$ is as sought.
\end{proof}

By Fact~\ref{criticalcof},
in the preceding proof, in case that $\kappa=\mu^+$,
we could have taken the map $f$ to have range $\mu$, thus establishing the following.

\begin{thm}\label{thm92} For an infinite cardinal $\mu$, all of the following are equivalent:
\begin{enumerate}[(1)]
\item There is a $\square_\mu^*$-sequence $\vec C$ such that $V'(\vec C)$ covers a club in $\mu^+$;
\item $\square_{<}^*(\mu^+)$ holds;
\item There is a special $\mu^+$-Aronszajn tree.\qed
\end{enumerate}
\end{thm}

In view of Theorem~\ref{thm84},
at this point we are left with adding a layer of club-guessing. We start with the case that $\mu$ is regular.

\begin{thm}\label{thm93} Suppose that $\mu$ is a regular uncountable cardinal and $\square^*_\mu$ holds.
Then there exists a $\square^*_\mu$-sequence $\vec C$ such that:
\begin{itemize}
\item $V'(\vec C)$ covers a club;
\item $1\in\Theta_1^\mu(\vec C \restriction E^{\mu^+}_\mu,E^{\mu^+}_\mu)$;
\item if $\mu$ is non-ineffable, then $\omega\in\Theta_1^\mu(\vec C \restriction E^{\mu^+}_\mu,E^{\mu^+}_\mu)$.
\end{itemize}
\end{thm}
\begin{proof} By Theorem~\ref{thm421}, we may fix a $\square_\mu^*$-sequence $\vec{C^0}=\langle C^0_\delta\mid \delta<\mu^+\rangle$
and a set $S\s E^{\mu^+}_\mu$ such that:
\begin{itemize}
\item for every $\delta\in S$, $C^0_\delta\s\acc(\delta)$.
\item $1\in\Theta_1^\mu(\vec C^0\restriction S,E^{\mu^+}_\mu)$, and if $\mu$ is non-ineffable, then moreover $\omega\in\Theta_1^\mu(\vec C^0\restriction S,E^{\mu^+}_\mu)$.
\end{itemize}

In addition, by Theorem~\ref{thm92}, let us fix a $\square_\mu^*$-sequence $\vec{C^1}=\langle C_\delta^1\mid \delta<\mu^+\rangle$ such that $V'(\vec{C^1})$ covers a club in $\mu^+$.
Define $\vec C=\langle C_\delta\mid\delta<\mu^+\rangle$, as follows:
\begin{itemize}
\item[$\br$] Let $C_0:=\emptyset$;
\item[$\br$] For every $\gamma<\mu^+$, let $C_{\gamma+1}:=\{\gamma\}$;
\item[$\br$] For every $\delta\in E^{\mu^+}_\mu$ such that $C_\delta^0\s\acc(\delta)$, let $C_\delta:=C^0_\delta$;
\item[$\br$] For every other $\delta<\mu^+$, let $C_\delta:=\{\beta+1 \mid \beta \in \nacc(C^1_\delta)\}\cup \acc(C_\delta^1)$, so that $\nacc(C_\delta)\s\nacc(\delta)$.
\end{itemize}
A moment's reflection makes it clear that $\vec C$ possesses all the required features.
\end{proof}

In the preceding, we constructed a $\mu$-bounded $C$-sequence $\vec C$ over $\mu^+$ for $\mu$ regular,
which outright ensures that $A(\vec C)$ covers the stationary set $E^{\mu^+}_\mu$.
As $A(\vec C)$ cannot be a reflecting stationary set,
and since, by \cite[Theorem~12.1]{cfm}, $\square^*_\mu$ may hold at a singular cardinal $\mu$,
and yet every stationary subset of $\mu^+$ reflects,
it is impossible to derive the setup of Theorem~\ref{thm84} just from $\square^*_\mu$.
In the next version of this paper, we will provide conditions for getting $C$-sequences
satisfying the hypotheses of Theorems \ref{thm52} and \ref{thm84}
for successors of singulars and for inaccessibles. This will prove Theorem~\ref{thmc},
as well as showing that if $\square_{<}(\kappa)$ holds,
then there is a $C$-sequence $\vec C$ over $\kappa$ such that $T(\rho_0^{\vec C})$ is a special $\kappa$-Aronszajn tree,
and $T(\rho_0^{\vec C})\nmeetarrow[\kappa]^{n,1}_\kappa$ holds for every positive integer $n$.
In particular, this is the case for every regular uncountable $\kappa$ that is not a Mahlo cardinal in $\mathsf{L}$.

\section{Connecting the dots}\label{sec10}

We are now ready to derive Theorem~\ref{thma}.

\begin{cor}\label{cor100} Suppose that $\kappa=\mu^+$ for a regular uncountable cardinal $\mu$ that is non-ineffable.
Then all of the following are equivalent:
\begin{enumerate}[(1)]
\item There exists a special $\kappa$-Aronszajn tree;
\item There exists a $\mu$-bounded $C$-sequence $\vec C$ over $\kappa$ such that $T(\rho_0^{\vec C})$ is a special $\kappa$-Aronszajn tree,
and $T(\rho_0^{\vec C})\nmeetarrow[\kappa]^{n,1}_{\kappa}$ holds for every positive integer $n$;
\item Any basis for the class of special $\kappa$-Aronszajn lines has size $2^\kappa$.
\end{enumerate}
\end{cor}
\begin{proof} $(1)\implies(2)$: By Theorems \ref{thm92} and \ref{thm93},
there exists a $\square^*_\mu$-sequence $\vec C$ such that $V'(\vec C)$ covers a club and $\omega\in\Theta^\mu_1(\vec C \restriction E^{\mu^+}_\mu,E^{\mu^+}_\mu)$
(and note that since $\vec C$ is $\mu$-bounded, $A(\vec C)\supseteq E^{\mu^+}_\mu$), so that Clause~(2) follows from Theorem~\ref{thm84}. Indeed, .

$(2)\implies(3)$: By Lemma~\ref{getent} and Proposition~\ref{entsier}.

$(3)\implies(1)$: This is trivial.
\end{proof}

We now arrive at Theorem~\ref{corb}.
\begin{cor} For every regular uncountable cardinal $\mu$, if there is a $\mu^+$-Aronszajn line,
then there is one without a $\mu^+$-Countryman subline.
\end{cor}
\begin{proof} Let $\mu$ be regular and uncountable.
If there is no $\mu^+$-Countryman line, then we are done.
Otherwise, by Remark~\ref{rmk214}, there exists a $\mu^+$-special Aronszajn line.

$\br$ If $\mu$ is non-ineffable, then Corollary~\ref{cor100} yields a streamlined $\mu^+$-Aronszajn tree $T$
satisfying $T\nmeetarrow[\mu^+]^2_{\mu^+}$.
By Fact~\ref{flippedorder} and Lemma~\ref{omitcm}, then,
$T$ admits an ordering that makes it into a $\mu^+$-Aronszajn line with no $\mu^+$-Countryman line.

$\br$ If $\mu$ is ineffable, then in particular, $\mu^{\aleph_1}=\mu$.
In addition, by Theorems \ref{thm92} and \ref{thm93},
there exists a $\square^*_\mu$-sequence $\vec C$ such that $V'(\vec C)$ covers a club and $1\in\Theta^\mu_1(\vec C \restriction E^{\mu^+}_\mu,E^{\mu^+}_\mu)$.
Since $\vec C$ is $\mu$-bounded, $A(\vec C)\supseteq E^{\mu^+}_\mu$.
So, the conclusion follows from Theorem~\ref{easyway}.
\end{proof}

An inspection of all the ingredients that go into the proof of the results of this section
reveals that the part involving walks on ordinals is always carried out along a $\vec C$ for which $A(\vec C)$ is stationary,
in which case any assumption of the form $\theta\in \Theta_1^\mu(\vec C\restriction S,T)$
may be relaxed to $\theta\in \Theta_1^\mu(\vec C\restriction S,T,\half)$ in the sense of \cite[Definition~6.7]{paper46}.
We decided to avoid this generality, since Theorem~\ref{thm93} as well as \cite[Lemma~6.9]{paper46} suggest that this subtle difference is mostly of interest at the level of $\aleph_1$ which was not our focus.
For the record, we confirm that $\mho$ is indeed sufficient for Clause~(2) of Corollary~\ref{cor100} to hold true with $\mu=\aleph_0$.
We also mention that by using proxy principles \cite{paper65}, it is possible to resurrect many of the results of this paper by walking along a $\vec C$ for which $A(\vec C)\cup R(\vec C)$ is nonstationary.
This opens the door to constructing Aronszajn trees satisfying strong negative partition relations in the presence of stationary reflection and at the level of higher Mahlo cardinals.
The details will appear elsewhere.

\section*{Acknowledgments}
The first author is supported by the Israel Science Foundation (grant agreements 2066/18 and 665/20).
The second author is partially supported by the Israel Science Foundation (grant agreement 203/22)
and by the European Research Council (grant agreement ERC-2018-StG 802756).

We are grateful to Chen Meiri for stimulating discussions on the content of Subsection~\ref{sec32}.

\newcommand{\etalchar}[1]{$^{#1}$}

\begin{theindex}
\label{indexpage}

\item $C$-sequences
\subitem $\chi(\vec C)$, \hyperpage{20}, \hyperpage{43}
\subitem $\mu$-bounded, \hyperpage{21}
\subitem $\square^*_\mu$, \hyperpage{21}, \hyperpage{24},
\hyperpage{59}
\subitem $\square_\mu$, \hyperpage{21}
\subitem $\square_{<}^*(\kappa)$, \hyperpage{21}, \hyperpage{42},
\hyperpage{57}
\subitem Avoiding levels, $A(\vec C)$, \hyperpage{21},
\hyperpage{25}, \hyperpage{47}
\subitem Coherent, \hyperpage{21}
\subitem Lower regressive, \hyperpage{21}
\subitem Lower regressive levels, $R(\vec C)$, \hyperpage{21},
\hyperpage{41}
\subitem Regressive, \hyperpage{21}
\subitem Regressive levels, $R'(\vec C)$, \hyperpage{21},
\hyperpage{57}
\subitem Strongly avoiding levels, $A'(\vec C)$, \hyperpage{21},
\hyperpage{37}
\subitem Strongly vanishing levels, $V'(\vec C)$, \hyperpage{21},
\hyperpage{36}, \hyperpage{57}
\subitem Trivial, \hyperpage{20}
\subitem Vanishing levels, $V(\vec C)$, \hyperpage{21},
\hyperpage{31}
\subitem Weakly coherent, \hyperpage{21, 22}

\indexspace

\item Club guessing
\subitem $\Theta_1(\vec C,T)$, \hyperpage{24}
\subitem $\Theta_1^\mu(\vec C,T)$, \hyperpage{24}, \hyperpage{50},
\hyperpage{53}
\subitem $\cg_\xi(S, T)$, \hyperpage{23}

\indexspace

\item Linear orders
\subitem Aronszajn line, \hyperpage{14, 15}
\subitem Countryman line, \hyperpage{14}
\subitem Entangled, \hyperpage{13}
\subitem Far, \hyperpage{13}
\subitem Lexicographic ordering, \hyperpage{12}, \hyperpage{43}
\subitem Monotonically far, \hyperpage{13}
\subitem Special Aronszajn line, \hyperpage{14}

\indexspace

\item Strong colourings
\subitem $\U(\kappa,\lambda,\theta,\chi)$, \hyperpage{16},
\hyperpage{24}, \hyperpage{37}
\subitem $\mathbf T\nmeetarrow[\kappa]^{n,m}_\theta$,
\hyperpage{12}, \hyperpage{50}, \hyperpage{53}
\subitem $\mathbf T\nmeetarrow[\kappa]^{n}_\theta$, \hyperpage{12}
\subitem $\onto(\mathcal A,J,\theta)$, \hyperpage{16}
\subitem $\ubd(\mathcal A,J,\theta)$, \hyperpage{16}
\subitem Erd\H{o}s-Hajnal's Problem, \hyperpage{17}
\subitem Fiber maps, \hyperpage{16}
\subitem Projection map $\varphi:{}^{<\omega}\omega\rightarrow\mathbb Z$,
\hyperpage{47}
\subitem Projection map $f_n:\mathbb Z\rightarrow\omega$,
\hyperpage{19}

\indexspace

\item Trees
\subitem $S$-coherent, \hyperpage{11}, \hyperpage{34},
\hyperpage{37}, \hyperpage{40}
\subitem $T(c)$, \hyperpage{16}
\subitem $\kappa$-tree, \hyperpage{9}
\subitem $\varsigma$-splitting, \hyperpage{9}
\subitem Aronszajn, \hyperpage{9}
\subitem Ascending path, \hyperpage{9}, \hyperpage{43}
\subitem Good colouring, \hyperpage{10}, \hyperpage{41}
\subitem Hausdorff, \hyperpage{11}
\subitem Partition tree, \hyperpage{14}
\subitem Souslin, \hyperpage{9}, \hyperpage{13}
\subitem Special, \hyperpage{9}, \hyperpage{33}, \hyperpage{40},
\hyperpage{57}
\subitem Streamlined, \hyperpage{11}

\indexspace

\item Walks on ordinals
\subitem $\Tr$, \hyperpage{22}
\subitem $\lambda$, \hyperpage{22}, \hyperpage{26}
\subitem $\lambda_2$, \hyperpage{23}, \hyperpage{32, 33}
\subitem $\last{\beta}{\gamma}$, \hyperpage{23}, \hyperpage{41}
\subitem $\rho_0$, \hyperpage{22}, \hyperpage{31}
\subitem $\rho_1$, \hyperpage{22}, \hyperpage{37}
\subitem $\rho_2$, \hyperpage{22}, \hyperpage{40}
\subitem $\tr$, \hyperpage{22}

\end{theindex}

\begin{thebibliography}{CFM{\etalchar{+}}18}

\bibitem[ARS85]{ARSh:153}
Uri Abraham, Matatyahu Rubin, and Saharon Shelah.
\newblock On the consistency of some partition theorems for continuous colorings, and the structure of $\aleph_ 1$-dense real order types.
\newblock {\em Annals of Pure and Applied Logic}, 29:123--206, 1985.

\bibitem[AS85]{Sh:114}
Uri Abraham and Saharon Shelah.
\newblock Isomorphism types of aronszajn trees.
\newblock {\em Israel Journal of Mathematics}, 50:75--113, 1985.

\bibitem[Asp14]{MR3307877}
David Asper\'{o}.
\newblock The consistency of a club-guessing failure at the successor of a regular cardinal.
\newblock In {\em Infinity, computability, and metamathematics}, volume~23 of {\em Tributes}, pages 5--27. Coll. Publ., London, 2014.

\bibitem[Bar23]{MR4568265}
Keegan~Dasilva Barbosa.
\newblock Decomposing {A}ronszajn lines.
\newblock {\em J. Math. Log.}, 23(1):Paper No. 2250017, 20, 2023.

\bibitem[Bau73]{MR317934}
James~E. Baumgartner.
\newblock All {$\aleph \sb{1}$}-dense sets of reals can be isomorphic.
\newblock {\em Fund. Math.}, 79(2):101--106, 1973.

\bibitem[Bau76]{MR416925}
James~E. Baumgartner.
\newblock A new class of order types.
\newblock {\em Ann. Math. Logic}, 9(3):187--222, 1976.

\bibitem[Bau82]{MR661296}
James~E. Baumgartner.
\newblock Order types of real numbers and other uncountable orderings.
\newblock In {\em Ordered sets ({B}anff, {A}lta., 1981)}, volume~83 of {\em NATO Adv. Study Inst. Ser. C: Math. Phys. Sci.}, pages 239--277. Reidel, Dordrecht-Boston, Mass., 1982.

\bibitem[BG23]{MR4586835}
Omer {Ben-Neria} and Thomas Gilton.
\newblock Club stationary reflection and the special {A}ronszajn tree property.
\newblock {\em Canad. J. Math.}, 75(3):854--911, 2023.

\bibitem[BMV23]{benneria2023aronszajntreesmaximality}
Omer {Ben-Neria}, Menachem Magidor, and Jouko Väänänen.
\newblock Aronszajn trees and maximality -- Part 2, 2023.

\bibitem[BR17]{paper20}
Ari~Meir Brodsky and Assaf Rinot.
\newblock Reduced powers of {S}ouslin trees.
\newblock {\em Forum Math. Sigma}, 5(e2):1-- 82, 2017.

\bibitem[BR19a]{paper29}
Ari~Meir Brodsky and Assaf Rinot.
\newblock Distributive {A}ronszajn trees.
\newblock {\em Fund. Math.}, 245(3):217--291, 2019.

\bibitem[BR19b]{paper32}
Ari~Meir Brodsky and Assaf Rinot.
\newblock A remark on {S}chimmerling's question.
\newblock {\em Order}, 36(3):525--561, 2019.

\bibitem[BRY25]{paper65}
Ari~Meir Brodsky, Assaf Rinot, and Shira Yadai.
\newblock Proxy principles in combinatorial set theory.
\newblock {\em Zb. Rad. (Beogr.)}, to appear.
\newblock \verb"http://assafrinot.com/paper/65"

\bibitem[CFM01]{cfm}
James Cummings, Matthew Foreman, and Menachem Magidor.
\newblock Squares, scales and stationary reflection.
\newblock {\em J. Math. Log.}, 1(1):35--98, 2001.

\bibitem[CFM{\etalchar{+}}18]{MR3796288}
James Cummings, Sy-David Friedman, Menachem Magidor, Assaf Rinot, and Dima Sinapova.
\newblock The eightfold way.
\newblock {\em J. Symb. Log.}, 83(1):349--371, 2018.

\bibitem[CIR25]{paper63}
Rodrigo Carvalho, Tanmay Inamdar, and Assaf Rinot.
\newblock Diamond on ladder systems and countably metacompact topological spaces.
\newblock {\em J. Symbolic Logic}, to appear.
\newblock \verb"https://doi.org/10.1017/jsl.2024.40"

\bibitem[CL17]{MR3620068}
Sean Cox and Philipp L\"ucke.
\newblock Characterizing large cardinals in terms of layered posets.
\newblock {\em Ann. Pure Appl. Logic}, 168(5):1112--1131, 2017.

\bibitem[DM40]{MR1919}
Ben Dushnik and E.~W. Miller.
\newblock Concerning similarity transformations of linearly ordered sets.
\newblock {\em Bull. Amer. Math. Soc.}, 46:322--326, 1940.

\bibitem[DS21a]{MR4199849}
Alan Dow and R.~M. Stephenson, Jr.
\newblock Productivity of cellular-{L}indel\"{o}f spaces.
\newblock {\em Topology Appl.}, 290:Paper No. 107606, 15, 2021.

\bibitem[DS21b]{Sh:1186}
Mirna D{\v{z}}amonja and Saharon Shelah.
\newblock {On wide Aronszajn trees in the presence of MA}.
\newblock {\em The Journal of Symbolic Logic}, 86(1):210--223, 2021.

\bibitem[EH71]{MR280381}
P.~Erd\H{o}s and A.~Hajnal.
\newblock Unsolved problems in set theory.
\newblock In {\em Axiomatic {S}et {T}heory ({P}roc. {S}ympos. {P}ure {M}ath.,
{V}ol. {XIII}, {P}art {I}, {U}niv. {C}alifornia, {L}os {A}ngeles, {C}alif.,
1967)}, Proc. Sympos. Pure Math., XIII, Part I, pages 17--48. Amer. Math. Soc., Providence, RI, 1971.

\bibitem[EHR65]{MR202613}
P.~Erd\H{o}s, A.~Hajnal, and R.~Rado.
\newblock Partition relations for cardinal numbers.
\newblock {\em Acta Math. Acad. Sci. Hungar.}, 16:93--196, 1965.

\bibitem[Gin53]{MR53994}
Seymour Ginsburg.
\newblock Some remarks on order types and decompositions of sets.
\newblock {\em Trans. Amer. Math. Soc.}, 74:514--535, 1953.

\bibitem[Gin54]{MR62808}
Seymour Ginsburg.
\newblock Further results on order types and decompositions of sets.
\newblock {\em Trans. Amer. Math. Soc.}, 77:122--150, 1954.

\bibitem[Gin55]{MR70683}
Seymour Ginsburg.
\newblock Order types and similarity transformations.
\newblock {\em Trans. Amer. Math. Soc.}, 79:341--361, 1955.

\bibitem[Hau36]{hausdorff1936zwei}
F~Hausdorff.
\newblock {\"U}ber zwei s{\"a}tze von g. Fichtenholz und l. Kantorovitch.
\newblock {\em Studia Mathematica}, 6(1):18--19, 1936.

\bibitem[HMR05]{MR2194042}
Michael Hru{\v{s}}{\'a}k and Carlos Mart{\'{\i}}nez-Ranero.
\newblock Some remarks on non-special coherent {A}ronszajn trees.
\newblock {\em Acta Univ. Carolin. Math. Phys.}, 46(2):33--40, 2005.

\bibitem[Ina23]{MR4529976}
Tanmay~C. Inamdar.
\newblock On strong chains of sets and functions.
\newblock {\em Mathematika}, 69(1):286--301, 2023.

\bibitem[IR23]{paper47}
Tanmay Inamdar and Assaf Rinot.
\newblock Was {U}lam right? {I}: {B}asic theory and subnormal ideals.
\newblock {\em Topology Appl.}, 323(C):Paper No. 108287, 53pp, 2023.

\bibitem[IR24]{paper53}
Tanmay Inamdar and Assaf Rinot.
\newblock Was {U}lam right? {I}{I}: {S}mall width and general ideals.
\newblock {\em Algebra Universalis}, 85(2):Paper No. 14, 47pp, 2024.

\bibitem[IR25]{paper46}
Tanmay Inamdar and Assaf Rinot.
\newblock A club guessing toolbox {I}.
\newblock {\em Bull. Symb. Log.}, to appear.
\newblock \verb"https://doi.org/10.1017/bsl.2024.27"

\bibitem[Jen72]{jen72}
Ronald B.~Jensen.
\newblock The fine structure of the constructible hierarchy.
\newblock {\em Ann. Math. Logic}, 4:229--308; erratum, ibid. 4 (1972), 443, 1972.

\bibitem[Kom21]{MR4347570}
Peter Komj\'ath.
\newblock Notes on some Erdos-Hajnal problems.
\newblock {\em J. Symb. Log.}, 86(3):1116--1123, 2021.

\bibitem[KRS24]{paper50}
Menachem Kojman, Assaf Rinot, and Juris Stepr\={a}ns.
\newblock Ramsey theory over partitions {II}: {N}egative {R}amsey relations and pump-up theorems.
\newblock {\em Israel J. Math.}, 2024.
\newblock Accepted May 2022.

\bibitem[Kru13]{MR3078820}
John Krueger.
\newblock Weak square sequences and special {A}ronszajn trees.
\newblock {\em Fund. Math.}, 221(3):267--284, 2013.

\bibitem[Kru18]{MR3826541}
John Krueger.
\newblock Club isomorphisms on higher {A}ronszajn trees.
\newblock {\em Ann. Pure Appl. Logic}, 169(10):1044--1081, 2018.

\bibitem[Kru19]{MR3913157}
John Krueger.
\newblock The approachability ideal without a maximal set.
\newblock {\em Ann. Pure Appl. Logic}, 170(3):297--382, 2019.

\bibitem[LHR18]{paper34}
Chris Lambie-Hanson and Assaf Rinot.
\newblock Knaster and friends {I}: Closed colorings and precalibers.
\newblock {\em Algebra Universalis}, 79(4):Art. 90, 39, 2018.

\bibitem[LHR21]{paper35}
Chris Lambie-Hanson and Assaf Rinot.
\newblock Knaster and friends {II}: {T}he {C}-sequence number.
\newblock {\em J. Math. Log.}, 21(1):2150002, 54, 2021.

\bibitem[LHR23]{paper36}
Chris Lambie-Hanson and Assaf Rinot.
\newblock Knaster and friends {III}: {S}ubadditive colorings.
\newblock {\em J. Symbolic Logic}, 88(3):1230--1280, 2023.

\bibitem[LS81]{MR603771}
Richard Laver and Saharon Shelah.
\newblock The {$\aleph _{2}$}-{S}ouslin hypothesis.
\newblock {\em Trans. Amer. Math. Soc.}, 264(2):411--417, 1981.

\bibitem[L{\"{u}}c17]{Lucke}
Philipp L{\"{u}}cke.
\newblock Ascending paths and forcings that specialize higher {A}ronszajn trees.
\newblock {\em Fund. Math.}, 239(1):51--84, 2017.

\bibitem[Mat16]{highandlowcontri}
American~{Institute of} Mathematics.
\newblock {H}igh and {L}ow {F}orcing, 2016.
\newblock \verb"https://aimath.org/WWN/highlowforcing/highlowforcing.pdf"

\bibitem[Mit73]{MR0313057}
William Mitchell.
\newblock Aronszajn trees and the independence of the transfer property.
\newblock {\em Ann. Math. Logic}, 5:21--46, 1972/73.

\bibitem[Moo06a]{MR2199228}
Justin~Tatch Moore.
\newblock A five element basis for the uncountable linear orders.
\newblock {\em Ann. of Math. (2)}, 163(2):669--688, 2006.

\bibitem[Moo06b]{lspace}
Justin~Tatch Moore.
\newblock A solution to the {$L$} space problem.
\newblock {\em J. Amer. Math. Soc.}, 19(3):717--736, 2006.

\bibitem[Moo08]{MR2444284}
Justin~Tatch Moore.
\newblock Aronszajn lines and the club filter.
\newblock {\em J. Symbolic Logic}, 73(3):1029--1035, 2008.

\bibitem[Moo09]{MR2480566}
Justin~Tatch Moore.
\newblock A universal {A}ronszajn line.
\newblock {\em Math. Res. Lett.}, 16(1):121--131, 2009.

\bibitem[Moo10]{MR2827783}
Justin~Tatch Moore.
\newblock The proper forcing axiom.
\newblock In {\em Proceedings of the {I}nternational {C}ongress of
{M}athematicians. {V}olume {II}}, pages 3--29. Hindustan Book Agency, New Delhi, 2010.

\bibitem[MR11a]{MR2995913}
Carlos Mart{\'{\i}}nez-Ranero.
\newblock {\em Contributions towards a fine structure theory of {A}ronszajn orderings}.
\newblock ProQuest LLC, Ann Arbor, MI, 2011.
\newblock Thesis (Ph.D.)--University of Toronto (Canada).

\bibitem[MR11b]{MR2822417}
Carlos Mart{\'{\i}}nez-Ranero.
\newblock Well-quasi-ordering {A}ronszajn lines.
\newblock {\em Fund. Math.}, 213(3):197--211, 2011.

\bibitem[MRT11]{MR2802589}
Carlos Mart{\'{\i}}nez-Ranero and Stevo Todor\v{c}evi\'{c}.
\newblock Gap structure of coherent {A}ronszajn trees.
\newblock {\em Math. Res. Lett.}, 18(3):565--578, 2011.

\bibitem[MV21]{MR4290492}
Rahman Mohammadpour and Boban Veli\u{c}kovi\'{c}.
\newblock Guessing models and the approachability ideal.
\newblock {\em J. Math. Log.}, 21(2):Paper No. 2150003, 35, 2021.

\bibitem[Nee14]{MR3201836}
Itay Neeman.
\newblock Forcing with sequences of models of two types.
\newblock {\em Notre Dame J. Form. Log.}, 55(2):265--298, 2014.

\bibitem[Nee17]{MR3696074}
Itay Neeman.
\newblock Two applications of finite side conditions at {$\omega_2$}.
\newblock {\em Arch. Math. Logic}, 56(7-8):983--1036, 2017.

\bibitem[Pen13]{peng}
Yinhe Peng.
\newblock {\em Characterization and Transformation of Countryman Lines and $\mathbb R$-Embeddable Coherent Trees in {Z}{F}{C}}.
\newblock 2013.
\newblock Thesis (Ph.D.)--National University of Singapore.

\bibitem[PR23]{paper60}
Márk Poór and Assaf Rinot.
\newblock A {S}helah group in {Z}{F}{C}.
\newblock Submitted July 2023.
\newblock \verb"http://assafrinot.com/paper/60"

\bibitem[Rin14a]{paper18}
Assaf Rinot.
\newblock Chain conditions of products, and weakly compact cardinals.
\newblock {\em Bull. Symb. Log.}, 20(3):293--314, 2014.

\bibitem[Rin14b]{paper15}
Assaf Rinot.
\newblock Complicated colorings.
\newblock {\em Math. Res. Lett.}, 21(6):1367--1388, 2014.

\bibitem[RYY23]{paper58}
Assaf Rinot, Shira Yadai, and Zhixing You.
\newblock The vanishing levels of a tree.
\newblock Submitted September 2023.
\newblock \verb"http://assafrinot.com/paper/58"

\bibitem[RYY24]{paper69}
Assaf Rinot, Zhixing You, and Jiachen Yuan.
\newblock Ketonen's question and other cardinal sins.
\newblock Submitted September 2024.
\newblock \verb"http://assafrinot.com/paper/69"

\bibitem[RZ21]{paper44}
Assaf Rinot and Jing Zhang.
\newblock Transformations of the transfinite plane.
\newblock {\em Forum Math. Sigma}, 9(e16):1--25, 2021.

\bibitem[She75]{MR427073}
Saharon Shelah.
\newblock Colouring without triangles and partition relation.
\newblock {\em Israel J. Math.}, 20:1--12, 1975.

\bibitem[She94]{Sh:400}
Saharon Shelah.
\newblock {Cardinal Arithmetic}.
\newblock In {\em {Cardinal Arithmetic}}, volume~29 of {\em Oxford Logic Guides}, chapter~IX. Oxford University Press, 1994.

\bibitem[She97]{Sh:462}
Saharon Shelah.
\newblock $\sigma $-entangled linear orders and narrowness of products of boolean algebras.
\newblock {\em Fundamenta Mathematicae}, 153:199--275, 1997.

\bibitem[She10]{Sh:922}
Saharon Shelah.
\newblock Diamonds.
\newblock {\em Proc. Amer. Math. Soc.}, 138(6):2151--2161, 2010.

\bibitem[Sie32]{sierpinski1932probleme}
Wac{\l}aw Sierpi{\'n}ski.
\newblock Sur un probleme concernant les types de dimensions.
\newblock {\em Fundamenta Mathematicae}, 19(1):65--71, 1932.

\bibitem[Spe49]{MR0039779}
Ernst Specker.
\newblock Sur un probl\`eme de {S}ikorski.
\newblock {\em Colloquium Math.}, 2:9--12, 1949.

\bibitem[SZ97]{MR1442306}
Saharon Shelah and Jind\v{r}ich Zapletal.
\newblock Embeddings of {C}ohen algebras.
\newblock {\em Adv. Math.}, 126(2):93--118, 1997.

\bibitem[Tod81]{MR631563}
Stevo Todor\v{c}evi\'{c}.
\newblock Trees, subtrees and order types.
\newblock {\em Ann. Math. Logic}, 20(3):233--268, 1981.

\bibitem[Tod83]{MR716846}
Stevo Todor\v{c}evi\'{c}.
\newblock Forcing positive partition relations.
\newblock {\em Trans. Amer. Math. Soc.}, 280(2):703--720, 1983.

\bibitem[Tod84]{MR776625}
Stevo Todor\v{c}evi\'{c}.
\newblock Trees and linearly ordered sets.
\newblock In {\em Handbook of set-theoretic topology}, pages 235--293. North-Holland, Amsterdam, 1984.

\bibitem[Tod85a]{MR792822}
Stevo Todor\v{c}evi\'{c}.
\newblock Directed sets and cofinal types.
\newblock {\em Trans. Amer. Math. Soc.}, 290(2):711--723, 1985.

\bibitem[Tod85b]{MR793235}
Stevo Todor\v{c}evi\'{c}.
\newblock Partition relations for partially ordered sets.
\newblock {\em Acta Math.}, 155(1-2):1--25, 1985.

\bibitem[Tod87]{TodActa}
Stevo Todor{\v{c}}evi{\'c}.
\newblock Partitioning pairs of countable ordinals.
\newblock {\em Acta Math.}, 159(3-4):261--294, 1987.

\bibitem[Tod96]{MR1407459}
Stevo Todor\v{c}evi\'{c}.
\newblock A classification of transitive relations on {$\omega_1$}.
\newblock {\em Proc. London Math. Soc. (3)}, 73(3):501--533, 1996.

\bibitem[Tod98]{MR1648055}
Stevo Todor\v{c}evi\'{c}.
\newblock Basis problems in combinatorial set theory.
\newblock In {\em Proceedings of the {I}nternational {C}ongress of
{M}athematicians, {V}ol. {II} ({B}erlin, 1998)}, number Extra Vol. II, pages 43--52, 1998.

\bibitem[Tod05]{MR2213652}
Stevo Todor\v{c}evi\'{c}.
\newblock Representing trees as relatively compact subsets of the first {B}aire class.
\newblock {\em Bull. Cl. Sci. Math. Nat. Sci. Math.}, (30):29--45, 2005.

\bibitem[Tod07a]{MR2329763}
Stevo Todor\v{c}evi\'{c}.
\newblock Lipschitz maps on trees.
\newblock {\em J. Inst. Math. Jussieu}, 6(3):527--556, 2007.

\bibitem[Tod07b]{MR2355670}
Stevo Todor\v{c}evi\'{c}.
\newblock {\em Walks on ordinals and their characteristics}, volume 263 of {\em Progress in Mathematics}.
\newblock Birkh\"{a}user Verlag, Basel, 2007.

\bibitem[Tod18]{transitivecolourings}
Stevo Todor\v{c}evi\'{c}.
\newblock Transitive colorings.
\newblock {\em unpublished note}, February 2018.

\bibitem[TZ99]{MR1703188}
Stevo Todor\v{c}evi\'{c} and Jind\v{r}ich Zapletal.
\newblock On the {A}laoglu-{B}irkhoff equivalence of posets.
\newblock {\em Illinois J. Math.}, 43(2):281--290, 1999.

\bibitem[Wei10]{Chris10}
Christoph Weiß.
\newblock Subtle and ineffable tree properties.
\newblock {\em PhD thesis, Ludwig-Maximilians-Universität München}, 2010.

\bibitem[Zap98]{MR1617909}
Jind\v{r}ich Zapletal.
\newblock A dichotomy for forcing notions.
\newblock {\em Math. Res. Lett.}, 5(1-2):213--226, 1998.

\end{thebibliography}
\end{document}